\documentclass{amsart}
\usepackage{latexsym,amsmath, amssymb, amscd,  graphicx, psfrag}
\usepackage{color, xcolor}
\usepackage{fullpage}
\usepackage{tikz-cd}
\newcommand{\si}{ {\sigma} }
\newcommand{\Si}{\Sigma}
\newcommand{\bSi}{\partial\Si}
\newcommand{\ep}{\epsilon}
\newcommand{\Ga}{ {\Gamma} }

\newcommand{\bA}{ \mathbb{A} }
\newcommand{\bC}{ \mathbb{C} }
\newcommand{\bD}{ \mathbb{D} }
\newcommand{\bF}{ \mathbb{F} }
\newcommand{\bK}{\mathbb{K}}
\newcommand{\bL}{\mathbb{L}}
\newcommand{\bP}{ \mathbb{P} }
\newcommand{\bQ}{ \mathbb{Q} }
\newcommand{\bR}{ \mathbb{R} }
\newcommand{\bS}{\mathbb{S}}
\newcommand{\bT}{ \mathbb{T} }
\newcommand{\bZ}{ \mathbb{Z} }

\newcommand{\bfr}{\mathbf{r}}
\newcommand{\bx}{\mathbf{x}}
\newcommand{\by}{\mathbf{y}}
\newcommand{\bmu}{\boldsymbol{\mu}}
\newcommand{\btau}{\boldsymbol{\tau}}
\newcommand{\one}{\mathbf{1}}

\newcommand{\cA}{\mathcal{A}}
\newcommand{\cB}{\mathcal{B}}
\newcommand{\cC}{\mathcal{C}}

\newcommand{\cF}{\mathcal{F}}

\newcommand{\cI}{\mathcal{I}}
\newcommand{\cL}{\mathcal{L}}
\newcommand{\cM}{\mathcal{M}}
\newcommand{\cO}{\mathcal{O}}
\newcommand{\cP}{\mathcal{P}}
\newcommand{\cQ}{\mathcal{Q}}
\newcommand{\cS}{\mathcal{S}}
\newcommand{\cT}{\mathcal{T}}
\newcommand{\cX}{\mathcal{X}}
\newcommand{\cV}{\mathcal{V}}
\newcommand{\cW}{\mathcal{W}}
\newcommand{\cD}{\mathcal{D}}
\newcommand{\cG}{\mathcal{G}}
\newcommand{\cU}{\mathcal{U}}
\newcommand{\cY}{\mathcal{Y}}
\newcommand{\cZ}{\mathcal{Z}}

\newcommand{\fl}{\mathfrak{l}}
\newcommand{\fm}{\mathfrak{m}}
\newcommand{\fM}{\mathfrak{M}}

\newcommand{\fo}{\mathfrak{o}}
\newcommand{\fp}{\mathfrak{p}}
\newcommand{\fr}{\mathfrak{r}}
\newcommand{\fs}{\mathfrak{s}}

\newcommand{\fC}{\mathfrak{C}}

\newcommand{\fX}{{\mathfrak{X}}}

\newcommand{\Aut}{\mathrm{Aut}}
\newcommand{\Def}{\mathrm{Def}}
\newcommand{\Hom}{\mathrm{Hom}}
\newcommand{\Ext}{\mathrm{Ext}}
\newcommand{\Ker}{\mathrm{Ker}}
\newcommand{\Coker}{\mathrm{Coker}}

\newcommand{\Spec}{\mathrm{Spec}}

\newcommand{\vir}{ {\mathrm{vir}}  }
\newcommand{\ev}{\mathrm{ev}}
\newcommand{\Res}{\mathrm{Res}}
\newcommand{\orb}{ {\mathrm{CR}} }
\newcommand{\ch}{ {\mathrm{ch}} }
\newcommand{\age}{\mathrm{age}}

\newcommand{\tor}{ {\mathrm{tor}} }
\newcommand{\Tor}{ {\mathrm{Tor}} }
\newcommand{\Pic}{ {\mathrm{Pic}} }
\newcommand{\Nef}{ {\mathrm{Nef}} }
\newcommand{\NE}{ {\mathrm{NE}} }
\newcommand{\BoxS}{\mathrm{Box}(\mathbf \Si)}
\newcommand{\Boxs}{\mathrm{Box}(\si)}

\newcommand{\rig}{ {\mathrm{rig}} }
\newcommand{\can}{ {\mathrm{can}} }

\newcommand{\wb}{\mathsf{b}}

\newcommand{\wu}{\mathsf{u}}
\newcommand{\ww}{\mathsf{w}}
\newcommand{\wv}{\mathsf{v}}

\newcommand{\w}{\mathsf{w}}

\newcommand{\hB}{\hat{B}}

\newcommand{\hu}{\hat{u}}

\newcommand{\hcG}{\hat{\cG}}
\newcommand{\hcM}{\hat{\cM}}

\newcommand{\hGa}{ {\hat{\Ga}} }

\newcommand{\tC}{ {\widetilde{C}} }

\newcommand{\tcD}{ {\widetilde{\mathcal{D}}}}

\newcommand{\tL}{ {\widetilde{L}} }
\newcommand{\tM}{ {\widetilde{M}} }
\newcommand{\tN}{ {\widetilde{N}} }

\newcommand{\tV}{ {\widetilde{V}} }
\newcommand{\tbT}{ {\widetilde{\bT}} }
\newcommand{\tb}{ {\tilde{b}} }

\newcommand{\vt}{\check{t}}

\newcommand{\tbeta}{{\widetilde{\beta}}}

\newcommand{\tmu}{ {\widetilde{\mu}} }
\newcommand{\tnu}{\widetilde{\nu}}

\newcommand{\tGamma}{ {\tilde{\Gamma}} }
\newcommand{\tGammar}{ {\tilde{\Gamma}}_{\mathrm{red}} }
\newcommand{\tNef}{\widetilde{\Nef}}
\newcommand{\tNE}{\widetilde{\NE}}
\newcommand{\tSi}{\widetilde{\Sigma}}
\newcommand{\tTheta}{\widetilde{\Theta}}
\newcommand{\trho}{\widetilde{\rho}}

\newcommand{\tit}{\tilde{t}}

\newcommand{\inv}{\mathrm{inv}}
\newcommand{\eff}{\mathrm{eff}}
\newcommand{\pt}{\mathrm{point}}

\newcommand{\Mbar}{\overline{\cM}}

\newcommand{\vmu}{ {\vec{\mu}} }

\newcommand{\lra}{\longrightarrow}

\newtheorem{dummy}{dummy}[section]
\newtheorem{lemma}[dummy]{Lemma}
\newtheorem{theorem}[dummy]{Theorem}

\newtheorem{proposition}[dummy]{Proposition}

\theoremstyle{definition}
\newtheorem{definition}[dummy]{Definition}
\newtheorem{remark}[dummy]{Remark}
\newtheorem*{remark*}{Remark}
\newtheorem{example}[dummy]{Example}
\newtheorem{assumption}[dummy]{Assumption}

\begin{document}

\title[Open-Closed GW Invariants of Calabi-Yau Smooth Toric DM Stacks]{Open-Closed Gromov-Witten Invariants of
3-dimensional Calabi-Yau Smooth Toric DM Stacks}
\author{Bohan Fang}
\address{Bohan Fang, Beijing International Center for Mathematical Research, Peking University, 5 Yiheyuan Road, Beijing 100871, China}
\email{bohanfang@gmail.com}

\author{Chiu-Chu Melissa Liu}
\address{Chiu-Chu Melissa Liu, Department of Mathematics, Columbia University,
2990 Broadway, New York, NY 10027}
\email{ccliu@math.columbia.edu}

\author{Hsian-Hua Tseng}
\address{Hsian-Hua Tseng, Department of Mathematics, Ohio State University,
100 Math Tower, 231 West 18th Avenue, Columbus, OH 43210}
\email{hhtseng@math.ohio-state.edu}

\date{January 30, 2022}

\begin{abstract}
We study open-closed orbifold Gromov-Witten invariants of  3-dimensional Calabi-Yau smooth toric Deligne-Mumford stacks 
(with possibly non-trivial generic stabilizers $K$ and semi-projective coarse moduli spaces) relative to Lagrangian branes of Aganagic-Vafa type. An Aganagic-Vafa brane in this paper is
a possibly ineffective $C^\infty$ orbifold which admits a presentation $[(S^1\times \bR^2)/G_\tau]$, where $G_\tau$ is a finite abelian group
containing $K$ and $G_\tau/K \cong \bmu_\fm$ is cyclic of some order $\fm\in \bZ_{>0}$. 
\begin{enumerate}
\item We present foundational materials of enumerative geometry of stable holomorphic maps from bordered orbifold Riemann surfaces
to a 3-dimensional Calabi-Yau smooth toric DM stack $\cX$ with boundaries mapped into a Aganagic-Vafa brane $\cL$. All genus open-closed Gromov-Witten
invariants of $\cX$ relative to $\cL$ are defined by torus localization and depend on the choice of a framing $f\in \bZ$ of  $\cL$.

\item We provide another definition of all genus open-closed Gromov-Witten invariants in (1)
based on algebraic relative orbifold Gromov-Witten theory.  This generalizes the definition in \cite{LLLZ} for smooth toric Calabi-Yau 3-folds 
and specifies an orientation on moduli of maps in (1) compatible with the canonical orientation 
determined by the complex structure on moduli of relative stable maps.

\item When $\cX$ is a toric Calabi-Yau 3-orbifold (i.e. when the generic stabilizer $K$ is trivial), so that $G_\tau=\bmu_\fm$,
we define generating functions $F_{g,h}^{\cX,(\cL,f)}$ of 
open-closed Gromov-Witten invariants or arbitrary genus $g$ and number $h$ of boundary circles; it takes
values in $H^*_{\orb}(\cB \bmu_\fm;\bC)^{\otimes h}$ where $H^*_{\orb}(\cB \bmu_\fm;\bC)\cong \bC^\fm$ is the Chen-Ruan
orbifold cohomology of the classifying space $\cB \bmu_\fm$ of $\bmu_\fm$.

\item We prove an open mirror theorem which relates the generating function  $F_{0,1}^{\cX,(\cL,f)}$ of orbifold disk invariants
to Abel-Jacobi maps of the mirror curve of $\cX$. This generalizes a conjecture by Aganagic-Vafa \cite{AV00} and Aganagic-Klemm-Vafa \cite{AKV02}
(proved in full generality by the first and the second authors in \cite{FL}) on the disk potential of a smooth semi-projective toric Calabi-Yau 3-fold. 
\end{enumerate}
\end{abstract}

\maketitle


\section{Introduction}

Open Gromov-Witten (GW) invariants of toric Calabi-Yau 3-folds
have been studied extensively by both mathematicians and physicists.
They correspond to ``A-model topological open
string amplitudes'' in the physics literature and can be interpreted
as intersection numbers of certain moduli spaces of holomorphic maps
from bordered Riemann surfaces to the 3-fold with boundaries
in a Lagrangian submanifold. The physics prediction of these open GW 
invariants comes from string dualities: {\em mirror symmetry} relates
the A-model topological string theory of a Calabi-Yau 3-fold $X$ to the B-model topological string theory of the mirror Calabi-Yau 3-fold $\check{X}$; the
{\em large-$N$ duality} relates the A-model topological string theory on Calabi-Yau 3-folds
(of complex dimension three) to the Chern-Simons theory on
3-manifolds (of real dimension three). 

\subsection{Open GW invariants of smooth toric Calabi-Yau  3-folds} 
Aganagic-Vafa \cite{AV00} introduce a class of Lagrangian submanifolds in smooth semi-projective toric Calabi-Yau 3-folds, which are diffeomorphic to $S^1\times \bR^2$.
By mirror symmetry, Aganagic-Vafa and Aganagic-Klemm-Vafa \cite{AV00, AKV02} relate genus-zero open GW invariants (disk invariants) of a smooth
toric Calabi-Yau 3-fold $X$ relative to such a Lagrangian submanifold $L$ to the classical Abel-Jacobi map of the mirror Calabi-Yau 3-fold $\check{X}$,
which can be further related to the Abel-Jacobi map to the mirror curve of $X$. This conjecture is proved in full generality in \cite{FL}.

By the large-$N$ duality, Aganagic-Klemm-Mari\~{n}o-Vafa
propose the topological vertex \cite{AKMV}, an algorithm of computing all genera generating functions $F_{\beta',\mu_1, ..., \mu_h}$ of open
GW invariants of $(X,L)$ obtained by fixing a topological type of the map (determined by the degree $\beta'\in H_2(X,L;\bZ)$  and
winding numbers $\mu_1,\ldots, \mu_h\in H_1(L;\bZ)=\bZ$) and summing over the genus of the domain. The algorithm of the topological vertex is proved
in full generality in \cite{MOOP}.

Bouchard-Klemm-Mari\~{n}o-Pasquetti propose the Remodeling Conjecture \cite{BKMP1}, an algorithm of constructing the B-model topological open
string amplitudes in all genera of $\check{X}$ following \cite{Mar}, using Eynard-Orantin's topological recursion from the theory of matrix models \cite{EO07}.
Combined with the mirror symmetry prediction, this gives an algorithm of computing generating functions $F_{g,h}$ of open GW
invariants of $(X,L)$ obtained by fixing a topological type of the domain (determined by the genus $g$ and number $h$ of boundary circles) and summing
over the topological types of the map. Eynard-Orantin study the Remodeling Conjecture for
any smooth symplectic toric Calabi-Yau threefolds \cite{EO15}.


\subsection{Open GW invariants for 3-dimensional Calabi-Yau smooth DM stacks}
There have been attempts to generalize some of the above results to 3-dimensional Calabi-Yau smooth toric Deligne-Mumford (DM) stacks.
The closed GW theory of orbifolds has been studied for a long time. 
The physics literature dates back to early 1990s such as \cite{CV, Za}, which study the quantum cohomology ring of orbifolds.
The mathematical definition is given by Chen-Ruan \cite{CR02} for symplectic orbifolds and by Abramovich-Graber-Vistoli \cite{AGV02,AGV08} for smooth
DM stacks. Toric varieties are defined by a fan, while smooth toric DM stacks are defined by a stacky fan 
\cite{BCS05}. The coarse moduli of a smooth toric DM stack $\cX$
is a toric variety $X_\Si$ defined by a simplicial fan $\Si$.  A toric orbifold is
a smooth toric DM stack with trivial generic stabilizer. Any smooth toric DM stack $\cX$ is a $K$-gerbe over its rigidification $\cX^\rig$, where $K$ is the generic stabilizer (which 
is a finite abelian group) and $\cX^\rig$ is a toric orbifold. 

The definition of Aganagic-Vafa branes can be extended to the setting of 3-dimensional Calabi-Yau smooth toric DM stacks with semi-projective coarse moduli spaces. 
These branes are diffeomorphic to $[(S^1\times \bR^2)/G_\tau]$ where $G_\tau$ is a finite abelian group containing the generic stabilizer $K$. 
The open GW invariants of 3-dimensional Calabi-Yau smooth toric DM stacks are defined via
localization \cite{Ro14}, generalizing the methods in \cite{KL01}. By localization, open and closed GW invariants of a smooth toric Calabi-Yau 3-fold 
can be obtained by gluing the GW vertex, a generating function of open GW invariants of $\bC^3$, which can
be reduced to a generating function of certain cubic Hodge integrals \cite{DF}.
Similarly, open and closed orbifold GW invariants of a 3-dimensional Calabi-Yau smooth toric DM stack
can be obtained by gluing the orbifold GW vertex, a generating function of open GW invariants of $[\bC^3/G]$ (where $G$ is a 
finite abelian subgroup of $SL(3,\bC)$), which
can be reduced to a generating function of certain cubic abelian Hurwitz-Hodge integrals \cite{Ro14}.
The GW vertex has been evaluated in the general case \cite{LLLZ, MOOP}.
The orbifold GW vertex has been evaluated for  $[\bC^2/\bZ_n]\times \bC$  where
$\bC^2/\mathbb{Z}_n$ is the $A_{n-1}$ surface singularity \cite{Zo15, RZ13, RZ15, Ro15}, but not
in the general case.

As for mirror symmetry, a mirror theorem for disk invariants of $[\mathbb{C}^3/\mathbb{Z}_4]$ is proved in \cite{BC11}.
The Remodeling Conjecture is also expected to predict higher genus open GW invariants of toric Calabi-Yau
3-orbifolds via mirror symmetry \cite{BKMP1, BKMP2}.

\subsection{Summary of results}
In this paper we study open-closed orbifold GW invariants of a 3-dimensional Calabi-Yau smooth toric DM stack $\cX$ 
relative to an Aganagic-Vafa A-brane $\cL$ at all genera.

Open GW invariants of the pair $(\cX, \cL)$ 
count holomorphic maps from orbicurves to $\cX$ with boundaries mapped to $\cL$. Open-closed orbifold GW invariants
of the pair $(\cX,\cL)$ depend on the following discrete data: 
\begin{itemize}
\item[(i)] topological type $(g,h)$ of the domain orbicurve $(\Sigma,\partial \Sigma)$ , where $g$ is the genus and $h$ is the number of boundary holes;
\item[(ii)] number of interior marked points $n$;
\item[(iii)] topological type of the map $u:(\Si,\bSi =\coprod_{i=1}^h R_i )\to (\cX,\cL)$ given by the degree $\beta' =u_*[\Si]\in H_2(\cX,\cL;\bZ)$
and each $[u_*(R_i)]\in H_1(\cL;\bZ)\cong \bZ\times G_\tau$, collectively denoted by $\vmu=([u_*(R_1)],\dots, [u_*(R_h)])$;
\item[(iv)] framing $f\in \bZ$ of the Aganagic-Vafa A-brane $\cL$.
\end{itemize}
Let $\Mbar_{(g,h),n}(\cX,\cL\mid \beta',\vmu)$ be the moduli space parametrizing holomorphic maps with discrete data (i)-(iii).
We use the evaluation maps $\ev_i$, $i=1,\ldots, n$ at interior points to pull
back classes in the orbifold Chen-Ruan cohomology $H^*_\orb(\cX)$ of $\cX$ to obtain open-closed GW invariants. More precisely,
$\cL$ intersects a unique $1$-dimensional orbit $\fo_\tau \cong \bC^*\times \cB G_\tau$.
Given $\gamma_1,\ldots, \gamma_n\in H_\orb^*(\cX;\bQ)$, we define open-closed orbifold GW
invariant $\langle \gamma_1,\ldots,\gamma_n\rangle^{\cX,(\cL,f)}_{g,\beta',\vmu}$ via localization
using a circle action determined by the framing $f\in \bZ$; this is a rational number depending on $f$
and can be viewed as an equivariant invariant. We also provide another definition based on algebraic relative Gromov-Witten theory, which generalizes
the definition in \cite{LLLZ} for smooth toric Calabi-Yau 3-folds.

When $\cX$ is a symplectic toric Calabi-Yau 3-orbifolds (i.e. when the generic stabilizer $K$ is trivial), 
$G_\tau \cong \bmu_\fm$ is cyclic.
In this case, for each topological type $(g,h)$ of the domain bordered
Riemann surface we define a generating function $F_{g,h}^{\cX,(\cL,f)}$ of open-closed GW invariants which
takes value in $H^*_\orb(\cB G;\bC)^{\otimes h}$, where $H^*_\orb(\cB G;\bC)=\oplus_{\lambda \in G}\bC \mathbf \one_\lambda$.

In particular, the disk potential $F_{0,1}^{\cX,(\cL,f)}$ takes values in $H^*_\orb(\cB G;\bC)$.  
When $\cL$ is an outer brane\footnote{We work with both inner and outer branes. 
See Section \ref{sec:AV-brane} for the definition.}, the A-model disk potential is 
\[
F_{0,1}^{\cX,(\cL,f)}(\btau_2, X)=\sum_{\beta',n\geq 0} \sum_{(\mu, \lambda)\in H_1(\cL;\bZ)\cong \bZ\times G_\tau}
\frac{\langle (\btau_2)^n\rangle^{\cX,(\cL,f)}_{0,\beta',(\mu,\lambda)}}{n!} \cdot X^\mu
(\xi_0)^{\bar{\lambda}} \mathbf\one_{\lambda^{-1}}
\]
where $\btau_2$ is certain equivariant second Chen-Ruan cohomology class of $\cX$, $\xi_0$ is
an $\fm$-th root of $-1$, and $\bar{\lambda}\in \{0,1,\ldots,\fm-1\}$ corresponds to $\lambda\in G$ under a group isomorphism 
$G_\tau \cong \bZ/\fm\bZ$. 
The precise definition of $\btau_2$ and $F_{(0,1)}^{\cX,(\cL,f)}$ will be given in Section \ref{sec:Fgh}.
(Throughout the paper, $\beta'$ denotes a relative homology class in $H_2(\cX,\cL;\bZ)$ whereas $\beta$ denotes an absolute homology class in $H_2(\cX;\bZ)$.)
In this paper, we prove a mirror theorem regarding the disk potential
$F_{0,1}^{\cX,(\cL,f)}$ when $\cX$ is a semi-projective toric Calabi-Yau 3-orbifold.
Mirror symmetry relates the A-model topological string theory on a Calabi-Yau 3-fold to the B-model topological string theory on the mirror Calabi-Yau 
3-fold. The mirror of a semi-projective toric Calabi-Yau 3-fold is a Calabi-Yau hypersurface in $\bC^2\times (\bC^*)^2$ 
defined by an equation $uv=H(x,y,q)$, where $(u,v)\in \bC^2$, $(x,y)\in (\bC^*)^2$, and $q$ is the complex moduli parametrizing the B-model.
The function $H(x,y,q)$ is determined by both the combinatorial toric data of $\cX$ and the framed brane $(\cL,f)$. 
The affine curve $C_q=\{H(x,y,q)=0\}$ in $(\bC^*)^2$ is called the \emph{mirror curve}. 
We can fix a labeling of the $\fm$ points with $x=0$  on the mirror curve by 
the elements in $G_\tau^*=\Hom(G,\bC^*)\cong \bZ/\fm\bZ$. For each $\eta\in G_\tau^*$,  there is a small open 
neighborhood $U^\epsilon_\eta$ in the (compactified) mirror curve of the $x=0$ point labelled by $\eta$, and a branch
$(\log y)_{U_\eta^\epsilon}$ of $\log y$ defined on $U_\eta^\epsilon$, where  $y=y(x,q)$ is defined implicitly by 
the equation $H(x,y,q)=0$.  When $\cL$ is an outer brane, the closure of the
1-dimensional orbit intersecting $\cL$ contains a unique torus fixed (stacky) point $\fp_\si = \cB G_\si$, where
$G_\si$ is the inertia group of $\fp_\si$. With the above convention, we state our open mirror theorem as follows. 
\begin{theorem}\label{thm:main}
Under the closed mirror map $\btau_2=\btau_2(q)$ and the open mirror map  $X=xe^{A(q)}$ (the explicit formula of $\btau_2(q)$ and 
$A(q)$ will be given in Section \ref{sec:mirror}),
\[
x\frac{\partial}{\partial x}\bigg(\sum_{\eta \in \bmu^*_\fm}(\log y)_{U_\eta^\epsilon} (q,x) \phi_\eta \bigg)
=  \frac{\partial^2}{\partial x^2} F^{\cX,(\cL,f)}_{0,1}(\btau_2,X)
\]
where $\{\phi_\eta\}_{\eta\in G}$ is the canonical basis of $H^*_\orb(\cB G_\tau;\bC)$. 
\end{theorem}

\begin{remark}\label{Ke-Zhou}
The definition of the disk function $F_{0,1}^{\cX,(\cL,f)}$
and the formulation of the above Theorem 1.1 are slightly different from those in
the first version of this paper in 2012 \cite{FLT}, but the above Theorem \ref{thm:main} implies \cite[Theorem 1.1]{FLT}, which is used to prove an open version of
Ruan's Crepant Resolution Conjecture for disk
invariants of toric Calabi-Yau 3-orbifold relative to an effective outer Aganagic-Vafa brane
\cite{KZ15}.
\end{remark}

\subsection{Similar results for compact Lagrangian tori}
There are other open GW invariants relative to different types of Lagrangian submanifolds. C.-H. Cho \cite{Cho}
and J. Solomon \cite{So} define disk invariants of a compact symplectic manifold in real dimension four and six  relative to a Lagrangian submanifold which is the fixed locus of an anti-symplectic involution.
The mirror theorem for disk invariants for the quintic 3-fold relative
to the real quintic is conjectured in \cite{Wa} and proved in \cite{PSW}. It has been generalized to compact Calabi-Yau 3-folds which are projective complete
intersections \cite{PZ}, where a mirror theorem for genus one open GW invariants (annulus invariants) is also proved.

Open orbifold GW invariants of compact toric orbifolds with respect to Lagrangian torus fibers of the moment map
are defined in \cite{CP}, which generalizes the work of \cite{FOOO} on compact toric manifolds. The third author and collaborators prove mirror theorems on disk invariants in this context \cite{CLLT, CCLT1}. The third author and collaborators also prove mirror theorems on disk invariants of
toric Calabi-Yau manifolds/orbifolds (which must be non-compact) with respect to Lagrangian torus fibers of the Gross fibration \cite{CLT, CCLT2}. 

\subsection{Applications} The main theorem (Theorem \ref{thm:main}) of this paper has several applications. Here we mention two important applications.

\begin{itemize}
\item As mentioned in Remark \ref{Ke-Zhou} above, Theorem \ref{thm:main} has been applied to prove
an open version of Ruan's Crepant Resolution Conjecture for disk invariants of a  toric Calabi-Yau 3-orbifold relative to an effective outer Aganagic-Vafa brane \cite{KZ15}. 
Using Theorem \ref{thm:main}, S. Yu \cite{Yu} proves an open version of Crepant Transformation Conjecture (and in particular Crepant Resolution Conjecture) for
disk invariants of  a semi-projective  toric Calabi-Yau 3-orbifold relative to a general (effective or ineffective, inner or outer) Aganagic-Vafa brane defined in Section \ref{sec:AV-brane} of this paper. 
 This generalizes Open Crepant Resolution Conjecture (OCRC) for disk invariants of  $[\bC^2/\bZ_n]\times \bC$ relative to possibly ineffective Aganagic-Vafa branes proved in \cite{BCR}.

\item Recently, the first two authors and Zong prove the BKMP Remodeling Conjecture
for all semi-projective toric Calabi-Yau 3-orbifolds \cite{FLZ}. Theorem \ref{thm:main} is one of the
key ingredients of this proof.
\end{itemize}

\subsection{Overview of the paper}
The rest of the paper is organized as follows. In Section \ref{sec:toricDMstacks}  we review the necessary materials concerning smooth toric DM stacks. In Section \ref{sec:OCGW}
we apply localization to relate open-closed GW invariants and descendant GW invariants
of 3-dimensional smooth Calabi-Yau toric DM stacks. In Section \ref{sec:mirror} we prove a mirror theorem for orbifold disk invariants.

\subsection*{Acknowledgments}
The first author would like to thank Kwokwai Chan, Naichung Conan Leung and Yongbin Ruan for valuable discussions.
We wish to thank Song Yu for corrections to the previous version of this paper. We wish to thank anonymous referees for their comments and suggestions. 
The research of the first author was partially supported by  NSF DMS-1206667,
a start-up grant at Peking University, NSFC 11831017, NSFC 11890661, and NSFC 12125101.
The research of the second author was partially supported by NSF DMS-1159416
and NSF DMS-1206667. The research of the third author was partially supported by Simons Foundation Collaboration Grant and NSF DMS-1506551.

\section{Smooth Toric DM Stacks}\label{sec:toricDMstacks}
In this section, we follow the definitions in \cite[Section 3.1]{Ir09}, with slightly different
notation. We work over $\bC$.

\subsection{Definition} \label{sec:definition}
Let $N$ be a finitely generated abelian group, and let $N_\bR=N\otimes_\bZ\bR$. We have a short exact sequence of
(additive) abelian groups:
$$
0 \to N_\tor  \to N \to \bar{N}=N/N_\tor \to 0,
$$
where $N_\tor$ is the subgroup of torsion elements in $N$. Then $N_\tor$ is a finite abelian group, and
$\bar{N}=\bZ^n$, where $n=\dim_\bR N_\bR$. The natural projection $N\to \bar{N}$ is denoted $b\mapsto \bar{b}$.
A \emph{smooth toric DM stack} is an extension of toric varieties \cite{Fu93, BCS05}.
A smooth toric DM stack is given by the following data:
\begin{itemize}
\item $b_1,\dots, b_{r'} \in N$ which generate a subgroup of $N$ of  finite index, and
\item a simplicial fan $\Si$ in $N_\bR$ such that the set of $1$-cones is
$$
\{\rho_1,\dots,\rho_{r'}\},
$$
where $\rho_i=\bR_{\ge 0} \bar b_i$, $i=1,\dots, r'$.
\end{itemize}
The datum $\mathbf \Si=(\Si, (b_1,\dots, b_{r'}))$ is a \emph{stacky fan} in the sense of \cite{BCS05}.
The vectors $b_1,\dots, b_{r'}$ may or may not generate $N$; if they do not, we choose \emph{additional} vectors  $b_{r'+1},\dots, b_r$ such that $b_1,\dots, b_r$ generate $N$.
There is a surjective group homomorphism
\begin{eqnarray*}
\phi: & \tN :=\oplus_{i=1}^r \bZ\tb_i & \lra  N,\\
        &  \tb_i  & \mapsto b_i.
\end{eqnarray*}
Define $\bL :=\Ker(\phi) \cong \bZ^k$, where $k:=r-n$. Then
we have the following short exact sequence of finitely generated abelian groups:
\begin{equation}\label{eqn:NtN}
0\to \bL  \stackrel{\psi }{\lra} \tN  \stackrel{\phi}{\lra} N\to 0.
\end{equation}
Applying $ - \otimes_\bZ \bC^*$ to \eqref{eqn:NtN}, we obtain an exact
sequence of abelian groups:
\begin{equation}\label{eqn:bT}
1 \to K \to G \to \tbT\to \bT \to 1,
\end{equation}
where
\begin{eqnarray*}
\bT&:=& N\otimes_\bZ \bC^* =\bar{N}\otimes_\bZ \bC^* \cong (\bC^*)^n,\\
\tbT &:=& \tN\otimes_\bZ \bC^* \cong (\bC^*)^r,\\
G &:=& \bL\otimes_\bZ \bC^* \cong (\bC^*)^k,\\
K &:=&  \Tor_1^{\bZ}(N,\bC^*)\cong N_\tor.
\end{eqnarray*}
The action of $\tbT$ on itself extends to a $\tbT$-action on $\bC^r = \Spec\bC[Z_1,\dots, Z_r]$.
The torus $G$ acts on $\bC^r$ via the group homomorphism $G\to \tbT$ in \eqref{eqn:bT}, so $K\subset G$
acts on $\bC^r$ trivially. The isomorphism $K\cong N_\tor$ is not canonical. 

With the above preparation, we are now ready to define a smooth toric DM stack $\cX$.
Let
$$
\cA=\{I\subset \{1,\dots, r\}: \text{$\sum_{i\notin I} \bR_{\ge 0} \bar b_i$ is a cone of $\Si$}\}
$$
be the set of anti-cones; note that $\{r'+1,\ldots, r\} \subset I$ for any anti-cone $I\subset \cA$. 
Given $I\subset \{1,\ldots, r\}$, let $\bC^I$ be the subvariety of $\bC^r$ defined by the ideal in 
$\bC[Z_1,\ldots, Z_r]$ generated by $\{ Z_i \mid i\notin  I\}$.  
Define the smooth toric DM stack $\cX$ as the  quotient stack
$$
\cX:=[U_\cA/ G],
$$
where
$$
U_\cA:=\bC^r \backslash \bigcup_{I\notin \mathcal A} \bC^I 
= \bigcap_{I\notin \mathcal{A}} \left(\bC^r \backslash \bC^I\right).
$$
Note that for $i=r'+1,\ldots,r$, $\bR_{\geq 0} b_i$ is not a cone in $\Si$, so 
$\{ i\}' := \{1,\ldots, r\} \backslash\{i\} \notin \cA$. Therefore, 
$$
U_\cA \subset \bigcap_{i=r'+1}^r \left(\bC^r\backslash \bC^{\{i\}'}\right) =\bC^{r'}\times (\bC^*)^{r-r'}. 
$$ 
The stack $\cX$ contains the DM torus $\cT := [\tbT/G]$ as a dense open subset, and
the $\tbT$-action on $U_{\cA}$ descends to a $\cT$-action on $\cX$.
The smooth toric DM stack $\cX$ is a {\em toric orbifold} if the $G$-action
on $\tbT$ is free. 

Let $G^\rig = G/K$. Then $G^\rig$ acts freely on $\tbT$ and $\tbT/G^\rig =\bT$. The {\em rigidification} of the smooth toric DM stack $\cX$ is the toric orbifold
$$
\cX^{\rig} = [U_{\cA}/G^\rig]. 
$$
The coarse moduli space of the stack $\cX$ is the simplical toric variety $X_\Si$ defined by the simplicial fan $\Si$ in $N_\bR\cong \bR^n$.  
By \cite[Theorem I]{FMN}, the morphism $\cX\to X_\Si$ factors canonically via toric morphisms 
\begin{equation}\label{eqn:rigid}
\cX \lra \cX^{\rig} \lra \cX^{\can}\lra X_{\Si}
\end{equation}
where
\begin{itemize}
\item $\cX\lra \cX^\rig$ is a $K$-gerbe over $\cX^\rig$;
\item $\cX^\rig \lra \cX^\can$ is a fibered product of roots of toric divisors;
\item $\cX^\can\lra X_\Sigma$ is the minimal orbifold having $X_\Si$ as coarse moduil space.
\end{itemize}
Restricting \eqref{eqn:rigid} to the open substack $\cT \subset \cX$, one obtains
$\cT \cong \bT \times \cB K \lra \bT \lra \bT \lra \bT$,  
where $\bT\times \cB K \lra \bT$ is the projection to the first factor, and $\bT\lra \bT$ is the identity map.

\begin{remark}
The purpose of introducing additional vectors $b_{r'+1},\dots, b_{r}$ is to ensure $G$ is \emph{connected}.
The stacky fan $\mathbf \Si$ together with the extra vectors $b_{r'+1},\dots, b_r$ is an \emph{extended stacky fan} in the
sense of Jiang \cite{Ji08}. It follows from the definition that $\{r'+1,\dots, r\}\subset I$ for any $I\in \mathcal A$.
\end{remark}

Let $M$, $\tM$, and $\bL^\vee$ be the character lattices
of the tori $\bT$, $\tbT$, and $G$, respectively:
\begin{eqnarray*}
M &=& \Hom(N,\bZ)  = \Hom(\bT,\bC^*), \\
\tM &=& \Hom(\tN,\bZ)= \Hom(\tbT,\bC^*),\\
\bL^\vee &=& \Hom(\bL,\bZ) =\Hom(G,\bC^*).
\end{eqnarray*}
Applying $\Hom(-,\bZ)$ to \eqref{eqn:NtN}, we obtain the following
exact sequence of (additive) abelian groups:
\begin{equation}
  \label{eqn:exact-sequence-M}
  0  \lra M \stackrel{\phi^\vee}{\lra} \tM \stackrel{\psi^\vee}{\lra} \bL^\vee \lra \Ext^1(N,\bZ) \lra 0
\end{equation}
Therefore, the group homomorphism $\psi^\vee: \tM \lra \bL^\vee$
is surjective if and only if $N_\tor =0$.

We now consider a class of examples of 3-dimensional Calabi-Yau smooth toric DM stacks of the form $[\bC^3/\bZ_3]$. Let 
$\omega=e^{\frac{2\pi\sqrt{-1}}{3}}$ be the generator of $\bZ_3$. 
Given $i,j,k\in \{0,1,2\}$ such that $i+j+k\in 3\bZ$, we define
$\cX_{i,j,k}$ to be the quotient stack of the following
$\bZ_3$-action on $\bC^3$:
$$
\omega \cdot (Z_1,Z_2, Z_3) = (\omega^i Z_1, \omega^j Z_2, \omega^k Z_3).
$$
In the following example, we consider
$$
\cX_{1,1,1},\quad 
\cX_{1,2,0}=[\bC^2/\bZ_3]\times \bC,\quad 
\cX_{0,0,0}=\bC^3\times \cB\bZ_3.
$$
\begin{example}\label{ex:definition}
\begin{enumerate}
\item $\cX=\cX_{1,1,1}$. The toric data are given as follows.
\begin{gather*}
N=\bZ^3,\quad N_\tor=0;\\
b_1=(1,0,1),b_2=(0,1,1),b_3=(-1,-1,1),b_4=(0,0,1);\\ 
r=4,r'=3,k=1;\\
\Si=\{\text{the $3$-cone spanned by $\{b_1,b_2,b_3\}$, and its faces, and faces of faces, etc.}\};\\
\cA=\{I\subset\{1,2,3,4\}: 4\in I\};\\
\bL\cong \bZ,\quad \bL^\vee\cong \bZ;\\
\end{gather*}
\begin{figure}[h]
\begin{center}
\psfrag{KP2}{\footnotesize $\cO_{\bP^2}(-3)$}
\psfrag{C3Z3}{\footnotesize $\cX_{1,1,1}$}
\includegraphics[scale=0.6]{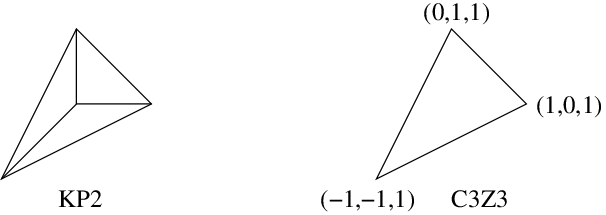}
\end{center}
\caption{$\cX_{1,1,1}$ and its crepant resolution $\cO_{\bP^2}(-3)$}
\end{figure}
\item $\cX=\cX_{1,2,0}$, transversal $A_2$-singularity. The toric data are given as follows.
\begin{gather*}
N=\bZ^3,\quad N_\tor =0;\\
b_1=(1,0,1),b_2=(0,3,1),b_3=(0,0,1),b_4=(0,1,1), b_5=(0,2,1);\\ 
r=5,r'=3,k=2;\\
\Si=\{\text{the $3$-cone spanned by $\{b_1,b_2,b_3\}$, and its faces, and faces of faces, etc.}\};\\
\cA=\{I\subset\{1,2,3,4,5\}: \{ 4, 5\} \subset I \};\\
\bL\cong \bZ^2,\quad \bL^\vee\cong \bZ^2.
\end{gather*}
\begin{figure}[h]
\begin{center}
\psfrag{X12}{\footnotesize $\cX_{1,2,0}$}
\includegraphics[scale=0.6]{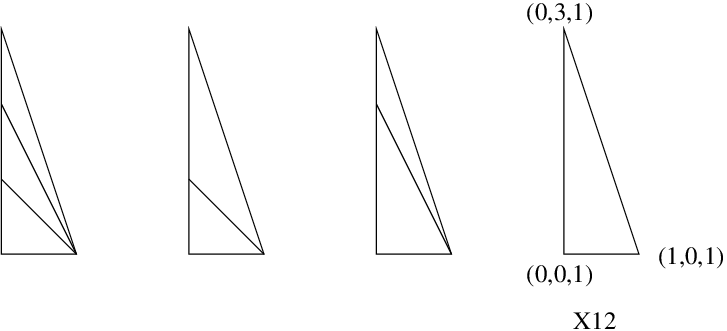}
\end{center}
\caption{$\cX_{1,2,0}$ and its (partial) crepant resolutions}
\end{figure}
\item $\cX=\cX_{0,0,0}$.  The toric data is given as follows.
\begin{gather*}
N=\bZ^3\oplus \bZ_3 \quad N_\tor =\bZ_3;\\
b_1=(1,0,0,0),b_2=(0,1,0,0),b_3=(0,0,1,0),b_4=(1,0,0,1);\\ 
r=4,r'=3,k=1;\\
\Si=\{\text{the $3$-cone spanned by $\{b_1,b_2,b_3\}$, and its faces, and faces of faces, etc.}\};\\
\cA=\{I\subset\{1,2,3,4\}: 4\in I \};\\
\bL\cong \bZ,\quad \bL^\vee\cong \bZ.
\end{gather*}
\end{enumerate}
\end{example}

\subsection{Equivariant line bundles and torus-invariant Cartier divisors}

A character $\chi\in \tM$ gives a $\tbT$-action on $\bC^r \times \bC$ by
$$
(\tilde{t}_1,\ldots, \tilde{t}_r)\cdot (Z_1,\ldots, Z_r, u)
=(\tilde{t}_1 Z_1,\ldots, \tilde{t}_r Z_r, \chi(\tilde{t}_1,\ldots, \tilde{t}_r) u),
$$
where
$$
(\tilde{t}_1,\ldots, \tilde{t}_r)\in \tbT \cong (\bC^*)^r, \quad
(Z_1,\ldots, Z_r)\in \bC^r,\quad u\in \bC.
$$
Therefore $\bC^r\times \bC$ can be viewed as the total space
of a $\tbT$-equivariant line bundle $\tL_\chi$ over $\bC^r$.
If
$$
\chi(\tilde{t}_1,\ldots, \tilde{t}_r) =\prod_{i=1}^r \tilde{t}_i^{c_i},
$$
where $c_1,\ldots, c_r\in \bZ$,  then
$$
\tL_{\chi}= \cO_{\bC^r}(\sum_{i=1}^r c_i \tcD_i),
$$
where $\tcD_i$ is the $\tbT$-divisor in $\bC^r$ defined by $Z_i=0$.
We have
$$
\tM \cong \Pic_{\tbT}(\bC^r) \cong H^2_{\tbT}(\bC^r;\bZ),
$$
where the first isomorphism is given by $\chi\mapsto \tL_\chi$
and the second isomorphism is given by the $\tbT$-equivariant
first Chern class $(c_1)_{\tbT}$.  Define
$$
D_i^\cT:= (c_1)_{\tbT}(\cO_{\bC^r}(\tcD_i)) \in H^2_{\tbT}(\bC^r;\bZ) \cong H^2_{\cT}([\bC^r/G];\bZ).
$$
Then $\{ D_1^\cT,\ldots, D_r^\cT \}$ is a $\bZ$-basis of
$H^2_{\tbT}(\bC^r;\bZ)\cong \tM$ dual to the $\bZ$-basis
$\{ \tb_1,\ldots, \tb_r\}$ of $\tN$.
We have a commutative diagram
$$
\begin{CD}
\Pic_{\tbT}(\bC^r) @>{\iota_{\cT}^*}>>  \Pic_{\tbT}(U_\cA) @>{\cong}>>\Pic_{\cT}(\cX) \\
@V{(c_1)_{\tbT}}VV  @V_{(c_1)_{\tbT}}VV @V_{(c_1)_{\cT}}VV \\
H^2_{\tbT}(\bC^r;\bZ) @>{\iota_{\cT}^*}>> H^2_{\tbT}(U_\cA;\bZ) @>{\cong}>> H^2_{\cT}(\cX;\bZ),
\end{CD}
$$
where $\iota_{\cT}^*$ is a surjective group homomorphism induced
by the inclusion $\iota: U_\cA\hookrightarrow \bC^r$, and
$$
\Ker(\iota_{\cT}^*)= \bigoplus_{i=r'+1}^r \bZ D_i^{\cT}
$$
Therefore,
$$
\Pic_{\cT}(\cX)\cong H^2_{\cT}(\cX;\bZ) \cong \tM/\oplus_{i=r'+1}^r \bZ D^\cT_i
$$
Let $\bar{D}_i^\cT := \iota_\cT^* D_i^{\cT}$. Then
$$
\bar{D}_i^\cT =0,\quad i=r'+1,\ldots, r,
$$
and
$$
H^2_{\cT}(\cX;\bZ) = \bigoplus_{i=1}^{r'} \bZ \bar{D}^\cT_i \cong\bZ^{r'}.
$$
For $i=1,\ldots,r'$, $\tcD_i\cap U_\cA$ is a $\tbT$-divisor in $U_\cA$, and it descends
to a $\bT$-divisor $\cD_i$ in $\cX$. We have
$$
\bar{D}_i^{\cT}=(c_1)_{\cT}(\cO_{\cX}(\cD_i)),\quad i=1,\ldots, r'.
$$
For $i=r'+1,\ldots, r$, $\tcD_i\cap U_\cA$ is empty, so it is the zero $\tbT$-divisor.

\subsection{Line bundles and Cartier divisors}
We have group isomorphisms
$$
\bL^\vee \cong \Pic_G(\bC^r) \cong H^2_G(\bC^r;\bZ),
$$
where the first isomorphism is given by $\chi\in \bL^\vee=\Hom(G,\bC^*)\mapsto \tL_\chi$, and
the second isomorphism is given by the $G$-equivariant first Chern class $(c_1)_G$.
We have a commutative diagram
$$
\begin{CD}
\Pic_G(\bC^r) @>{\iota^*}>>  \Pic_G(U_\cA) @>{\cong}>>\Pic(\cX) \\
@V{(c_1)_G}VV  @V_{(c_1)_G}VV @V_{c_1}VV \\
H^2_G(\bC^r;\bZ) @>{\iota^*}>> H^2_G(U_\cA;\bZ) @>{\cong}>> H^2(\cX;\bZ),
\end{CD}
$$
where $\iota^*$ is a surjective group homomorphism induced
by the inclusion $\iota: U_\cA\hookrightarrow \bC^r$.
The surjective map $H^2_G(\bC^r;\bZ)\to H^2(\cX;\bZ)$ is the restriction
of the Kirwan map
$$
\kappa: H^*_G(\bC^r;\bZ)\lra H^*(\cX;\bZ).
$$
Define
$$
D_i:= (c_1)_G(\cO_{\bC^r}(\tcD_i)) \in H^2_G(\bC^r;\bZ) \cong H^2([\bC^r/G];\bZ).
$$
Then
$$
\Ker(\iota^*)= \bigoplus_{i=r'+1}^r \bZ D_i.
$$
Therefore,
$$
\Pic(\cX)\cong H^2(\cX;\bZ) \cong \bL^\vee/\oplus_{i=r'+1}^r \bZ D_i.
$$
Recall that $\psi^\vee: \tM\lra \bL^\vee$
is surjective if and only if $N_\tor=0$.
Let
$$
\bar{D}_i = c_1(\cO_{\cX}(\cD_i))\in H^2(\cX;\bZ),\quad i=1,\ldots, r.
$$
The map
$$
\bar{\psi}^\vee: \Pic_{\cT}(\cX)\cong H^2_{\cT}(\cX;\bZ) \to \Pic(\cX)\cong H^2(\cX;\bZ),
$$
given by
$$
\bar{D}_i^\cT \mapsto \bar{D}_i,\quad i=1,\ldots, r',
$$
is surjective if and only if $N_\tor=0$. In general,
$\Coker(\psi^\vee)\cong \Coker(\bar{\psi}^\vee)$ is a finite abelian group.

Pick a $\bZ$-basis $\{e_1,\ldots, e_k\}$ of $\bL\cong \bZ^k$, and let
$\{e_1^\vee,\ldots, e_k^\vee\}$ be the dual $\bZ$-basis of $\bL^\vee$. For each
$a\in \{1,\ldots,k\}$, we define a {\em charge vector}
$$
l^{(a)}= (l^{(a)}_1,\ldots, l^{(a)}_r) \in \bZ^r
$$
by
$$
\psi(e_a)=\sum_{i=1}^r l^{(a)}_i \tb_i,
$$
where $\psi: \bL \lra \tN$ is the inclusion map.
Then
$$
D_i=\psi^\vee(D_i^\cT)=\sum_{a=1}^k l^{(a)}_i e_a^\vee,\quad i=1,\ldots,r,
$$
and
$$
\sum_{i=1}^r l^{(a)}_{i} b_{i} = \phi\circ \psi(e_a)=0,\quad a=1,\ldots, k.
$$

\begin{example}\label{ex:picard}
We use the notation in Example \ref{ex:definition}.
\begin{enumerate}
\item $\cX=\cX_{1,1,1}$.
\begin{gather*}
D_1=D_2=D_3=1,\ D_4=-3;\\
l^{(1)}=(1,1,1,-3);\\
\Pic_{\cT}(\cX)\cong \bZ^3,\quad  \Pic(\cX)\cong \bZ/3\bZ;
\end{gather*}
\item $\cX=\cX_{1,2,0}$.
\begin{gather*}
D_1=(0,0),\  D_2= (0,1),\ D_3=(1,0),\ D_4=(-2,1),\ D_5=(1,-2);\\
l^{(1)}=(0,0,1,-2,1),\  l^{(2)}= (0,1,0,1,-2);\\
\Pic_{\cT}(\cX)=\bZ^3,\quad \Pic(\cX) =\bZ^2 \big/ \big(\bZ(-2,1)\oplus \bZ(1,-2)\big) \cong \bZ/3\bZ.
\end{gather*}
\item $\cX=\cX_{0,0,0}$.
\begin{gather*}
D_1=3,\  D_2= 0 ,\ D_3= 0,\ D_4=-3;\\
l^{(1)}=(3,0,0,-3);\\
\Pic_{\cT}(\cX)=\bZ^3,\quad \Pic(\cX) = \bZ/3\bZ.
\end{gather*}
\end{enumerate}
\end{example}

\subsection{Torus invariant subvarieties and their generic stabilizers}\label{sec:T-subvariety}
Let $\Si(d)$ be the set of $d$-dimensional cones. For each $\si\in \Si(d)$, we 
define
$$
I_\si :=\{ i\in \{1,\ldots,r\}\mid \rho_i\not\subset \si \} \in \cA, \quad I_\si' := \{1,\ldots, r\}\setminus I_\si.
$$
Then $|I_\si'|=d$ and $|I_\si|=r-d$. Let $\tilde{V}(\si)\subset U_\cA$ be the closed subvariety defined by the ideal of $\bC[Z_1,\ldots, Z_r]$ generated by
$$
\{Z_i=0\mid \rho_i\subset \si\} = \{ Z_i=0\mid i \in I'_\si\}.
$$
Then $\cV(\si) := [\tV(\si)/G]$ is an $(n-d)$-dimensional $\cT$-invariant closed
substack of $\cX= [U_\cA/G]$.

The group homomorphism $G\cong (\bC^*)^k \lra \tbT\cong (\bC^*)^r$ is given by
$$
g\mapsto (\chi_1(g),\ldots, \chi_r(g)),
$$
where $\chi_i\in \Hom(G,\bC^*)=\bL^\vee$ is given by
$$
\chi_i(u_1,\ldots, u_k)= \prod_{a=1}^k u_a^{ l_i^{(a)} }.
$$
Let
$$
G_\si:= \{ g\in G\mid g\cdot z = z \textup{ for all } z\in \tV(\si)\} =\bigcap_{i\in I_\si} \Ker(\chi_i).
$$
Then $G_\si$ is the generic stabilizer of $\cV(\si)$. It is a finite subgroup of $G$.
If $\tau\subset \si$ then $I_\si\subset I_\tau$, so $G_\tau\subset G_\si$. There
are two special cases:
\begin{itemize}
\item Let $\{0\}$ be the unique 0-dimensional cone. Then $G_{\{0\}}=K$ is the generic
stabilizer of $\cV(\{0\})=\cX$.
\item If $\si\in \Si(n)$ where $n=\dim_\bC \cX$, then $\fp_\si:= \cV(\si)$ is a $\cT$ fixed point in $\cX$, and
$\fp_\si=\cB G_\si$.
\end{itemize}

\begin{example}\label{ex:stabilizer}
We use the notation in Example \ref{ex:definition}.
Let $\si\subset N_\bR\cong \bR^3$ denote the 3-dimensional cone spanned by $\bar{b}_1,\bar{b}_2,\bar{b}_3$.
For $j=1,2,3$, let $\tau_j$ denote the 2-dimensional cone in $N_\bR$
spanned by $\{\bar{b}_i: i\in \{1,2,3\}-\{j\}\}$.
\begin{enumerate}
\item $\cX=\cX_{1,1,1}$:
$G_\si=\bZ_3,\quad G_{\tau_1}=G_{\tau_2}=G_{\tau_3}=\{1\}$.
\item $\cX=\cX_{1,2,0}$:
$G_\si=\bZ_3 = G_{\tau_3}, \quad G_{\tau_1}=G_{\tau_2}=\{1\}$.
\item $\cX=\cX_{0,0,0}$:
$G_\si=\bZ_3 = G_{\tau_1}=G_{\tau_2}=G_{\tau_3}$.
\end{enumerate}
\end{example}

Define the set of flags in $\Si$ to be
$$
F(\Si)=\{(\tau,\si)\in \Si(n-1)\times \Si(n): \tau \subset \si\}.
$$
Given $(\tau,\si)\in F(\Si)$, let $\fl_\tau:=\cV(\tau)$ be the 1-dimensional
$\cT$-invariant subvariety of $\cX$. Then $\fp_\si$ is contained in $\fl_\tau$.
There is a unique $i\in \{1,\ldots, r'\}$ such  that $i\in I'_\si \setminus I'_\tau$.
The representation of $G_\si$ on the tangent line
$T_{\fp_\si}\fl_\tau$ to $\fl_\tau$ at the stacky point $\fp_\si$ is given
by $\chi_{(\tau,\si)}:= \chi_i|_{G_\si}: G_\si\to \bC^*$. The image $\chi_i(G_\si) \subset \bC^*$
is a cyclic subgroup of $\bC^*$; we define the order of this group to be $r(\tau,\si)$.
Then there is a short exact sequence of finite abelian groups:
$$
1\to G_\tau\lra G_\si \stackrel{\chi_{(\tau,\si)} }{\lra} \bmu_{r(\tau,\si)}\to 1,
$$
where $\bmu_a=\{ z\in \bC^*\mid z^a=1\}$ is the group of $a$-th roots of unity.

\subsection{The extended nef cone and the extended Mori cone} \label{sec:nef-NE}
In this paragraph, $\bF=\bQ$, $\bR$, or $\bC$.
Given a finitely generated abelian group $\Lambda$ with $\Lambda/\Lambda_\tor \cong \bZ^m$, define
$\Lambda_\bF= \Lambda\otimes_\bZ \bF \cong \bF^m$.
We have the following short exact sequences of vector spaces:
\begin{eqnarray*}
&& 0\to \bL_\bF\to \tN_\bF \to N_\bF \to 0,\\
&& 0\to  M_\bF\to \tM_\bF\to \bL^\vee_\bF\to 0.
\end{eqnarray*}
We also have the following isomorphisms of vector spaces over $\bF$:
\begin{eqnarray*}
&& H^2(\cX;\bF)\cong H^2(X;\bF) \cong \bL^\vee_\bF/\oplus_{i=r'+1}^r \bF D_i,\\
&& H^2_{\cT}(\cX;\bF) \cong H^2_\bT(X;\bF) \cong \tM_\bF/\oplus_{i=r'+1}^r \bF D_i^\cT,
\end{eqnarray*}
where $X$ is the coarse moduli space of $\cX$.

From now on, we assume all the maximal cones in $\Si$ are $n$-dimensional, where
$n=\dim_\bC \cX$. Given a maximal cone $\si\in \Si(n)$, we define
$$
\bK_\si^\vee := \bigoplus_{i\in I_\si }\bZ D_i.
$$
Then $\bK_\si^\vee$ is a sublattice of $\bL^\vee$ of finite index.
We define the {\em extended  $\si$-nef cone} to be
$$
\tNef_\si = \sum_{i\in I_\si}\bR_{\geq 0} D_i,
$$
which is a $k$-dimensional cone in $\bL^\vee_\bR  \cong \bR^k$. The {\em extended nef cone} of
the extended stacky fan $(\Si,b_1,\ldots, b_r)$ is
$$
\tNef_{\cX}:=\bigcap_{\si\in \Si(n)} \tNef_\si.
$$
The {\em extended $\si$-K\"{a}hler cone} $\tC_\si$ is defined to be the interior of
$\tNef_\si$; the {\em extended K\"{a}hler cone} of $\cX$, $\tC_{\cX}$,  is defined
to be the interior of the extended nef cone $\tNef_{\cX}$.

Let $\bK_\si $ be the dual lattice of $\bK_\si^\vee$; it can be viewed as an additive subgroup of $\bL_\bQ$:
$$
\bK_\si =\{ \beta\in \bL_\bQ \mid \langle D, \beta\rangle \in \bZ \  \forall D\in \bK_\si^\vee \},
$$
where $\langle-, -\rangle$ is the natural pairing between $\bL^\vee_\bQ$ and $\bL_\bQ$.
Define
$$
\bK:= \bigcup_{\si\in \Si(n)} \bK_\si.
$$
Then $\bK$ is a subset (which is not necessarily a subgroup) of $\bL_\bQ$, and $\bL\subset \bK$.

We define the {\em extended $\si$-Mori cone} $\tNE_\si\subset \bL_\bR$ to be the dual cone of
$\tNef_\si\subset \bL_\bR^\vee$:
$$
\tNE_\si=\{ \beta \in \bL_\bR\mid \langle D,\beta\rangle \geq 0 \ \forall D\in \tNef_\si\}.
$$
It is a $k$-dimensional cone in $\bL_\bR$. The {\em extended Mori cone} of
the extended stacky fan $(\Si,b_1,\ldots, b_r)$ is
$$
\tNE_{\cX}:= \bigcup_{\si\in \Si(n)} \tNE_\si.
$$
Finally, we define
$$
\bK_{\eff,\si}:= \bK_\si\cap \tNE_\si,\quad \bK_{\eff}:=
\bK\cap \tNE(\cX)= \bigcup_{\si\in \Si(n)} \bK_{\eff,\si}.
$$

\begin{example}\label{ex:effective}
\begin{enumerate}
\item $\cX=\cX_{1,1,1}$.
\begin{gather*}
\bK^\vee \cong 3\bZ, \quad \tNef_{\cX}=\bR_{\leq 0};\\
\bK \cong \frac{1}{3}\bZ , \quad \tNE_{\cX}=\bR_{\leq 0},\quad \bK_\eff=\frac{1}{3}\bZ_{\leq 0}.
\end{gather*}
\begin{figure}[h]
\begin{center}
\includegraphics[scale=0.6]{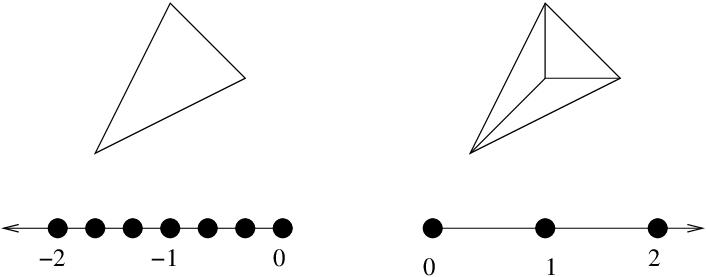}
\end{center}
\caption{$\bK_\eff$ of $\cX_{1,1,1}$ and its crepant resolution $\cO_{\bP^2}(-3)$}
\end{figure}

\item $\cX=\cX_{1,2,0}$.
\begin{gather*}
\bK^\vee\cong \bZ (-2,1)\oplus \bZ(1,-2),\quad \tNef_{\cX}= \bR_{\geq 0}(-2,1) +\bR_{\geq 0}(1,-2);\\
\bK\cong \bZ(-\frac{2}{3}, -\frac{1}{3})\oplus \bZ(-\frac{1}{3}, -\frac{2}{3}), \quad
\tNE_{\cX} = \bR_{\geq 0}(-\frac{2}{3}, -\frac{1}{3}) +  \bR_{\geq 0}(-\frac{1}{3}, -\frac{2}{3}),\\
\bK_\eff =\bZ_{\geq  0}(-\frac{2}{3},-\frac{1}{3}) +  \bZ_{\geq 0}(-\frac{1}{3},-\frac{2}{3}).
\end{gather*}
\begin{figure}[h]
\begin{center}
\includegraphics[scale=0.6]{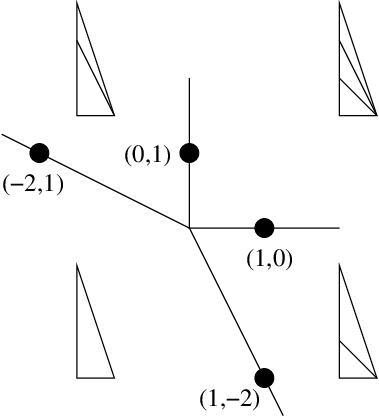}
\end{center}
\caption{The secondary fan of the crepant resolution of $\cX_{1,2,0}$}
\end{figure}
\begin{figure}[h]
\begin{center}
\includegraphics[scale=0.6]{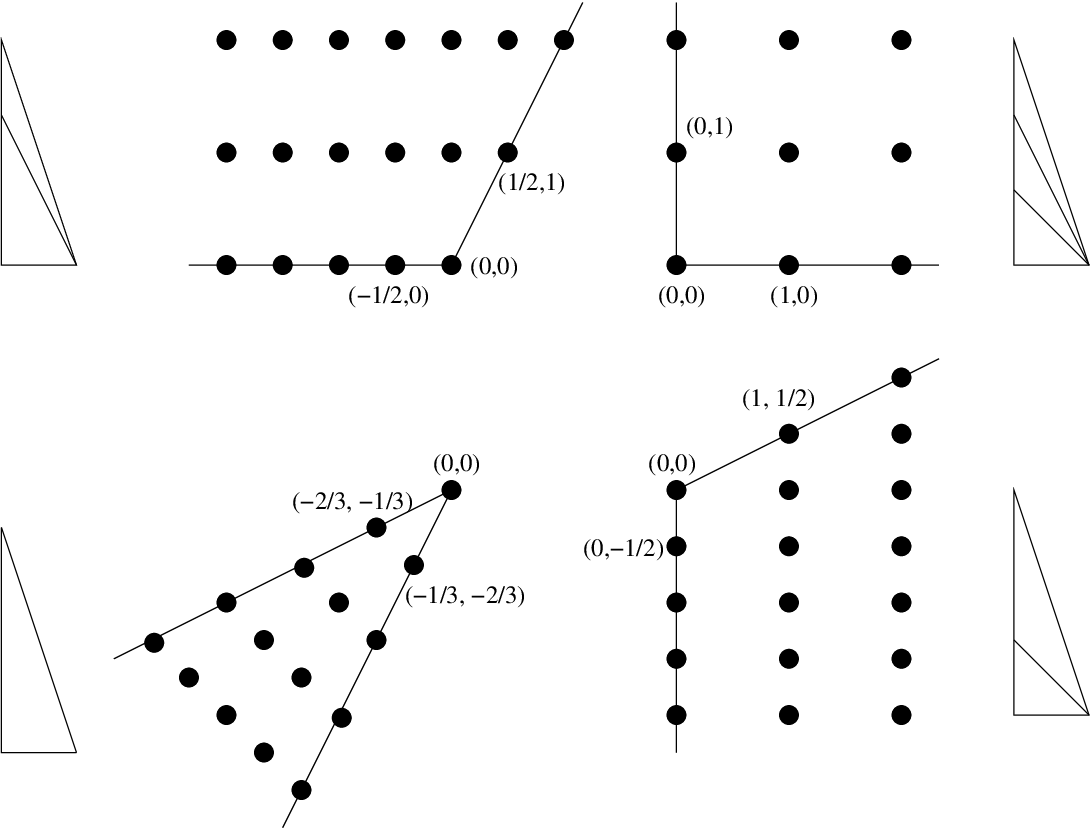}
\end{center}
\caption{$\bK_\eff$ of $\cX_{1,2,0}$ and its (partial) crepant resolutions}
\end{figure}

\item $\cX=\cX_{0,0,0}$.
\begin{gather*}
\bK^\vee\cong 3\bZ,\quad \tNef_{\cX}= \bR_{\leq 0};\\
\bK\cong \frac{1}{3}\bZ, \quad \tNE_{\cX} = \bR_{\leq 0},\quad \bK_\eff =\frac{1}{3}\bZ_{\leq 0}.
\end{gather*}

\end{enumerate}

\end{example}

\begin{assumption} \label{semi-proj}
From now on, we make the following assumptions on $\cX$.
\begin{enumerate}
\item[(a)] The coarse moduli space $X_\Si$ of $\cX$ is semi-projective.
\item[(b)] We may choose $b_{r'+1},\ldots, b_r$ such that
$\hat \rho:=D_1+\dots + D_r$ is contained in the closure
of the extended K\"ahler cone $\tC_\cX$.
\end{enumerate}
\end{assumption}

\begin{remark}
\begin{enumerate}
\item We make the above assumptions (a) and (b) so that the equivariant mirror theorem \cite[Theorem 31]{CCIT}
takes a particularly simple form. See Section \ref{sec:I} in this paper for the precise statement. 

\item By \cite[Proposition 14.4.1]{CLS}, $X_\Si$ is semi-projective if and only if $|\Si|$ is equal to the
cone spanned by $b_1,\ldots, b_r$. For example, the total space of $\cO_{\bP^1}(-3)\oplus \cO_{\bP^1}(1)$ is
a smooth toric Calabi-Yau 3-fold which is not semi-projective:
\begin{figure}[h]
\begin{center}
\includegraphics[scale=0.5]{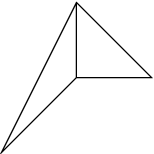}
\end{center}
\caption{$\cO_{\bP^1}(-3)\oplus \cO_{\bP^1}(1)$}
\end{figure}

\item When $\cX$ is a Calabi-Yau smooth toric DM stack, Assumption (b) holds if its coarse moduli space $X_\Si$
has a toric crepant resolution of singularities; see \cite[Remark 3.4]{Ir09}.  By \cite[Proposition 11.4.19]{CLS}, any 3-dimensional
Gorenstein toric variety $X_\Si$ has a resolution of singularities $\phi:X_{\Si'}\to X_{\Si}$ such that $\phi$ is projective and
crepant. So Assumption \ref{semi-proj} (b) holds for any 3-dimensional Calabi-Yau smooth toric DM stacks.
\end{enumerate}
\end{remark}

\subsection{Smooth toric DM stacks as symplectic quotients}

Let $G_\bR\cong U(1)^k$ be the maximal compact subgroup of $G\cong (\bC^*)^k$. Then
the Lie algebra of $G_\bR$ is $\bL_\bR$. Let
$$
\tmu: \bC^r\to \bL^\vee_\bR =\bigoplus_{a=1}^k \bR e_a ^\vee
$$
be the moment map of the Hamiltonian $G_\bR$-action on $\bC^r$, equipped with
the K\"{a}hler form
$$
\sqrt{-1} \sum_{i=1}^r dZ_i\wedge d\bar{Z}_i.
$$
Then
$$
\tmu(Z_1,\ldots, Z_r) = \sum_{i=1}^r \sum_{a=1}^k l_i^{(a)}|Z_i|^2 e^\vee_a.
$$
If $\bfr=\sum_{a=1}^k r_a e_a^\vee$ is in the extended K\"{a}hler cone of $\cX$, then
$$
\cX =[\tmu^{-1}(\bfr)/G_\bR].
$$
The generic stabilizer $K$ (which is a finite subgroup of $G\cong (\bC^*)^k$) is contained in the maximal compact subgroup $G_\bR$ of $G$. The quotient 
$G_\bR^\rig:= G_\bR/K\cong U(1)^k$
is the maximal compact subgroup of $G^\rig=G/K\cong (\bC^*)^k$, and 
$$
\cX^\rig =[\tmu^{-1}(\bfr)/G_\bR^\rig]
$$
as a symplectic quotient.

The real numbers $r_1,\ldots, r_k$ are extended K\"{a}hler parameters. The symplectic structure $\omega(\bfr)$ depends on $\bfr$. The map $\bfr\mapsto [\omega(\bfr)]$ is given by $\bL^\vee_\bR \to H^2(\cX;\bR)$.
Let $T_a=-r_a+\sqrt{-1}\theta_a$ be complexified extended K\"{a}hler parameters of $\cX$.

\subsection{The inertia stack and the Chen-Ruan orbifold cohomology} \label{sec:CR}
Given $\si\in \Si$, define
$$
\Boxs:=\Big\{ v\in N: \bar{v}=\sum_{i\in I'_\si} c_i \bar{b}_i, \quad 0\leq c_i <1\Big\}.
$$
Then $N_\tor\subset \Boxs\subset N$. If $\si$ is a $d$-dimensional cone, then the set
$\{ \sum_{i\in I'_\si} c_i \bar{b}_i: c_i\in \bR, 0\leq c_i<1 \}$
is a fundamental domain of the action of $\displaystyle{ \bar{N}_\si =\oplus_{i\in I'_\si}\bZ \bar{b}_i \cong \bZ^d }$ on 
$\displaystyle{ N_\si\otimes_{\bZ}\bR = \oplus_{i\in I'_{\si}}\bR\bar{b}_i\cong \bR^d}$.
If $\tau\subset \sigma$ then $I'_\tau\subset I'_\si$, so $\mathrm{Box}(\tau)\subset \Boxs$.

Let $\si\in \Si(n)$ be a maximal cone in $\Si$. We
have a short exact sequence of abelian groups
$$
0\to \bK_\si/\bL\to \bL_\bR/\bL\to \bL_\bR/\bK_\si\to 0,
$$
which can be identified with the following short exact sequence of multiplicative abelian groups
$$
1\to G_\si\to G_\bR \to (G/G_\si)_\bR\to 1
$$
where $G_\bR \cong U(1)^k$ is the maximal compact subgroup of $G\cong (\bC^*)^k$,
and $(G/G_\si)_\bR\cong U(1)^k$ is the maximal compact subgroup of $(G/G_\si)\cong(\bC^*)^k$.

Given a real number $x$, we recall some standard notation:
$\lfloor x \rfloor$ is the greatest integer less than or equal to  $x$,
$\lceil x \rceil$ is the least integer greater or equal to $x$,
and $\{ x\} = x-\lfloor x \rfloor$ is the fractional part of $x$.
Define $v: \bK_\si\to N$ by
$$
v(\beta)= \sum_{i=1}^r \lceil \langle D_i,\beta\rangle\rceil b_i.
$$
Then
$$
\overline{v(\beta)} = \sum_{i\in I'_\si} \{ -\langle D_i,\beta\rangle \}\bar{b}_i,
$$
so $v(\beta)\in \Boxs$. Indeed, $v$ induces a bijection $K_\si/\bL\cong \Boxs$.

For any $\tau\in \Si$ there exists $\si\in \Si(n)$ such that
$\tau\subset \si$. The bijection $G_\si \to \Boxs$ restricts
to a bijection $G_\tau\to \mathrm{Box}(\tau)$.

Define
$$
\BoxS:=\bigcup_{\si\in \Si}\Boxs =\bigcup_{\si\in\Si(n)}\Boxs.
$$
Then $N_\tor \subset \BoxS \subset N$. There is a bijection $\bK/\bL\to \BoxS$.

Given $v\in \Boxs$, where $\si\in \Si(d)$, define $c_i(v)\in [0,1)\cap \bQ$ by
$$
\bar{v}= \sum_{i\in I'_\si} c_i(v) \bar{b}_i.
$$
Suppose that  $k \in G_\si$  corresponds to $v\in \Boxs$ under the bijection $G_\si\cong\Boxs$, then
$$
\chi_i(k) = \begin{cases}
1, & i\in I_\si,\\
e^{2\pi\sqrt{-1} c_i(v)},& i \in I'_\si.
\end{cases}
$$
Define
$$
\age(k)=\age(v)= \sum_{i\notin I_\si} c_i(v).
$$

Let $IU=\{(z,k)\in U_\cA\times G\mid k\cdot z = z\}$,
and let $G$ act on $IU$ by $h\cdot(z,k)= (h\cdot z,k)$. The
inertia stack $\cI\cX$ of $\cX$ is defined to be the quotient stack
$$
\cI\cX:= [IU/G].
$$
Note that $(z=(Z_1,\ldots,Z_r), k)\in IU$ if and only if
$$
k\in \bigcup_{\si\in \Si}G_\si \textup{ and }  Z_i=0 \textup{ whenever } \chi_i(k) \neq 1.
$$
So
$$
IU=\bigcup_{v\in \BoxS} U_v,
$$
where
$$
U_v:= \{(Z_1,\ldots, Z_m)\in U_\cA: Z_i=0 \textup{ if } c_i(v) \neq 0\}.
$$
The connected components of $\cI\cX$ are
$$
\{ \cX_v:= [U_v/G] : v\in \BoxS\}.
$$
The involution $IU\to IU$, $(z,k)\mapsto (z,k^{-1})$ induces involutions
$\inv:\cI\cX\to \cI\cX$ and $\inv:\BoxS\to \BoxS$ such that
$\inv(\cX_v)=\cX_{\inv(v)}$.

In the remainder of this subsection, we consider rational cohomology, and
write $H^*(-)$ instead of $H^*(-;\bQ)$.

As a graded vector space over $\bQ$ (and as the state-space of the relevant quantum theory in physics \cite{Za}), the Chen-Ruan orbifold cohomology \cite{CR04} is defined to be
$$
H^*_\orb (\cX)=\bigoplus_{v\in \BoxS}  H^*(\cX_v)[2\age(v)].
$$
Let $\mathbf 1_v$ be the unit in $H^*(\cX_v)$. Then $\mathbf 1_v\in H^{2\age(v)}_\orb(\cX)$. In particular,
$$
H^0_\orb(\cX) =\bigoplus_{v\in N_\tor} \bQ \mathbf 1_v.
$$
Suppose that $\cX$ is a {\em proper} toric DM stack. Then
the orbifold Poincar\'e pairing on $H^*_\orb(\cX)$ is defined as
\begin{equation}\label{eqn:Poincare}
(\alpha,\beta):=\int_{\cI\cX} \alpha\cup \inv^*(\beta),
\end{equation}
We also have an equivariant pairing on $H^*_{\orb,\bT}(\cX)$:
\begin{equation}\label{eqn:T-Poincare}
(\alpha,\beta)_{\bT} := \int_{\cI\cX_{\bT}} \alpha\cup \inv^*(\beta),
\end{equation}
where
$$
\int_{ \cI\cX_{\bT}}: H_{\orb,\bT}^*(\cX) \to H_{\bT}^*(\pt) = H^*(B\bT)
$$
is the equivariant pushforward to a point. When $\cX$ is not proper,
\eqref{eqn:Poincare} is not defined, but we can still define via \eqref{eqn:T-Poincare}
an equivariant pairing $H^*_{\orb,\bT}(\cX)\otimes H^*_{\orb,\bT}(\cX)
\to \cQ_\bT$, where $\cQ_\bT$ is the fractional field of
the ring $H^*(B\bT)$.


\begin{example}\label{ex:Horb}

\begin{enumerate}
\item $\cX=\cX_{1,1,1}$.
\begin{gather*}
N=\bZ^3,\quad \BoxS=\{(0,0,0),(0,0,1),(0,0,2)\};\\
H^0_\orb(\cX)=\bQ \one_{(0,0,0)},\
H^2_\orb(\cX)=\bQ \one_{(0,0,1)},\
H^4_\orb(\cX)=\bQ \one_{(0,0,2)}.
\end{gather*} 

\item $\cX=\cX_{1,2,0}$.
\begin{gather*}
N=\bZ^3,\quad \BoxS=\{ (0,0,0), (0,2,1), (0,1,1)\};\\
H^0_\orb(\cX)=\bQ \one_{(0,0,0)},\quad
H^2_\orb(\cX)=\bQ \one_{(0,2,1)} \oplus \bQ\mathbf 1_{(0,1,1)}.
\end{gather*} 

\item $\cX=\cX_{0,0,0}$.
\begin{gather*}
N=\bZ^3\oplus \bZ_3,\quad \BoxS=N_\tor=\bZ_3=\{0,1,2\};\\
H^0_\orb(\cX)=\bQ \one_0\oplus \bQ \one_1 \oplus \bQ \one_2.
\end{gather*} 

\end{enumerate}
\end{example}

\section{All Genus Open-Closed Gromov Witten invariants}\label{sec:OCGW}

In this section, $\cX$ is a 3-dimensional Calabi-Yau smooth Deligne-Mumford stack. 

\subsection{Rigidification} \label{sec:rigid}
The rigidification $\cX^\rig$ of $\cX$ is a toric Calabi-Yau 3-orbifold.
The Calabi-Yau condition implies $\cX^{\rig}=\cX^{\can}$, where $\cX^{\can}$ is determined by the simplicial fan $\Si$ and then by choosing each $b_i$ to be the primitive generator of each $1$-cone (cf. Equation \eqref{eqn:rigid} in Section \ref{sec:definition}).
 Let $\cT'$ (resp. $\bT'$) be the subtorus of 
$\cT\cong (\bC^*)^3 \times \cB K$ (resp. $\bT\cong (\bC^*)^3$) preserving the Calabi-Yau 3-form on 
$\cX$ (resp. $\cX^\rig$). Then $\bT'\cong (\bC^*)^2$ and $\cT'\cong (\bC^*)^2 \times \cB K$. 
There is a primitive $\wu_3 \in M=\Hom(\bT, \bC^*)\cong \bZ^3$ such that $\Ker(\wu_3)=\bT'$. 
Define $M':=M/\langle \wu_3 \rangle\cong \bZ^2$. Then $\bar{N}':=\wu_3^\perp =\{ \wv \in \bar{N}: 
\langle \wu_3, \wv \rangle =0\}$ is the dual lattice of $M'=\Hom(\bT',\bC^*)$.

The simplicial fan $\Si$ is determined by a convex polytope $\Delta\subset N'_\bR=\bar{N}'\otimes_\bZ \bR$ 
together with a triangulation of $\Delta$ where all the vertices are in the lattice $\bar{N}$. 
The fan $\Si$ is a cone over this triangulation in $N'_{1,\bR}\subset N_\bR$, where
$N'_{1,\bR}=\{ \wv\in N_\bR : \langle \wu_3, \wv\rangle =1\}$. 

\subsection{Toric graph}
\label{sec:toric-graph}

Let $\bT_\bR\cong U(1)^3$ (resp. $\bT'_\bR\cong U(1)^2$) be the maximal compact subgroup
of $\bT\cong (\bC^*)^3$ (resp. $\bT'\cong (\bC^*)^2$), and we choose an $\bfr$ in the extended K\"ahler cone.
The $\bT$-action on $\cX^\rig$ restricts to a Hamiltonian $\bT_\bR$-action on the K\"{a}hler orbifold  
$(\cX^\rig,\omega(\bfr))$. Since $M_\bR$ (resp. $M_\bR'$) is canonically identified with the dual of the Lie algebra of $\bT_\bR$ (resp. $\bT'_\bR$), the K\"{a}hler form $\omega(\bfr)$ determines
a moment map $\mu_{\bT_\bR}: \cX^\rig \longrightarrow M_\bR$ up to translation by a vector in $M_\bR$.
The image $\mu_{\bT_\bR}(\cX^\rig)$ is a convex polyhedron.
The moment map $\mu_{\bT'_\bR}:\cX^\rig \longrightarrow M'_\bR$ is the composition $\pi\circ \mu_{T_\bR}$,
where $\pi:M_\bR\cong \bR^3 \to M'_{\bR}\cong \bR^2$ is the projection.
The map $\mu_{\bT'_\bR}$ is surjective. Let $\cX^\rig_1 \subset \cX^\rig$ be the union of
0-dimensional and 1-dimensional $\bT$-orbits in $\cX^\rig$. The toric graph is defined by
$\Gamma:= \mu_{T_\bR'}(\cX^\rig_1)\subset M'_\bR \cong \bR^2$. It is determined by the K\"{a}hler class 
$[\omega(\bfr)]\in H^2(\cX^\rig;\bR)=H^2(X_\Si;\bR)$ up to translation
by a vector in $M'_\bR$. The vertices (resp. edges) of $\Gamma$ are in one-to-one correspondence
to $3$-dimensional (resp. $2$-dimensional) cones in $\Si$. Conversely, the K\"{a}hler class $[\omega(\bfr)]\in H^2(\cX^\rig;\bR)$ is determined by the toric graph. 

Pulling back under  the map  $\cX\lra \cX^\rig$ defines a one-to-one correspondence between K\"{a}hler forms/classes on $\cX$ and on its rigidification $\cX^\rig$.

\subsection{Aganagic-Vafa A-branes}
\label{sec:AV-brane}
In \cite{AV00}, Aganagic-Vafa introduced a class of Lagrangian submanifolds
of semi-projective smooth toric Calabi-Yau 3-folds. In this section, we generalize this construction 
and define Aganagic-Vafa A-branes in a general 3-dimensional Calabi-Yau smooth toric
DM stack with semi-projective coarse moduli space. 

Let $\cX=[\tmu^{-1}(\bfr)/G_\bR]$ be a 3-dimensional Calabi-Yau smooth toric DM stack,
where 
$$
\bfr =\sum_{a=1}^k r_a e_a^\vee \in \tC(\cX)\subset \bL^\vee_\bR,
$$
and  $\tmu^{-1}(\mathbf{r})\subset \bC^{k+3}$ is defined by the following equations:
$$
\sum_{i=1}^{k+3} l_i^{(a)} |X_i|^2 =r_a,\quad a=1,\ldots, k.
$$
Write $X_i=\rho_i e^{\sqrt{-1}\phi_i}$, where $\rho_i=|X_i|$. An Aganagic-Vafa brane is a Lagrangian
sub-orbifold of $\cX$ of the form
$$
\cL=[\tL/G_\bR]
$$
where
$$
\tL=\{ (X_1,\ldots, X_{k+3})\in \tmu^{-1}(\mathbf{r}): \sum_{i=1}^{k+3}\hat{l}^1_i |X_i|^2=c_1,
\sum_{i=1}^{k+3} \hat{l}_i^2 |X_i|^2 = c_2, \sum_{i=1}^{k+3}\phi_i =\textup{const} \}
$$
for some $\hat{l}^\alpha_i \in \bZ, \sum_{i=1}^{k+3}\hat{l}^\alpha_i =0$, $\alpha=1,2$.
Note that the action of $G_\bR$ on $\bC^{k+3}$ preserves the subsets $\tmu^{-1}(\bfr)$ and $\tL$. If we 
view $\cX = [\tmu^{-1}(\bfr)/G_\bR]$ as a Lie groupoid (and in particular a category)  then $\cL=[\tL/G_\bR]$ is a full subcategory.

An Aganagic-Vafa brane $\cL$ intersects a unique 1-dimensional orbit 
$\fo_\tau \cong \bC^*\times \cB G_\tau$ along $\cS_\tau := \cL\cap \fo_\tau\cong S^1\times \cB G_\tau$.
The inclusion $\cS_\tau \subset \cL$ is a homotopic equivalence, so the fundamental group
of $\cL$ is
$$
\pi_1(\cL) \cong \pi_1(S^1\times \cB G_\tau) \cong \bZ \times G_\tau. 
$$
In particular, it is abelian, so it is isomorphic to its abelianization $H_1(\cL;\bZ)$.

If $(\tau,\si)\in F(\Sigma)$ then there is an inclusion $\iota^{(\tau,\si)}: \cS_\tau \hookrightarrow \cX_\si =[\bC^3/G_\si]$ which 
induces
$$
\iota^{(\tau,\si)}_*: \pi_1(\cS_\tau) \cong \bZ\times G_\tau \to \pi_1(\cX_\si) \cong G_\si. 
$$

\subsection{Moduli spaces of stable maps to $(\cX,\cL)$}
In \cite{KL01}, Katz-Liu introduced stable maps to a symplectic manifold with Lagrangian boundary conditions at all genera; the domain
of such a map is a prestable bordered Riemann surface, i.e. a smooth or nodal bordered Riemann surface. (See also \cite{Liu02}, \cite{FOOO}.) In \cite[Section 2]{CP}, Cho-Poddar
define stable maps to a symplectic {\em orbifold} $\cX$ with Lagrangian boundary conditions, under the assumption that the Lagrangian sub-orbifold $\cL$ does not contain any stacky points (so that
$\cL$ is indeed a smooth manifold); the domain of such a map is a prestable bordered {\em orbifold} Riemann surface in the sense of \cite[Section 2]{CP}, i.e. a smooth  or nodal bordered orbifold Riemann surface, where a stacky point is either an interior marked point or an interior node.

In general, $\cL$ is a sub-orbifold which contains stacky points. To obtain compactness of the moduli spaces when $\cX$ and $\cL$ are compact, one needs to allow orbifold structures
at boundary marked points and boundary nodes. In the present paper $\cL$ may contain stacky points, but we do not need to allow orbifold structures on the boundary of the domain, for the following two reasons.
\begin{enumerate}
\item[(i)] Our enumerative problem only requires interior insertions, so we do not need to introduce any boundary marked points.
\item[(ii)] In our case $\cX$ and $\cL$ are non-compact and we will define and compute open GW invariants by torus localization on moduli space of stable maps $\cX$ with boundaries in $\cL$. If a stable map represents a torus fixed point in the moduli space then any node in the domain must be mapped to a torus fixed (scheme or stacky) point in $\cX$, but $\cL$ does not contain any torus fixed point, so the domain does not contain any boundary nodes.
\end{enumerate}

Let $(\Si,x_1,\ldots, x_n)$ be a prestable bordered orbifold Riemann surface
with $n$ interior marked point. Then the coarse moduli space
$(\bar{\Si},\bar{x}_1,\ldots, \bar{x}_n)$ is a prestable bordered Riemann surface
with $n$ interior marked points, defined in \cite[Section 3.6]{KL01} and \cite[Section 3.2]{Liu02}.
We define the topological type $(g,h)$ of $\Si$
to be the topological type of $\bar{\Si}$ (see \cite[Section 3.2]{Liu02}).
 
Let $(\Si,\bSi)$ be a prestable bordered orbifold Riemann surface of type $(g,h)$,
and let $\bSi = R_1\cup \cdots \cup R_h$ be union of connected components. Each connected
component is a circle which contains no orbifold points. 
A (bordered) {\em prestable map} to the pair $(\cX,\cL)$ is a map $u:(\Si,\partial \Si)\to (\cX,\cL)$
where $\Si$ is a prestable bordered orbifold Riemann surface, 
such that $u\circ \nu:\hat{\Si}\to \cX$ is holomorphic, where $\nu: \hat{\Si}\to \Si$ is the normalization
(so $\hat{\Si}$ is a possibly disconnected smooth bordered orbifold Riemann surface); a prestable map to $(\cX,\cL)$ is {\em stable} if its automorphism group is finite.  The topological type of a stable map $u$ is given by 
the degree $\beta'= \bar{u}_*[\bar{\Si}]\in H_2(X,L;\bZ)$ (where $X$ and $L$ are the coarse moduli spaces of $\cX$ and $\cL$ respectively, and
$\bar{u}:\bar{\Si}\to X$ is the map between coarse moduli spaces)
and $\bar \mu_i=u_*[R_i] = (\mu_i, \lambda_i)\in H_1(\cL;\bZ)\cong \bZ\times G_\tau $ (where $\mu_i\in \bZ$ is the winding number and $\lambda_i\in G_\tau$ is the monodromy).  
Given $\beta'\in H_2(X,L;\bZ)$ and
$$\vmu= ( (\mu_1,\lambda_1),\dots, (\mu_h,\lambda_h)) \in H_1(\cL;\bZ)^h.$$
Let $\Mbar_{(g,h),n}(\cX,\cL\mid \beta',\vmu)$ be the moduli space of stable
maps of type $(g,h)$, degree $\beta'$, winding numbers $\mu_i\in\bZ$ and monodromies $\lambda_i\in G_\tau$,
with $n$ interior marked points.

\subsection{The tangent-obstruction complex and the virtual dimension}
Similar to \cite[Section 4.2]{KL01}, the tangent space $\cT_\xi^1$ and the obstruction space  $\cT_\xi^2$ at a moduli point
$$
\xi=[u:( (\Si, x_1,\ldots,x_n), \bSi)\to (\cX,\cL)]
\in \Mbar_{(g,h),n}(\cX,\cL\mid \beta',\vmu)
$$
fit into the following exact sequence of real vector spaces:
\begin{equation}\label{eqn:tangent-obstruction}
\begin{aligned}
0 \to & \Aut((\Si,x_1,\ldots,x_n),\bSi)\to H^0(\Si,\bSi,u^*T_\cX, (u|_{\bSi})^*T_\cL)\to \cT^1_\xi \\
  \to &  \Def((\Si,x_1,\ldots, x_n),\bSi)\to H^1(\Si,\bSi, u^*T_\cX, (u|_{\bSi})^* T_\cL)\to \cT^2_\xi,
\end{aligned}
\end{equation}
where 
\begin{itemize}
\item $\Aut((\Si,x_1,\ldots, x_n),\bSi)$ is the space of infinitesimal automorphism of the domain $((\Si,x_1,\ldots, x_n),\bSi)$
and is equal to $H^0(\Si,\bSi, T_{\Si}(-\sum_{j=1}^nx_j), T_{\bSi})$ when $\Si$ is a smooth bordered orbifold Riemann surface;
\item $\Def((\Si,x_1,\ldots, x_n),\bSi)$ is the space of infinitesimal deformations of the domain, and is equal to 
$H^1(\Si,\bSi,T_{\Si}(-\sum_{j=1}^n x_i), T_{\bSi})$ when $\Si$ is a smooth bordered orbifold Riemann surface;
\item  $H^0(\Si,\bSi,u^*T_{\cX}, (u|_{\bSi})^*T_{\cL})$ is the space of infinitesimal deformation of the map 
for a fixed domain; 
\item $H^1(\Si,\bSi,u^*T_{\cX},(u|_{\bSi})^*T_{\cL})$ is the space of obstructions
to deforming the map for a fixed domain. 
\end{itemize}
Globally on the moduli space $\Mbar_{(g,h),n}(\cX,\cL,\beta',\vmu)$, there is an exact sequence of sheaves
\begin{equation}\label{eqn:B}
0 \to B_1 \to B_2 \to \cT^1 \to B_4 \to B_5 \to \cT^2 \to 0
\end{equation}
whose fiber at the moduli point $\xi$ is \eqref{eqn:tangent-obstruction}.

Let $\fM_{(g,h),n}$ be the moduli of prestable bordered orbifold Riemann surfaces of type $(g,h)$ with $n$ interior marked point. Then $\fM_{(g,h),n}$ is a differentiable stack (with corners) of real dimension 
\begin{equation}\label{eqn:domain}
3(2g-2+h) + 2n = \dim_\bR  \Def((\Si,x_1,\ldots, x_n),\bSi) - \dim_\bR \Aut((\Si,x_1,\ldots, x_n),\bSi). 
\end{equation}

There are evaluation maps (at interior marked points)
$$
\ev_j:\Mbar_{(g,h),n}(\cX,\cL\mid \beta',\vmu) \to \cI\cX,\quad j=1,\ldots,n.
$$
Given $\vec{v}=(v_1,\ldots, v_n)$, where $v_1,\ldots, v_n\in \BoxS$, define
$$
\Mbar_{(g,h),\vec{v}}(\cX,\cL\mid\beta',\vmu) := \bigcap_{j=1}^n \ev_j^{-1}(\cX_{v_j}).
$$
Suppose that $\xi \in \Mbar_{(g,h),\vec{v}}(\cX,\cL\mid\beta',\vmu)$. By the Riemann-Roch theorem for prestable bordered orbifold Riemann
 surface (which can be derived by combining the proof of the Riemann-Roch theorem for prestable bordered Riemann surfaces and  
 prestable  orbifold closed Riemann surfaces), 
\begin{equation}\label{eqn:map}
\dim_\bR H^0(\Si,\bSi,u^*T_{\cX}, (u|_{\bSi})^*T_{\cL}) -\dim_\bR H^1(\Si,\bSi, u^*T_{\cX}, (u|_{\bSi})^*T_{\cL}) =
3(2-2g-h) - 2\sum_{j=1}^n \age(v_j).
\end{equation}
The above Equation \eqref{eqn:map} is the relative virtual dimension of 
$\Mbar_{(g,h),\vec{v}}(\cX,\cL\mid\beta',\vmu)\rightarrow \fM_{(g,h),n}$ which sends a stable map to its domain.
 The virtual (real) dimension of $\Mbar_{{ (g,h)},\vec{v}}(\cX,\cL\mid \beta',\vmu)$ is equal to 
$$
\dim_\bR \cT^1_\xi - \dim_\bR \cT^2_\xi = 2 \sum_{j=1}^n(1-\age(v_j)),
$$
where  $\age(v_j)\in \{0,1,2\}$.

\subsection{Torus action and equivariant invariants} 
Let $\bT'_\bR\cong U(1)^2$ be the maximal compact subgroup of $\bT'\cong (\bC^*)^2$. For any $t \in \bT'_\bR$, the map
$\phi_t: \cX \to \cX$ given by $x\mapsto t\cdot x$ is an automorphism of the smooth toric DM stack $\cX$, and $\phi_t(\cL) =\cL$,
so $\bT'_\bR$ acts on the moduli spaces $\Mbar_{(g,h),n}(\cX,\cL \mid\beta',\vmu)$; here we use the notion of group actions on stacks 
in \cite{Rom}. Let $F \subset \Mbar_{(g,h),n}(\cX,\cL\mid \beta',\vmu)$ be
the substack of $\bT'_\bR$ fixed points. The restriction of the exact sequence \eqref{eqn:B} to the substack $F$ is the direct sum of  two exact sequences
\begin{equation}\label{eqn:f}
0 \to B_1^f \to B_2^f\to \cT^{1,f} \to B_4^f\to B_5^f \to \cT^{2,f}\to 0,
\end{equation}
\begin{equation}\label{eqn:m}
0 \to B_1^m \to B_2^m \to \cT^{1,m} \to B_4^m\to B_5^m \to \cT^{2,m}\to 0,
\end{equation}
where \eqref{eqn:f} is the subcomplex fixed by the torus action.
The virtual tangent bundle $\cT^\vir_F$ of $F$ is 
$$
\cT^\vir_F =\cT^{1,f}-\cT^{2,f}
$$
whose ranks can be different on different connected components of $F$. We will see that each connected component of $F$ is a compact orbifold, and that
$\cT^\vir_F$ is equal to the tangent bundle $\cT_F$  of $F$. So
$$
[F]^\vir = [F].
$$
The virtual normal bundle $N^\vir$ of $F$ in $\Mbar_{(g,h),n}(\cX,\cL\mid \beta', \vmu)$ is
$$
N^\vir =\cT^{1,m}-\cT^{2,m}.
$$
Given $\gamma_1,\ldots, \gamma_n \in H^*_{\bT',\orb}(\cX;\bQ)=H^*_{\bT'_\bR,\orb}(\cX,\bQ)$, we define
\begin{equation}\label{eqn:local}
\langle \gamma_1,\ldots, \gamma_n\rangle^{\cX,\cL,\bT_\bR'}_{g,\beta',\vmu}:=
\int_{[F]^{\vir}} \frac{\prod_{j=1}^n (\ev_j^*\gamma_i)|_F}{e_{\bT_\bR'}(N^\vir)} \in \cQ_{ \bT'_\bR}
\end{equation}
where $\cQ_{\bT_\bR'}$ is the fractional field of $H^*_{\bT_\bR'}(\pt;\bQ)$, and
$$
\frac{1}{e_{\bT_\bR'}(N^\vir_F)}  =\frac{e_{\bT_\bR'}(\cT^{2,m})}{e_{\bT_\bR'}(\cT^{1,m}) }=\frac{ e_{\bT_\bR'}(B_1^m) e_{\bT_\bR'}(B_5^m)}{e_{\bT_\bR'}(B_2^m)e_{\bT_\bR'}(B_4^m)}.
$$
More precisely, the definition  \eqref{eqn:local} also requires an orientation on  the virtual tangent bundle $\cT^1 -\cT^2$, which we will specify later.

\subsection{Tangent weights: the 3-torus, the Calabi-Yau 2-torus, and the framing 1-torus}\label{sec:weights}
Let $\fo_\tau \cong \bC^*\times \cB G_\tau$ be the unique 1-dimensional $\bT$-orbit 
which intersects the Aganagic-Vafa A-brane $\cL$, where $\tau\in \Si(2)$, as before.
Let $\fl_\tau$ be the closure of $\fo_\tau$, and let $\ell_\tau$ be the coarse moduli of $\fl_\tau$. Then $\ell_\tau$ is either $\bP^1$ or $\bC$.
\begin{definition}
We say $\cL$ is an {\em inner} brane if $\ell_\tau\cong \bP^1$; we say $\cL$ is an {\em outer} brane if $\ell_\tau \cong \bC$.
\end{definition}
If $\cL$ is an outer brane, let $\si\in \Si(3)$ be the unique 3-cone such that $(\tau,\si)\in F(\Si)$; if $\cL$ is an inner brane, we choose
$\si\in \Si(3)$ such that $(\tau,\si)\in F(\Si)$ and let $\si_- \in \Si(3)$ denote the other choice, so that
$\fp_\si$ and $\fp_{\si'}$ are the two torus fixed points in $\fl_\tau$. For inner branes, we also denote $\si_+=\si$.

By permuting $b_1,\ldots, b_r$ if necessary, we may assume that $I'_{\si}=\{1,2,3\}$, and  $(\tau_1,\sigma)= (\tau,\sigma)$, 
$(\tau_2,\sigma)$ and $(\tau_3,\sigma)$ are  three flags in the toric graph in the counter-clockwise direction such that
$$
I'_{\tau_1}= \{2, 3\},\quad I'_{\tau_2}=\{3, 1\},\quad I'_{\tau_3}=\{1,2\}.
$$
Here we fixed an orientation of $\bR^2$. If $\cL$ is an inner brane, we assume in addition $I'_{\si_-}=\{ 2,3, 4\}$. 

Recall from Section \ref{sec:T-subvariety} that for any flag $(\tau,\si)\in F(\Si)$, $\chi_{(\tau,\si)}\in \Hom(G_\si,\bC^*)$ is the character of the 1-dimensional  $G_\si$ representation $T_{\fp_\si}\fl_\tau$. Let 
$$
\fr:= r(\tau,\si) = |G_\si/G_\tau|,\quad \fm:= |G_\tau/K|.
$$
Then we have the following two short exact sequences of finite abelian groups
$$
1\to G_\tau \lra G_\si \stackrel{\chi_{(\tau,\si)} }{\lra} \bmu_{\fr}\to 1,\quad
1\to K \lra G_\tau \stackrel{\chi_{(\tau_3,\si)}}{\lra} \bmu_{\fm}\to 1.
$$
Note that for any $\lambda\in G_\tau$, $\chi_{(\tau,\si)}(\lambda)=1$, and $\chi_{(\tau_2,\si)}(\lambda)\chi_{(\tau_3,\si)}(\lambda)=1$. Let $\bar{\lambda}$ denote the unique element in 
$\{0,1,\ldots,\fm-1\}$ such that 
$$
\chi_3(\lambda)=e^{2\pi\sqrt{-1}\bar{\lambda}/\fm}.
$$
Let $\wu_3\in M$ be defined as in Section \ref{sec:rigid}, so that $\langle \wu_3, \bar{b}_i\rangle =1$.
We may choose a $\bZ$-basis $\{ \wv_1,\wv_2, \wv_3\}$ of $\bar{N}$  such that $ \langle \wu_3,\wv_i\rangle =\delta_{i,3}$, and 
$$
\bar{b}_1 =\fr \wv_1 - \fs \wv_2 + \wv_3,\quad \bar{b}_2 =\fm \wv_2 +\wv_3,\quad \bar{b}_3=\wv_3.
$$
Moreover, the choice  $(\wv_1,\wv_2,\wv_3)$ is unique if we require $\fs \in \{0,1,\ldots,\fr-1\}$. Let $\{ \wu_1,\wu_2,\wu_3\}$
be the $\bZ$-basis of $M$ which is dual to the $\bZ$-basis $\{\wv_1,\wv_2,\wv_2\}$ of $\bar{N}$. 
Let $\{\ww_1,\ww_2,\ww_3\}$ be the $\bQ$-basis of $M_\bQ$ which is dual
to the $\bQ$-basis $\{ \bar{b}_1,\bar{b}_2, \bar{b}_3\}$ of $N_\bQ = N\otimes_{\bZ}\bQ$. Then 
 \[
\ww_1=\frac{1}{\fr}\wu_1,\ \ww_2=\frac{\fs}{\fr \fm}\wu_1+\frac{1}{\fm}\wu_2,\
\ww_3=-\frac{\fs+\fm}{\fr \fm}\wu_1-\frac{1}{\fm}\wu_2+\wu_3.
\]
Moreover, for $i\in \{1,2,3\}$,
$$
\ww_i= e_{\cT} (T_{\fp_\si}\fl_{\tau_i}) = e_{\cT}(\mathcal O_{\cX}(\mathcal D_i))\Big|_{\fp_\si}.
$$
The inclusion $\bT'\subset \bT$ induces the following surjective ring homomorphism
\begin{equation}\label{eqn:CYtorus}
H^*(\cB \bT;R)=R[\wu_1,\wu_2,\wu_3]\lra H^*(\cB \bT';R)=R[\wu'_1,\wu'_2],\quad
\wu_1\mapsto \wu_1',\ \wu_2 \mapsto \wu_2',\  \wu_3 \mapsto 0 
\end{equation}
where $R=\bZ$ or $\bQ$.

Given a {\em framing} which is an integer $f\in \bZ$, let $\bT_f \subset \bT'$ be the kernel of the 
character $\wu_2'-f\wu'_1\in \Hom(\bT';\bC^*)$. Then $\bT_f\cong \bC^*$ is a 1-dimensional subtorus 
of the Calabi-Yau torus $\bT'$. The inclusion $\bT_f\subset \bT'$ induces a surjective ring homomorphism 
\begin{equation}\label{eqn:framing}
H^*(\cB \bT';R)=R[\wu'_1,\wu'_2] \lra H^*(\cB \bT_f;R)=R[\wu],\quad
\wu'_1\mapsto \wu,\ \wu'_2 \mapsto f\wu 
\end{equation}
where $R=\bZ$ or $\bQ$. For $i=1,2,3$, let $\w'_i$ denote the image of $\ww_i$ under the ring homomorphism
\eqref{eqn:CYtorus}, and let
$\ww^f_i$ denote the image of $\ww'_i$ under the ring homomorphism \eqref{eqn:framing}. Then
\begin{equation}\label{eqn:w-plus}
\ww'_1=\frac{1}{\fr}\wu'_1,\ \ww'_2=\frac{\fs}{\fr \fm}\wu'_1+\frac{1}{\fm}\wu'_2,\
\ww'_3=-\frac{\fs+\fm}{\fr \fm}\wu'_1-\frac{1}{\fm}\wu'_2 =-\ww'_1-\ww'_2 \in H^2(\cB \bT')=\bQ \wu'_1 \oplus \bQ\wu'_2,
\end{equation}
and $\ww^f_i = w_i \wu$, where $w_i\in \bQ$ are given by 
$$
w_1=\frac{1}{\fr},\quad   w_2= \frac{\fs + \fr f}{\fr \fm},\quad 
w_3 = -w_1-w_2 = \frac{-\fm-\fs-\fr f}{\fr\fm}.
$$

\subsection{Disk factor as equivariant open GW invariants}

A framed Aganagic-Vafa Lagrangian brane is a pair $(\cL,f)$ where $\cL$ is a Aganagic-Vafa brane together with a choice
of a flag $(\tau,\si)\in F(\Si)$ such that $\fo_\tau$ is the unique 1-dimensional orbit intersecting $\cL$ and a choice
of framing $f\in \bZ$. Given a framed Aganagic-Vafa Lagrangian brane $(\cL,f)$, we choose an isomorphism 
$\pi_1(\cL) \cong \bZ\times G_\tau$ such that if $h=\iota_*^{(\tau,\si)}(d_0,\lambda)$ (where
$\iota_*^{(\tau,\si)}$ is defined in Section \ref{sec:AV-brane}) then
$$
\chi_{(\tau_1,\si)}(h) = e^{2\pi\sqrt{-1}d_0 w_1},\quad
\chi_{(\tau_2,\si)}(h) = e^{2\pi\sqrt{-1}d_0 (w_2-\frac{\bar{\lambda}}{\fm})},\quad
\chi_{(\tau_3,\si)}(h) = e^{2\pi\sqrt{-1}d_0 (w_3 +\frac{\bar{\lambda}}{\fm})}. 
$$
Let $\ell_\tau$ be the coarse moduli of $\fl_\tau$, as before. Let $p_\si\in \ell_\tau$ be the coarse moduli of $\fp_\si\cong \cB G_\si$, 
and let $S_\tau := L\cap \ell_\tau \cong S^1$ be the coarse moduli of $\cS_\tau =\cL\cap \fl_\tau \cong S^1\times \cB G_\tau$.

\subsubsection{$(\cL,f)$ is a framed outer brane} In this case $\ell_\tau =\bC$. Let $D\subset \ell_\tau$ be the disk which contains
$p_\si$ with boundary $S_\tau$, oriented by the complex structure on $\ell_\tau$, and let $b=[D]\in H_2(X,L;\bZ)$. Given
$(d_0,\lambda)\in H_1(\cL;\bZ)\cong \bZ\times G_\tau$, where $d_0>0$, define
$$
\Mbar(d_0,\lambda):= \Mbar_{(0,1),1}(\cX,\cL\mid d_0 b, (d_0,\lambda)).
$$
The virtual real dimension of $\Mbar(d_0,\lambda)$ is  $2(1-\age(h(d_0,\lambda)))$,
where $h(d_0,\lambda):= \iota_*^{(\tau,\si)}(d_0,\lambda)\in G_\sigma$. Define the disk factor
$$
D_{d_0,\lambda}:= \langle \one_{h(d_0,\lambda)}\rangle_{0, d_0 b, (d_0,\lambda)}^{\cX,\cL,\bT'_\bR}
$$
which is a rational function in $\ww_1',\ww_2'$, homogeneous of degree $\age(h(d_0,\lambda))-1$.
The disk factor is computed in \cite{BC11} when $G_\si$ is cyclic,
and in \cite[Section 3.3]{Ro14} for general $G_\si$. In our notation,
the formula in \cite[Section 3.3]{Ro14}
says\footnote{The disk function in \cite[Section 3.3]{Ro14} and our disk factor are the same when $h(d_0,\lambda)\neq 0$. When
$h(d_0,\lambda)=0$, the disk function is $\langle \ \rangle^{\cX,\cL}_{\ldots}$ (no insertion), while the disk factor is
$\langle 1 \rangle^{\cX,\cL}_{\ldots}$ (one insertion of $1$), so
there is an additional factor of $(\frac{\fr}{\w_1})^{\delta_{0,h(d_0,\lambda)}}$ in the disk function in \cite[Section 3.3]{Ro14}.}
\begin{equation}\label{eqn:disk}
\begin{aligned}
D_{d_0,\lambda}=  (\frac{\fr\w'_1}{d_0})^{\age(h(d_0,\lambda))-1} \frac{1}{d_0 |G_\tau|}
\cdot \frac{ \prod_{a=1}^{ \lfloor \frac{d_0}{\fr}\rfloor+ \age(h(d_0,\lambda)) -1}\left(
\frac{d_0 \w'_2}{\fr\w'_1} + a-c_2\left(h(d_0,\lambda)\right) \right) }{\lfloor \frac{d_0}{\fr}\rfloor!}
\end{aligned}
\end{equation}
where $c_i(\cdot)\in \bQ\cap [0,1)$ is defined in Section \ref{sec:CR}. More explicitly,
$$
c_2(h(d_0,\lambda)) = \langle d_0 w_2 -\frac{\bar{\lambda}}{\fm}\rangle.
$$

\subsubsection{$(\cL,f)$ is a framed inner brane} In this case $\ell_\tau \cong \bP^1$. 
It contains two torus fixed points
$p_+ = p_\si$ and $p_-=p_{\si_-}$, where $\si_- \in \Si(3)$. The circle $S_\tau$ is the intersection of two disks $D_+$ and $D_-$ which contain $p_+$ and $p_-$, respectively. Let
$$
b=[D]\in H_2(X,L;\bZ),\quad \alpha =[\ell_\tau]\in H_2(X;\bZ).
$$
Then $[D_-]=\alpha-b\in H_2(X,L;\bZ)$. Given $(d_0,\lambda)\in H_1(\cL;\bZ)\cong \bZ\times G_\tau$, where $d_0\neq 0$, we define
$$
\Mbar(d_0,\lambda) :=
\begin{cases}
\Mbar_{(0,1),1}(\cX,\cL\mid d_0 b, (d_0,\lambda)), & d_0>0,\\
\Mbar_{(0,1),1}(\cX,\cL\mid -d_0(\alpha-b), (d_0, \lambda)), & d_0<0.
\end{cases}
$$
Then
$$
\textup{virtual dimension of }\Mbar(d_0,\lambda)= 
\begin{cases} 1 -\age(h^+(d_0,\lambda)), & d_0 >0,\\ 1-\age(h^-(d_0,\lambda)), & d_0<0, \end{cases}
$$
where $h^{\pm}(d_0,\lambda) = \iota_*^{(\tau,\si_\pm )}(d_0,\lambda)\in  G_{\si_\pm}$. Define
$$
D_{d_0,\lambda}:= \begin{cases}
\langle 1_{h^+(d_0,\lambda)}\rangle_{0,d_0 b, (d_0,\lambda)}^{\cX,\cL,\bT'_\bR}, & d_0 >0,\\
& \\
\langle 1_{h^-(d_0,\lambda)}\rangle_{0,-d_0 (\alpha-b), (d_0,\lambda)}^{\cX, \cL,\bT'_\bR}, & d_0 <0.
\end{cases}.
$$
Then $D_{d_0,\lambda}$ is a rational function in $\wu'_1,\wu'_2$, homogeneous
of degree $\age(h^\pm(d_0,\lambda))-1$ if $\pm d_0>0$.

More precisely, the disk factor $D_{d_0,\lambda}$ is defined up to a sign depending on choice of orientation of
$\Mbar(d_0,\lambda)$, which will be clarified in Section \ref{sec:relative} by relative GW invariants.

\subsection{Normal bundle to $\fl_\tau$} \label{degenerate}
Let $\cL$ be an inner brane, so that $\fl_\tau$ is a proper smooth toric DM curve. Let $\hat{\fl}_\tau$ be the
image of $\fl_\tau$ under the morphism $\cX\to \cX^\rig$. We have
$$
\fl_\tau \lra \hat{\fl}_\tau \lra \fl_\tau^\rig \lra \ell_\tau \cong \bP^1. 
$$
where $\fl_\tau\to \hat{\fl}_\tau$ is a $K$-banded gerbe, $\hat{\fl}_\tau \to \fl_\tau^\rig$ is a $\bmu_m$-banded gerbe, 
and $\fl_\tau \to \fl_\tau^\rig$ is a $G_\tau$-banded gerbe. The normal bundle $\fl_\tau$ in $\cX$ is a direct sum of two $\bT$-equivariant
line bundles over $\fl_\tau$:
$$
N_{\fl_\tau/\cX} =L_2 \oplus L_3,
$$
where $L_2=\cO_{\cX}(\cD_2)\bigr|_{\fl_\tau}$ and $L_3 =\cO_{\cX}(\cD_3)\bigr|_{\fl_\tau}$. The total space of $N_{\fl_\tau/\cX}$ is a smooth  toric DM stack which is
isomophic to the open substack $\cY:= \cX_{\si}\cup \cX_{\si_-}$ of $\cX$. Let
$\hat{\cD}_i$ be the image of $\cD_i$ under $\cX\to \cX^\rig$. Then 
$$
N_{\hat{\fl}_\tau/\cX^\rig} = \hat{L}_2\oplus \hat{L}_3.
$$
where $\hat{L}_2 =\cO_{\cX^\rig}(\hat{\cD}_2)\bigr|_{\hat{\fl}_\tau}$ and $\hat{L}_3 =\cO_{\cX^\rig}(\hat{\cD}_3)\bigr|_{\hat{\fl}_\tau}$.
The total space of $N_{\hat{\fl}_\tau/{\cX^\rig}}$ is a toric orbifold which is isomorphic to the open substack $\cY^\rig=\cX^\rig_\si \cup \cX^\rig_{\si_-}$ of
the toric orbifold $\cX^\rig$.

Let $\Si_0$ be the simplicial fan in $N_\bR$ consisting of $\si$, $\si_-$ and their subcones. The stacky fan of 
$\cY^\rig$ is given by $(\Si_0, (\bar{b}_1, \bar{b}_2, \bar{b}_3, \bar{b}_4))$, where
$$
\bar{b}_1 =\fr \wv_1 -\fs \wv_2 + \wv_3,\quad
\bar{b}_2 = \fm \wv_2 +\wv_3,\quad
\bar{b}_3 = \wv_3,  \quad
\bar{b}_4 =-\fr_- \wv_1 + c\wv_2 +\wv_3,
$$
where $c$ is some integer. For inner branes, we also denote $\fr_+=\fr$. We have $\cY =[U/G_0]$, where
\begin{eqnarray*}
U&=&\{(Z_1,Z_2,Z_3,Z_4)\in \bC^4: (Z_1,Z_4)\neq (0,0)\},\\
G_0&=&\{ (\tit_1,\tit_2,\tit_3,\tit_4)\in (\bC^*)^4: \tit_1^{\fr} (\tit_4)^{-\fr_-}
= (\tit_1)^{-\fs}\tit_2^{\fm} \tit_4^c =\tit_1\tit_2\tit_3\tit_4=1\}. 
\end{eqnarray*}
We have a short exact sequence of abelian groups:
$$
1\to \bmu_\fm  \lra G_0 \stackrel{\chi_1\times \chi_4}{\lra} G_{\fr, \fr_-} \to 1,
$$
where $G_{\fr,\fr_-}=\{ (\tit_1,\tit_4)\in (\bC^*)^2: \tit_1^{\fr} (\tit_4)^{-\fr_-}=1\}$, and
$\chi_i(\tit_1,\tit_2, \tit_3, \tit_4)=\tit_i$. The subgroup $\bmu_{\fm}$ of  $G_0$ acts trivially
on  $V =\{ (Z_1,Z_2,Z_3,Z_4)\in U_{\cA}: Z_2=Z_3=0\}$, so the $G_0$-action on $V$ factors through a $G_{\fr,\fr_-}$-action on $V$, and 
$$
\hat{\fl}_\tau = [V/G_0],\quad \fl_{\tau}^\rig = [V/G_{\fr,\fr_-}] \cong \cF_{\fr,\fr_-}
$$
where $\cF_{\fr,\fr_-}$ denotes the football obtained by
gluing $[\bC/\bmu_{\fr}]$ and $[\bC/\bmu_{\fr-}]$ along $[\bC^*/\bmu_{\fr}] \cong [\bC^*/\bmu_{\fr_-}]\cong \bC^*$. The two torus fixed points in 
$\fl_{\tau}^\rig$ are
$$
\fp_x = [ (\{(0,0,0)\} \times \bC^*)/G_{\fr,\fr_-}] \cong \cB \bmu_\fr,\quad
\fp_y = [ (\bC^* \times \{(0,0,0)\})/G_{\fr,\fr_-}] \cong \cB \bmu_{\fr_-},
$$
and $\fl_\tau^\rig- \{ \fp_x, \fp_y\} \cong \bC^*$. We have a surjective group homomorphism
$\bZ\oplus \bZ \to \Pic(\fl_\tau^\rig)$ sending $(n_x,n_y)$ to  $\cO_{\fl_\tau^\rig}(n_x \fp_x + n_y \fp_y)$; the kernel is 
$\bZ(\fr,-\fr_-)$.  

Let $\cO(-1)$ denote the tautological line bundle over $\cB \bC^*$ associated with the fundamental 
representation $\bC^* \to GL(1,\bC)$, $t\mapsto t$. 
Given a line bundle $L$ over a DM stack $\cZ$ and a positive integer $\fm$, 
let $\sqrt[\fm]{L/\cZ}$ denote the following fiber product (cf. \cite[Definition 2.2.6]{Ca07}): 
$$
\begin{CD}
\sqrt[\fm]{L/\cZ} = \cZ \times_{\cB\bC^*} \cB\bC^* & @>{p_2}>> & \cB\bC^*\\
@V{p_1}VV &  & @VV{\otimes \fm}V\\
\cZ & @>{\phi_L}>> & \cB\bC^*
\end{CD}
$$
where the morphism $\phi_L: \cZ \longrightarrow \cB\bC^*$ is defined by $L$ (so that $\phi_L^*\cO(-1)= L$),
and $\cB\bC^*\to \cB\bC^*$ is induced by the $\fm$-th power map from $\bC^*$ to itself. 
Then $p_1:\sqrt[\fm]{L/\cZ} \to \cZ$ is a $\bmu_\fm$-banded gerbe. Let
$\sqrt[\fm]{L}:=p_2^*\cO(-1)\in \Pic(\sqrt[\fm]{L/\cZ})$. Then $(\sqrt[\fm]{L})^{\otimes \fm}= p_1^*L$, i.e., 
$\sqrt[\fm]{L}$ is an $m$-th root of $p_1^*L$. 

It is straightforward to check that
\begin{itemize}
\item $\hat{\fl}_\tau$ is isomorphic to $\sqrt[\fm]{\cO_{\fl_\tau^\rig}(\fs\fp_x - c\fp_y)/\fl^\rig_\tau}$
as a $\bmu_\fm$-banded gerbe over $\fl_\tau^\rig \cong \cF_{\fr,\fr_-}$, and 
\item $\hat{L}_2 \cong \sqrt[\fm]{\cO_{\fl_\tau^\rig}(\fs\fp_x-c\fp_y)}$, 
$\hat{L}_3 = \hat{L}_2^{-1} \otimes p_1^*\cO_{\fl_\tau^\rig}(-\fp_x-\fp_y)$, where
$p_1: \hat{\fl}_\tau \to \fl_\tau^\rig$ and $\cO_{\fl_\tau^\rig}(-\fp_x-\fp_y)$ is the cotangent bundle of $\fl_\tau^\rig$. 
\end{itemize}
 
\subsection{Degeneration}
Let $\Si_1 =\{ \{0\}, \bR_{\geq 0} \wv_1, \bR_{\geq 0}(-\wv_1)\}$ be the complete 1-dimensional fan in $\bR\wv_1\cong\bR$, and let
$\Si_2=\{ \{0\}, \bR_{\geq 0} \wv_4 ,\bR_{\geq 0}(-\wv_4)\}$ be the complete 1-dimensional fan in $\bR\wv_4\cong \bR$. Then 
$X_{\Si_1}=X_{\Si_2}=\bP^1$. The stacky fan $(\Si_1, (\fr\wv_1,-\fr_-\wv_1))$ defines the 1-dimensional toric orbifold 
$\cF_{\fr,-\fr_-}$, and the stacky fan 
$$
\mathbf{\Si}_{\Box}= (\Si_1\times \Si_2, (b_1'=\fr\wv_1, b_2'=-\fr_-\wv_1, b_3'=\wv_4, b_4'= -\wv_4))
$$ 
defines the 2-dimensional toric orbifold $\cF_{\fr,\fr_-}\times \bP^1$. 
The 1-dimensional cones  in the fan $\Si_1\times \Si_2$ are
$\{ \rho_i = \bR_{\geq 0}b_i': 1\leq i\leq 4\}$. 
Let $\Si'$ be the fan obtained by adding a 1-dimensional cone $\rho_5 = \bR_{\geq 0} b_5'$ where
$b_5'=-\wv_1-\wv_4$.  Let $\cS'$ be the 2-dimensional toric orbifold defined by the stacky fan
$\mathbf{\Si}'=(\Si',(b_1',b_2',b_3',b_4',b_5'))$, and let $\fl_i'=\cV(\rho_i) \subset \cS'$ be 1-dimensional closed toric substack associated with
the ray $\rho_i$. The morphism  $\mathbf{\Si'}\to \mathbf{\Si}_{\Box}$ of stacky fans induces a morphism
$\nu: \cS' \to \cF_{\fr,\fr_-} \times \bP^1$ of toric orbifolds; $\nu$ contracts the divisor $\fl'_5$ to the torus fixed point 
$[0,1]\times [0,1] \cong \cB \bmu_{\fr_-}$ in $\cF_{\fr,\fr_-}\times \bP^1$. 
Let $p: \cF_{\fr,\fr_-}\times \bP^1\to \bP^1$ be the projection to the second factor. The composition
$\pi':= p\circ \nu$ is a flat morphism, and 
$$
(\pi')^{-1}([0,1]) =\fl'_3, \quad (\pi')^{-1}([1,0])=\fl'_4\cup \fl'_5, 
$$
where $\fl'_3\cong \cF_{\fr,\fr_-}$, $\fl'_4 \cong \cF_{\fr,1}$, and $\fl'_5\cong \cF_{1,\fr_-}$. 
The torus fixed points in $\cS'$ are
$$
\fp_x^0 = \fl_1'\cap \fl_3' \cong \cB\bmu_\fr,\ 
\fp_y^0 = \fl_2'\cap \fl_3'\cong \cB\bmu_{\fr_-},\  
\fp_x^\infty = \fl_1'\cap \fl'_4 \cong \cB \bmu_\fr,\ 
\fp_y^\infty = \fl_2'\cap \fl'_5 \cong \cB \bmu_{\fr_-},\ 
p_z =\fl'_4\cap \fl'_5,
$$
where $p_z$ is a scheme point.  

Given any $f\in \bZ$, define
$$
\hat{\cS}= \sqrt[\fm]{ \cO_{\cS'}(\fs\fl_1' -c\fl_2'+f\fl_5')/\cS'}
$$
which is a $\bmu_\fm$-banded gerbe over $\cS'$, and let $\hat{q}: \hat{\cS} \to \cS'=\cS^\rig$ 
be the morphism to the rigidification. 
Define $\pi:= \hat{q}\circ \pi': \cS\to \bP^1$, and let $\hat{\fl}_i\subset \cS$ be the divisor which corresponds to $\fl_i'\subset \cS'$ under 
$\hat{q}:\hat{\cS}\to \cS'$.
Then $\hat{q}_i:= \hat{q}|_{\fl_i}: \hat{\fl}_i\to \fl'_i =\hat{\fl}_i^\rig$ is a $\bmu_\fm$-banded gerbe. We have
$$
\pi^{-1}([0,1]) =\hat{\fl}_3\cong \hat{\fl}_\tau \quad  \pi^{-1}([1,0])=\hat{\fl}_4\cup \hat{\fl}_5. 
$$
Define $\tilde{L}_2, \tilde{L}_3 \in \Pic(\hat{\cS})$ by 
$$
\tilde{L}_2 := \sqrt[\fm]{ \cO_{\cS'}(\fs\fl_1' -c\fl_2'+f\fl_5')}
\quad \tilde{L}_3 := \tilde{L}_2^{-1}\otimes q^* \cO_{\cS'}(-\fl_1'-\fl_2').
$$
Then 
$$
\tilde{L}_2\big|_{\hat{\fl}_3}= \sqrt[\fm]{\cO_{\fl_3'}(\fs\fp^0_x -c\fp^0_y)} \cong \hat{L}_2, \quad 
\tilde{L}_3\big|_{\hat{\fl}_3}=\sqrt[\fm]{\cO_{\fl_3'}(-\fs\fp^0_x + c\fp^0_y)}\otimes \hat{q}_3^*\cO_{\fl_3'}(-\fp_x^0-\fp_y^0) \cong \hat{L}_3.
$$
For $i\in \{2,3\}$, define $\hat{L}_i^+ = \tilde{L}_i\big|_{\hat{\fl}_4}$ and $\hat{L}_i^- =\tilde{L}_i\big|_{\hat{\fl}_5}$. Then
$$
\begin{array}{ll}
\hat{L}_2^+= \sqrt[\fm]{\cO_{\fl_4'}(\fs\fp_x^\infty+fp_z )},
& \hat{L}_3^+ = \sqrt[\fm]{\cO_{\fl'_4}(-\fs\fp^\infty_x-fp_z)}\otimes \hat{q}_4^*\cO_{\fl'_4}(-\fp_x^\infty),\\
\hat{L}_2^-= \sqrt[\fm]{\cO_{\fl'_5}(-c \fp_y^\infty-fp_z)},
& \hat{L}_3^- = \sqrt[\fm]{\cO_{\fl'_5}(c \fp^\infty_y + fp_z)}\otimes \hat{q}_5^*\cO_{\fl'_5}(-\fp_y^\infty),
\end{array}
$$ 
To summarize:
\begin{itemize}
\item $\hat{\cS}$ is a degeneration from $\hat{\fl}_\tau$ to a nodal DM curve $\hat{\fl}_4 \cup \hat{\fl}_5$, and
$\cS'=\hat{\cS}^\rig$ is a degeneration from  the football $\fl_\tau^\rig \cong \cF_{\fr,\fr_-}$ 
to the nodal DM curve $\fl_4'\cup \fl_5'$. 
\item For $i=2,3$, the line bundle $\tilde{L}_i$ on $\hat{\cS}$ defines a degeneration of the line bundle $\hat{L}_i\to \hat{\fl}_\tau$
to a line bundle on $\hat{\fl}_4\cup \hat{\fl}_5$ which 
restricts to $\hat{L}_i^+$ on $\hat{\fl}_4$ and $\hat{L}_i^-$ on $\hat{\fl}_5$.
\item $\hat{\cS}$, $\tilde{L}_i$, $\hat{\fl}_4$, $\hat{\fl}_5$, and $\hat{L}^\pm_i$  depend on $f$, 
while $\cS'$, $\fl_4'$ and $\fl_5'$ do not. 
\end{itemize}
Moreover, the total space of $\tilde{L}_2\oplus \tilde{L}_3 \to \mathcal{\hat{S}} $ is a four dimensional toric orbifold $\hat{\cW}$ defined by 
a stacky fan 
\begin{equation}\label{eqn:Wrig}
(\Si_f, (\bar{b}_1, \bar{b}_2, \bar{b}_3, \bar{b}_4, \wv_4, -\wv_4, -\wv_1 -f\wv_2 -\wv_4))
\end{equation}
where $\Si_f$ is a simplicial fan in $\bigoplus_{i=1}^4 \bR \wv_i$, and $(\bar{b}_1,\ldots,-\wv_1-f\wv_2 -\wv_4)$ is a 7-tuple of vectors 
in $\bar{N}\oplus \bZ\wv_4 \cong \bZ^4$.  Let $\cW$ be the four dimensional smooth toric DM stack
defined by the the stacky fan
\begin{equation}
(\Si_f, (b_1, b_2, b_3, b_4, \wv_4, -\wv_4, -\tilde{\wv}_1 -f\tilde{\wv}_2-\wv_4)),
\end{equation}
where $\tilde{\wv}_1,\tilde{\wv}_2 \in N$ are lifts of $\wv_1,\wv_2 \in \bar{N}$, so that $(b_1,\ldots, -\tilde{\wv}_1- f\tilde{\wv}_2- \wv_4)$ is a 7-tuple of elements
in  $N\oplus \bZ \wv_4$.
Then $\cW$ is a $K$-banded gerbe over $\hat{\cW} =\cW^\rig$, and is a degeneration from 
the total space $\cY$ of $N_{\fl_\tau/\cX}$ to the total space $\cY_\infty$
of a direct sum $L_2^\infty\oplus L_3^\infty$ of line bundles over a nodal DM curve $\fl_+\cup \fl_-$;
$\fl_+\cup \fl_-$  is a $K$-banded gerbe over $\hat{\fl}_4\cup \hat{\fl}_5$ and a $G_\tau$-banded gerbe 
over $\fl_4'\cup \fl_5'$. For $i=2,3$, let $L_i^\pm = L_i^\infty|_{\fl_\pm}$. Then $L_i^\pm$ is the pullback of $\hat{L}_i^\pm$. Let $\fp_0\cong \cB G_\tau$ be the node which 
is the intersection of $\fl_+$ and $\fl_-$ and let $\fp_\pm$ be the unique torus fixed point in $\fl_\pm -\{ \fp_0\}$. Then 
$\fp_+ \cong \cB G_\si$ and $\fp_-\cong \cB G_{\si_-}$.  Define $\bT'$ weights
$$
\ww^\pm_1 := (c_1)_{\bT'}(T_{\fp_\pm}\fl^\pm),\quad
\ww^\pm_2 := (c_1)_{\bT'}(L_2)|_{\fp_\pm},\quad
\ww^\pm_3 := (c_1)_{\bT'}(L_3)|_{\fp_\pm} \in H^2(\cB\bT')=\bQ\wu_1'\oplus \bQ\wu_2'.
$$
Then $\ww^+_i=\ww'_i$ is given by Equation \eqref{eqn:w-plus}, $\ww_1^\pm +\ww_2^\pm +\ww_3^\pm =0$, and 
\begin{equation}\label{eqn:w-minus}
\w_1^-=-\frac{1}{\fr_-}\wu_1',\quad
\w_2^-=\frac{c}{\fr_-\fm}\wu_1' +\frac{1}{\fm}\wu_2',\quad 
\w_3^-=\frac{-c+\fm}{\fr_-\fm}\wu_1'-\frac{1}{\fm}\wu_2'.
\end{equation}
We also have
$$
(c_1)_{\bT'}(T_{\fp_0}\fl_\pm) = \mp \wu_1',\quad
(c_1)_{\bT'}(L_2^\pm)_{\fp_0}=\frac{\wu'_2-f\wu'_1}{m}= -(c_1)_{\bT'}(L_3^\pm)_{\fp_0}.
$$
The above weights are summarized in Figure 7 below.
\begin{figure}[h] \label{fig:degenerate}
\psfrag{w1+}{\small $\ww_1^+$}
\psfrag{w2+}{\small $\ww_2^+$}
\psfrag{w3+}{\small $\ww_3^+$}
\psfrag{-v1/r}{\small $-\wu_1'$}
\psfrag{v1/r}{\small $\wu_1'$}
\psfrag{v2-fv1}{\small $\frac{\wu'_2-f\wu'_1}{m}$}
\psfrag{-(v2-fv1)}{\small $-\frac{\wu'_2-f\wu'_1}{m}$}
\psfrag{p0}{\small $\fp_0$}
\psfrag{w1-}{\small $\ww_1^-$}
\psfrag{w2-}{\small $\ww_2^-$}
\psfrag{w3-}{\small $\ww_3^-$}
\psfrag{p+}{\small $\fp_+$}
\psfrag{p-}{\small $\fp_-$}
\begin{center}
\includegraphics[scale=0.5]{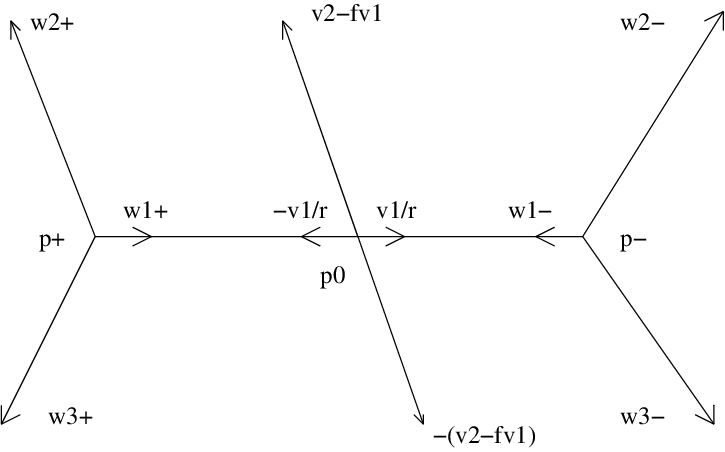}
\end{center}
\caption{Degenerated $N_{\fl_\tau/\cX}$ and the $\bT'$-weights}
\end{figure}

\subsection{Disk factor as equivariant relative GW invariants} \label{sec:relative} When $d_0>0$,
let $\Mbar_{0,1}(\fl_+/\fp_0,(d_0,\lambda))$ be the moduli space of relative maps to $(\fl_+,\fp_0)$ with the relative condition $(d_0,\lambda)$, where $\lambda \in G_\tau$
\cite{AF16}. A relative stable map to $(\fl_+,\fp_0)$ is a morphism to $\fl_+[m]$ which is the union of $\fl_+$ and a chain of
$m$ copies of $\bP^1\times \cB G_\lambda$. Let $\cM_{0,1}(\fl_+/\fp_0,(d_0,\lambda))\subset \Mbar_{0,1}(\fl_+/\fp_0,(d_0,\lambda))$ be the open substack where the target is $\fl_+[0]=\fl_+$. 
The tangent space $\cT^1_\xi$ and the obstruction space $\cT^2_\xi$ at a moduli point
$\xi=[u:(\cC,x,y)\to (\fl_+,\fp_0)]$ in $\cM_{0,1}(\fl_+/\fp_0,(d_0,\lambda))$ (where $u^{-1}(\fp_0)= d_0 y$ as Cartier divisors) fit into the following exact sequence of complex vector spaces:
\begin{equation}\label{eqn:relative-tangent-obstruction}
\begin{aligned}
0 \to & \Ext^0(\Omega_{\cC}(x+y), \cO_{\cC})\to H^0(\cC, u^*(T_{\fl_+}(-\fp_0)) \to \cT^1_\xi \\
  \to & \Ext^1(\Omega_{\cC}(x+y),\cO_{\cC}) \to H^1(\cC, u^*(T_{\fl_+}(-\fp_0)) \to \cT^2_\xi.
\end{aligned}
\end{equation}
Globally on $\cM_{0,1}(\fl_+/\fp_0,(d_0,\lambda))$, there is an exact sequence of sheaves
\begin{equation}\label{eqn:relative-B}
0 \to B_1 \to B_2 \to \cT^1 \to B_4 \to B_5 \to \cT^2 \to 0
\end{equation}
whose fiber at the moduli point $\xi$ is \eqref{eqn:relative-tangent-obstruction}.

Let $\pi:\cU_+\to \cM_{0,1}(\fl_+/\fp_0 ,(d_0,\lambda))$ be the universal domain curve and let $F_+:\cU_+\to \fl_+$ be the evaluation map. We define 
$$
V^+_{0,1}:=R^\bullet\pi_*F_+^*(L_2^+\oplus L_3^+)\in K_{\bT'}\left(\cM_{0,1}(\fl_+/\fp_0,(d_0,\lambda))\right)
$$
where $R^\bullet\pi_*$ is the K-theoretic push-forward.

For $d_0>0$, we define
\begin{align*}
D_{d_0,\lambda}&=\langle 1_{h^+(d_0,\lambda)} \rangle^{\cX,\cL,T_\bR'}_{0,d_0 b,(d_0,\lambda )}
\\
&=\int_{[\cM_{0,1}(\fl_+,\fp_0,(d_0,\lambda))]^\vir}
\ev^*(1_{h^+(d_0,\lambda)})e_{\bT'}(V^+_{0,1}).
\end{align*}

When $d_0<0$, let $\Mbar_{0,1}(\fl_-/\fp_0, (-d_0,\lambda^{-1}))$ be the moduli space of relative stable maps to $(\fl_-,\fp_0)$ with relative condition $(-d_0,\lambda^{-1})$, and let
$\cM_{0,1}(\fl_-/\fp_0,(-d_0,\lambda^{-1}))$ be the open substack where the target is $\fl_-$. Let $\pi:\cU_-\to \cM_{0,1}(\fl_-/\fp_0 ,(-d_0,\lambda^{-1}))$ be the universal domain curve
 and let $F_-:\cU_-\to \fl_-$ be the evaluation map. We define
$$
V^-_{0,1}=R^\bullet\pi_* F_-^*(L_2^-\oplus L_3^-) \in K_{\bT'}\left(\cM_{0,1}(\fl_-/\fp_0,(-d_0,\lambda^{-1}))\right),
$$
and define
\begin{align*}
D_{d_0,\lambda}&=\langle 1_{h^-(d_0,\lambda)} \rangle^{\cX,\cL,T'_\bR}_{0,-d_0 (\alpha-b), (d_0,\lambda))}
\\
&=\int_{[\cM_{0,1}(\fl_-,\fp_0,(-d_0,\lambda^{-1}))]^\vir}
\ev^*(1_{h^-(d_0,\lambda)})e_{\bT'}(V^-_{0,1}).
\end{align*}

Let $u:(\cC,x,y)\to \fl_+ $ be a relative stable map which
represents a point in $\cM$. Suppose that $u$ is fixed by the torus action.  Recall that $c_i:G_\si\to [0,1)\cap \bQ$ is defined by $\chi_i(k)=\exp(2\pi\sqrt{-1}c_i(k))$. In the computation below, let $k^\pm=h^\pm(d_0,\lambda)$. For $j=1, 2,3$, let $\ep_j= c_j(k^+)$. Then $\ep_1=\langle\frac{d_0}{\fr_+}\rangle$. We have the following weights
\begin{align*}
\ch_{\bT'}\bigl( H^0(\cC,u^*L_1^+) \bigr) = \sum_{a=0}^{\lfloor d_0w_1\rfloor} e^{a \frac{\wu_1'}{d_0}},&\quad
\ch_{\bT'}\bigl( H^1(\cC,u^*L_1^+)  \bigr) = 0, \\
\ch_{\bT'}\bigl( H^0(\cC,u^*L_2^+\otimes \cO_y) \bigr) =
\delta_{\langle d_0w_2- \epsilon_2\rangle, 0} e^{\frac{\wu_2'-f\wu_1'}{m}},&\quad
\ch_{\bT'}\bigl( H^1(\cC,u^*L_2^+\otimes \cO_y) \bigr) = 0,
\end{align*}
\begin{align*}
\ch_{\bT'}\bigl( H^0(\cC, u^*L^+_2) \bigr) &=
\begin{cases}
\displaystyle{ \sum_{a= -\lfloor d_0w_2-\ep_2\rfloor}^0 e^{\w'_2+(a-\ep_2) \frac{\wu'_1}{d_0 } } }, & f \geq 0,\\
0, & f<0,\\
\end{cases} \\
\ch_{\bT'}\bigl( H^1(\cC, u^*L^+_2) \bigr) &=
\begin{cases}
0, & f \geq 0,\\
\displaystyle{ \sum_{a=1}^{-\lfloor d_0w_2-\ep_2\rfloor-1} e^{\w'_2+(a-\ep_2)\frac{\wu'_1}{d_0 } } }, & f<0
\end{cases}\\
\ch_{\bT'}\bigl( H^0(\cC, u^*L^+_3) \bigr) &=
\begin{cases}
\displaystyle{ \sum_{a=-\lfloor d_0w_3-\ep_3\rfloor}^0 e^{\w'_3+(a-\ep_3) \frac{\wu'_1}{d_0 } } }, & f<0,\\
0, & f\geq 0,\\
\end{cases}\\
\ch_{\bT'}\bigl( H^1(\cC, u^*L^+_3) \bigr) &=
\begin{cases}
0, & f<0 ,\\
\displaystyle{ \sum_{a=1}^{-\lfloor d_0w_3 - \ep_3\rfloor-1} e^{\w'_3+(a-\ep_3)\frac{\wu'_1}{d_0 } } }, & f\geq 0.
\end{cases}
\end{align*}
and the following identities
\begin{align*}
\sum_{a=-\lfloor w_3d_0-\ep_3 \rfloor}^0 e^{\w'_3+(a-\ep_3) \frac{\wu'_1}{d_0 } }
&= \sum_{a=d_0w_1+\ep_2+\ep_3}^{-\lfloor d_0 w_2-\ep_2\rfloor -1 +
\delta_{\langle d_0 w_2-\ep_2 \rangle , 0}}  e^{-\w'_2 +(\ep_2-a)\frac{\wu'_1}{d_0 }} \\
\sum_{a=1}^{-\lfloor w_3 d_0 +1 - \ep_3 \rfloor} e^{\w'_3+(a-\ep_3)\frac{\wu'_1}{d_0 } }
&= \sum_{a=-\lfloor d_0w_2-\ep_2\rfloor +\delta_{\langle d_0w_2-\epsilon_2  \rangle}}^{\frac{d_0 }{s^+_1} +\ep_2+\ep_3-1}
e^{-\w'_2 +(\ep_2-a)\frac{\wu_1'}{d_0 }}\\
d_0w_1 + \ep_2 +\ep_3 &= \lfloor \frac{d_0 }{s^+_1}\rfloor + \age(k^+).
\end{align*}

The $\bT'$-equivariant Euler classes are
\begin{align*}
e_{\bT'}(B_1^m) &= 1,\\
e_{\bT'}(B_2^m) &= \lfloor d_0 w_1 \rfloor! (\frac{\wu'_1}{d_0})^{\lfloor d_0w_1 \rfloor}, \\
\frac{e_{\bT'}(B_5^m)}{e_{\bT'}(B_4^m)} &=
(-1)^{\lceil d_0 w_2-\ep_2 \rceil +\lfloor d_0w_1 \rfloor +\age(k^+)-1}
\prod_{a =1}^{\lfloor d_0w_1 \rfloor  + \age(k^+) -1}(\w'_2 + (a-\ep_2)\frac{\wu'_1}{d_0}).
\end{align*}
Since $|\Aut(f)| = d_0|G_\tau|$, by localization
\begin{align*}
  &D_{d_0,\lambda}=D(d_0,k^+,k^-) = \frac{1}{|\Aut(f)|} \frac{e_{\bT'}(B_1^m) e_{\bT'}(B_5^m)}{e_{\bT'}(B_2^m) e_{\bT'}(B_4^m)} \\
= &\frac{1}{d_0|G_\tau|}\frac{\prod_{a=1}^{\lfloor d_0 w_1\rfloor + \age(k^+) -1}(\w'_2 + (a-\ep_2)\frac{\wu'_1}{d_0})}
{\lfloor d_0 w_1\rfloor! (\frac{\wu'_1}{d_0})^{\lfloor d_0w_1\rfloor}} \cdot
(-1)^{\lceil d_0w_2-\ep_2\rceil +\lfloor d_0w_1\rfloor +\age(k^+)-1 } \\
=&(-1)^{\lceil d_0w_2-\ep_2\rceil +\lfloor d_0w_1\rfloor +\age(k^+)-1 } (\frac{\wu_1'}{d_0})^{\age(k^+)-1} \cdot
\frac{1}{d_0|G_\tau|}\frac{\prod_{a=1}^{\lfloor d_0w_1\rfloor +\age(k^+)-1}
(\frac{d_0 \w'_2}{s^+_1\w'_1} +a-\ep_2)} {\lfloor d_0w_1\rfloor !} \\
= &- (-1)^{\lfloor d_0w_3+\frac{\bar\lambda}{m} \rfloor } (\frac{\wu_1'}{d_0})^{\age(k^+)-1} \cdot
\frac{1}{d_0|G_\tau|}\frac{\prod_{a=1}^{\lfloor d_0w_1\rfloor +\age(k^+)-1}
(\frac{d_0 \w'_2}{s^+_1\w'_1} +a-\ep_2)} {\lfloor d_0w_1\rfloor !}
\end{align*}

Let $\wu:= \iota_f^* \wu'_1 \in H^2(\bT_f;\bZ)$. Then $H^*(\bT_f;\bQ)=\bQ[\wu]$ and $\iota_f^*\wu_2'=f\wu$.
Define $D_{d_0,\lambda,f}=\iota_{f}^* D_{d_0,\lambda}$. Hence when $d_0>0$
\begin{align*}
D_{d_0,\lambda,f} &=  -(-1)^{\lfloor d_0w_3+\frac{\bar \lambda}{m}\rfloor}(\frac{\wu}{d_0})^{\age(h^+(d_0,\lambda))-1} \cdot
\frac{1}{d_0|G_\tau|}\frac{\prod_{a=1}^{\lfloor d_0w_1\rfloor +\age(h^+(d_0,\lambda))-1}
(d_0 w_2 +a-c_2(h^+(d_0,\lambda))} {\lfloor d_0w_1\rfloor !}
\end{align*}
If $d_0 <0$, similar computation shows (notice $\overline{ \lambda^{-1}}=(1-\delta_{\bar\lambda,0})(m-\bar \lambda)\in \{0,\dots,m-1\}$)
\begin{align*}
D_{d_0 ,\lambda,f}
=&  -(-1)^{\lfloor d_0w_2^- + (1-\frac{\bar\lambda}{m}-\delta_{\bar \lambda,0})\rfloor} (\frac{\wu}{d_0 })^{\age(h^-(d_0,\lambda))-1} \cdot
\frac{1}{-d_0 |G_\tau|}\frac{\prod_{a=1}^{\lfloor d_0 w_1^- \rfloor+\age(h^-(d_0,\lambda))-1}(d_0w_3^--c_3(h^-(d_0,\lambda))+a)}{\lfloor d_0w_1^-\rfloor!}.
\end{align*}
If $\cL$ is an outer brane, it is the same as $d_0>0$.  Define
\begin{align*}
D_{d_0,\lambda,f} &=  -(-1)^{\lfloor d_0w_3+ \frac{\bar \lambda}{m}\rfloor}(\frac{\wu}{d_0})^{\age(h(d_0,\lambda))-1} \cdot
\frac{1}{d_0|G_\tau|}\frac{\prod_{a=1}^{\lfloor d_0w_1\rfloor +\age(h(d_0,\lambda))-1}
(d_0 w_2 +a-c_2(h(d_0,\lambda))} {\lfloor d_0w_1\rfloor !}.
\end{align*}

\subsection{Open-closed GW invariants and descendant GW invariants}
For any torus fixed point $\fp_\si$ of $\cX$, where $\si\in \Si(3)$, we have
$$
H^*_\orb(\fp_\si) =\bigoplus_{k \in G_\si} \bQ \one_k,\quad
H^*_{\orb,\bT'}(\fp_\si) =\bigoplus_{k \in G_\si}\bQ[\w'_1,\w'_2] \one_k.
$$
The inclusion $\iota_\si: \fp_\si \hookrightarrow \cX$ induces
$$
\iota_{\si *} : H^*_{\orb,\bT'}(\fp_\si) = H^*_{\bT'}(\cI\fp_\si) \to H^*_{\orb,\bT'}(\cX) = H^*_{\bT'}(\cI\cX).
$$
Define
\[
\phi_{\si,k} = \iota_{\si *} \one_k\in H^*_{\orb,\bT'}(\cX),\ 
\phi_{\si,k}^f=\iota_f^* \phi_{\si,k}\in H^*_{\orb,\bT_f}(\cX),
\]
\begin{proposition}[framed inner brane]\label{inner-psi}
Suppose that $(\cL,f)$ is a framed inner brane, and
$$
\vmu=((\mu_1,\lambda_j),\ldots, (\mu_h,\lambda_h)),
$$
where
$(\mu_j,\lambda_j)\in H_1(\cL;\bZ)\cong \bZ\times \cB G_\tau$. Let $J_\pm=\{ j\in \{1,\ldots,h\}: \pm \mu_j>0\}$, and 
let $k_j^\pm = h^\pm(\mu_j,\lambda_j)$.
Then
\begin{eqnarray*}
&& \langle \gamma_1,\ldots, \gamma_n\rangle_{g,\beta',\vmu}^{\cX, (\cL,f)}\\
&=&\prod_{j=1}^h D'_{\mu_j,\lambda_j,f} 
\cdot \int_{[\Mbar_{g,n+h}(\cX,\beta)]^\vir} \frac{ \big(\prod_{i=1}^n\ev_i^*\gamma_i
\prod_{j\in J_+}\ev_{n+j}^*\phi^f_{\si_+, (h^+(d_0,\lambda_j))^{-1}} \prod_{j\in J_-} \ev_{n+j}^*\phi^f_{\si_-, (h^-(d_0,\lambda_j))^{-1}} \big)}
{\prod_{j=1}^h\frac{\wu}{\mu_j}(\frac{\wu}{\mu_j} -\bar{\psi}_{n+j})}
\end{eqnarray*}
where
$$
\beta\in H_2(\cX),\quad
\beta'= \beta + \big(\sum_{j\in J_+} \mu_j\big)b -\big(\sum_{j\in J_-} \mu_j\big)(\alpha-b) \in H_2(\cX,\cL).
$$
\end{proposition}
\begin{proof}
There exists  $\beta\in H_2(\cX)$ such that
$$
\beta'=\beta +(\sum_{j\in J_+} \mu_j)b + \sum_{j\in J_-}(-\mu_j)(\alpha-b).
$$
Let $\langle k_j^+\rangle$ be the cyclic subgroup generated by $k_j^+$,
and let $r_j$ be the cardinality of $\langle k_j^+\rangle$ for $j\in J_+$ and
$r_j$ be the cardinality of $\langle k_j^-\rangle$ for $j\in J_-$.

We have
\begin{eqnarray*}
\Mbar_{(g,h),n}(\cX,\cL\mid \beta',\vmu)^{\bT'_\bR} &=& \bigcup_{\Ga \in   G_{g,n}(\cX,\cL\mid \beta',\vmu)} F_\Gamma \\
\Mbar_{g,n+h}(\cX, \beta)^{\bT'_\bR} = \Mbar_{g,n+h}(\cX, \beta)^{\bT'}  &=& \bigcup_{\hGa\in G_{g,n+h}(\cX,\beta) } F_{\hGa} 
\end{eqnarray*}
In the remaining part of this subsection, we use the following abbreviations:
\begin{eqnarray*}
&& \cM = \Mbar_{g,n}(\cX,\cL\mid \beta',\vmu),\quad \hcM = \Mbar_{g,n+h}(\cX, \beta), \\
&& \cM_j =
\begin{cases}
\Mbar_{(0,1),1}(\cX,\cL\mid \mu_j b, (\mu_j,\lambda)), & j\in J_+\\
\Mbar_{(0,1),1}(\cX, \cL\mid -\mu_j(\alpha-b), (\mu_j, \lambda)), & j\in J_-
\end{cases} \\
&& \cG = G_{g,n}(\cX,\cL\mid \beta',\vmu),\quad \hcG = G_{g,n+h}(\cX,\beta)\\
&& \bx=(x_1,\ldots, x_n),\quad \by=(y_1,\ldots, y_h).
\end{eqnarray*}

Given $u:(\Si,\bx,\bSi)\to (\cX, \cL)$ which represents a point  $\xi\in \cM^{\bT'_\bR}$, we have
$$
\Si = \cC \cup \bigcup_{j=1}^h D_j,
$$
where $\cC$ is an orbicurve of genus $g$, $x_1,\ldots, x_n \in \cC$,
$D_j = [\{ z\in \bC\mid |z|\leq 1\}/\bZ_{r_j}]$,  $\cC$ and $D_j$ intersect
at $y_j = B\bZ_{r_j}$. Let $u_j=u|_{D_j}$ and $\hu= u|_\cC$. Then
\begin{enumerate}
\item For $j=1,\ldots, h$, $u_j: (D_j,\partial D_j)\to (\cX, \cL)$ represents
a point in $\cM_j^{\bT'_\bR}$.
\item $\hu:(\cC,\bx,\by)\to \cX$ represents a point $\hat{\xi}\in \hcM^{\bT'}$, and
$\hu(\by_j) =[\fp_\pm, (k_j^\pm)^{-1}] \in \cI \fp_\pm \subset \cI\cX$ if $j\in J_\pm$.
  \end{enumerate}
Let $x_{n+j}=y_j$. Let $F_\Ga$ be the connected component of $\cM^{\bT'}$ associated
to the decorated graph $\Ga\in \cG$, and let $F_{\hGa}$ be the connected component
of $\hcM^{\bT'_\bR}$ associated to the decorated graph $\hat \Ga\in \cG$. Then
for any $\Ga\in \cG$ there exists $\hGa\in \hcG$ such that
$$
\ev_{n+j}(F_{\hGa}) = (\fp_\pm,(k_j^\pm)^{-1}) \in \cI \fp_\pm \subset \cI \cX
$$
if $j\in J_\pm$,  and $F_{\Ga}$ can be identified with $F_{\hGa}$ up to a finite morphism. More precisely,
\begin{eqnarray*}
[F_\Ga]^\vir &=& \prod_{j\in J_+} \frac{|G_{\si_+}|}{r_j|\Aut(u_j)|}
\prod_{j\in J_-} \frac{|G_{\si_-}|}{r_j|\Aut(u_j)|} [F_{\hGa}]^\vir  \\
&=&\prod_{j\in J_+} \frac{s^+_1}{r_j\mu_j}
\prod_{j\in J_-} \frac{s^-_1}{-r_j\mu_j} [F_{\hGa}]^\vir.
\end{eqnarray*}
We have
$$
\frac{1}{e_{\bT'_\bR}(N_\Ga^\vir)} =\frac{e_{\bT'_\bR}(B_1^m) e_{\bT'_\bR}(B_5^m)}{e_{\bT'_\bR}(B_2^m) e_{\bT'_\bR}(B_4^m)},\quad
\frac{1}{e_{\bT'_\bR}(N_\hGa^\vir)} =\frac{e_{\bT'_\bR}(\hB_1^m) e_{\bT'_\bR}(\hB_5^m) }{e_{\bT'_\bR}(\hB_2^m) e_{\bT'_\bR}(\hB_4^m)},
$$
where
\begin{eqnarray*}
e_{\bT'_\bR}(B_1^m)&=& e_{\bT'_\bR}(\hB_1^m), \\
e_{\bT'_\bR}(B_4^m) &=& e_{\bT'_\bR}(\hB_4^m)
\prod_{j\in J_+}(\frac{\wu_1'}{r_j \mu_j} -\frac{\bar{\psi}_j}{r_j})
\prod_{j\in J_-}(\frac{\wu_1'}{r_j \mu_j}-\frac{\bar{\psi}_j}{r_j})
\end{eqnarray*}
For $k=0,1$ and $j=1,\ldots, h$, let
$$
H^k(D_j) = H^k\bigl(D_j, \partial D_j, u_j^*T\cX, (u_j|_{\partial D_j})^*T\cL\bigr).
$$
Then there is a long exact sequence
\begin{eqnarray*}
&& 0 \to B_2 \to \hB_2 \oplus \bigoplus_{j=1}^h H^0(D_j) \to
\bigoplus_{j\in J_+} (T_{\fp_+}\cX)^{k_j^+}\oplus
\bigoplus_{j\in J_-} (T_{\fp_-}\cX)^{k_j^-} \\
&& \to B_5 \to \hB_5 \oplus \bigoplus_{j=1}^h H^1(D_j)\to 0,
\end{eqnarray*}
where $(T_{\fp_\pm}\cX)^{k_j^\pm}$ denote the $k_j^\pm$-invariant part of
$T_{\fp_\pm}\cX$. Note that
$$
(T_{\fp_\pm}\cX)^{k_j^\pm} = T_{(\fp_\pm, k_j^\pm)} \cI \cX
= T_{(\fp_\pm, k_j^{-1})} \cI \cX.
$$
\begin{eqnarray*}
\frac{e_{\bT'_\bR}(H^1(D_j)^m)}{e_{\bT'_\bR}(H^0(D_j)^m)} = |\mu_j ||G_\tau| D_{\mu_j,\lambda_j}
\end{eqnarray*}

Let
\begin{eqnarray*}
e_{\bT_f}(N^\vir_\Gamma) &:=& \iota_f^* e_{\bT'_\bR}(N^\vir_\Gamma) =  e_{\bT'_\bR}(N^\vir_\Gamma)\Big|_{\wu_1'=\wu, \wu_2'=f\wu},\\
e_{\bT_f}(N^\vir_\hGa) &:= &\iota_f^* e_{\bT'_\bR}(N^\vir_\hGa) =  e_{\bT'_\bR}(N^\vir_\hGa)\Big|_{\wu_1'=\wu, \wu_2'=f\wu}.
\end{eqnarray*}

Given $\gamma_1,\ldots,\gamma_n \in H^*_{\orb,\bT_f}(\cX)$, we define 
\begin{align*}
&\langle \gamma_1,\dots,\gamma_h\rangle^{\cX,(\cL,f)}_{\beta',((\mu_1,\lambda_1),\dots,(\mu_h,\lambda_h))}:=\int_{[F_\Ga]^\vir} \frac{ \big(\prod_{i=1}^n\ev_i^*\gamma_i\big)|_{F_\Gamma}}{e_{\bT_f}(N_\Ga^\vir)} \\
=& \prod_{j\in J_+ }\frac{\fr_+}{r_j\mu_j} \prod_{j\in J_-}\frac{\fr_-}{- r_j\mu_j} \prod_{j=1}^h (|\mu_j||G_\tau|D_{\mu_j,
\lambda_j,f}) \\
&\cdot \int_{[F_{\hGa}]^\vir} \frac{ \big(\prod_{i=1}^n\ev_i^*\gamma_i
\prod_{j\in J_+}\ev_{n+j}^*\phi_{\si_+, (k_j^+)^{-1}} \prod_{j\in J_-} \ev_{n+j}^*\phi^f_{\si_-, (k_j^-)^{-1}}\big)|_{F_\Gamma}}
{\prod_{j\in J_+}(\frac{\wu_1'}{r_j \mu_j} -\frac{\bar{\psi}_{n+j}}{r_j})
\prod_{j\in J_-}(\frac{\wu_1'}{r_j \mu_j} -\frac{\bar{\psi}_{n+j}}{r_j})e_{\bT_f}(N_\hGa^\vir)} \\
=& \prod_{j=1}^hD'_{\mu_j,\lambda_j,f} \cdot \int_{[F_{\hGa}]^\vir} \frac{ \big(\prod_{i=1}^n\ev_i^*\gamma_i
\prod_{j\in J_+}\ev_{n+j}^*\phi^f_{\si_+, (k_j^+)^{-1}} \prod_{j\in J_-}\ev_{n+j}^*\phi^f_{\si_-, (k_j^-)^{-1}}\big)|_{F_\Gamma}}{
\prod_{j=1}^h\frac{\wu}{\mu_j}(\frac{\wu}{\mu_j} -\bar{\psi}_{n+j})e_{\bT_f}(N_\hGa^\vir)}
\end{align*}
where
$$
D'_{d_0,\lambda,f}
=   \begin{cases}
\displaystyle{  -(-1)^{\lfloor d_0w^+_3 + \frac{\bar \lambda}{m}\rfloor}(\frac{\wu}{d_0})^{\age(k^+)} \cdot
\frac{\fr_+}{d_0}\cdot \frac{\prod_{a=1}^{\lfloor d_0w^+_1\rfloor +\age(k^+)-1}
(d_0 w^+_2 +a-c_2(k^+))} {\lfloor d_0w^+_1\rfloor !}}, \quad \quad d_0>0, & \\
& \\
\displaystyle{    -(-1)^{\lfloor d_0w_2^- + (1-\frac{\bar\lambda}{m}-\delta_{\bar \lambda,0})\rfloor} (\frac{\wu}{d_0 })^{\age(k^-)} \cdot
\frac{\fr_-}{-d_0}\cdot \frac{\prod_{a=1}^{\lfloor d_0 w_1^- \rfloor+\age(k^-)-1}(d_0w_3^--c_3(k^-)+a)}{\lfloor d_0w_1^-\rfloor!}  } , \quad\quad d_0<0. & \\
\end{cases}
$$
\end{proof}

Suppose that $(\cL,f)$ is a framed outer brane. Define
$$
D'_{d_0,\lambda,f}=   -(-1)^{\lfloor d_0w_3+ \frac{\bar \lambda}{m}\rfloor}(\frac{\wu}{d_0})^{\age(k)} \cdot
\frac{\fr}{d_0}\cdot \frac{\prod_{a=1}^{\lfloor d_0w_1\rfloor +\age(k)-1}
(d_0 w_2 +a-c_2(k))} {\lfloor d_0w_1\rfloor !}.
$$
where $k=h(d_0,\lambda)$. By the $d_0>0$ part of the proof of Proposition \ref{inner-psi}, we obtain: 
\begin{proposition}[framed outer brane]\label{outer-psi}
Suppose that $(\cL,f)$ is a framed inner brane, and $\vmu=((\mu_1,\lambda_1),\ldots, (\mu_h,\lambda_h))$,
where $(\mu_j,\lambda_j)\in H_1(\cL;\bZ)$. Then
$$ \langle \gamma_1,\ldots, \gamma_n\rangle_{g,\beta',\vmu}^{\cX, (\cL,f)}
=\prod_{j=1}^h D'_{\mu_j,\lambda,f} \cdot \int_{[\Mbar_{g,n+h}(\cX,\beta)]^\vir} \frac{ \big(\prod_{i=1}^n\ev_i^*\gamma_i
\prod_{j=1}^h\ev_{n+j}^*\phi^f_{\si, (h(d_0,\lambda))^{-1}}\big)}
{\prod_{j=1}^h\frac{\wu}{\mu_j}(\frac{\wu}{\mu_j} -\bar{\psi}_{n+j})}
$$
where
$$
\beta\in H_2(\cX),\quad \beta'= \beta + \big(\sum_{j=1}^h \mu_j\big)b.
$$
\end{proposition}

\subsection{Generating functions of open-closed GW invariants}\label{sec:Fgh}
From now on, we assume the generic stabilizer $K$ is trivial, so that $\cX=\cX^\rig$ is 
a toric Calabi-Yau 3-orbifold. Then
$$
\chi_3: G_\tau \lra \bmu_\fm, \quad \lambda \mapsto e^{2\pi\sqrt{-1}\bar{\lambda}/\fm}
$$
is a group isomorphism. We have $\bar{N}=N$ and $\bar{b}_i=\bar{b}_i$. In particular,
$$
b_1 = \fr \wv_1 -\fs \wv_2 + \wv_3,\quad b_2 = \fm \wv_2 + \wv_3,\quad b_3 = \wv_3.
$$
There exists $m_a,n_a\in \bZ$, such that
$$
b_{3+a} = m_a \wv_1 + n_a \wv_2 + \wv_3,\quad a=1,\ldots,k.
$$

Introduce variables $\{ X_j \mid j=1,\ldots,h\}$
and let
$$
\btau_2 = \sum_{i=1}^m \tau_i u_i
$$
where $u_1,\ldots, u_m$ form a basis of $H^2_\orb(\cX;\bQ)$.
We choose $\bT'$-equivariant lifting of $\btau_2$ as follows:
for each $u_i\in H^2_\orb(\cX;\bQ)$, we choose the unique $\bT'$-equivariant
lifting $u_i^{\bT'} \in H^2_{\orb,\bT'}(\cX;\bQ)$ such that 
$\iota_\si^* u_i^{\bT'} = 0\in H^2_{\orb,\bT'}(\fp_\si;\bQ)$, where 
$\iota_\si^*: H^2_{\orb,\bT'}(\cX;\bQ)\to H^2_{\orb,\bT'}(\fp_\si;\bQ)$ is induced by the
inclusion map $\iota_\si:\fp_\si\to \cX$.

We define $\xi_0:= e^{-\pi\sqrt{-1}/\fm}$. If $(\cL,f)$ is a framed outer brane, define
\begin{equation}
\begin{aligned}
& F_{g,h}^{\cX,(\cL,f)}(\btau_2, Q^b,  X_1,\ldots,X_h)\\
& = \sum_{\beta', n\geq 0}\sum_{(\mu_j, \lambda_j)\in H_1(\cL;\bZ)}
\frac{\langle (\iota_f^*\btau_2) ^n\rangle_{g,\beta,(\mu_1,\lambda_1),\ldots, (\mu_h,\lambda_h)}^{\cX,(\cL,f)}}{n!} 
\prod_{j=1}^h (Q^bX_j)^{\mu_j}
 \cdot \xi_0^{\bar{\lambda}_1} \one_{\lambda_1^{-1}}
\otimes \cdots \otimes  \xi_0^{\bar{\lambda}_h}   \one _{\lambda_h^{-1}}.
\end{aligned}
\end{equation}
which is a function which takes values in $H^*_\orb(\cB G_\tau;\bC)^{\otimes h}$, where
$$
H^*_\orb(\cB G_\tau;\bC)=\bigoplus_{\lambda \in G_\tau} \bC \one_\lambda.
$$
When $\lambda=1$ is the identity element of $G_\tau$, $\one_1=1$ is the unit
of $H^*_\orb(\cB G_\tau;\bC)$. 

If $(\cL,f)$ is a framed inner brane, define
\begin{equation}
\begin{aligned}
& F_{g,h}^{\cX,(\cL,f)}(\btau_2, Q^b, X_1,\ldots, X_h) =  
\sum_{\beta', n\geq 0}\sum_{(\mu_j,\lambda_j)\in H_1(\cL;\bZ)}
\frac{\langle (\iota_f^*\btau_2)^n\rangle_{g,\beta,(\mu_1,\lambda_1),\ldots, (\mu_h,\lambda_h)}^{\cX,(\cL,f)}}{n!} \\
& \quad \quad \cdot  \prod_{\substack{ j\in \{1,\ldots, h\}\\ \mu_j>0} } (Q^b X_j)^{\mu_j}
\prod_{\substack{ j\in \{1,\ldots, h\}\\  \mu_j<0} } (Q^{b-\alpha}X_j)^{\mu_j}
\cdot \xi_0^{\bar{\lambda}_1}  \one_{\lambda_1^{-1}}
\otimes \cdots \otimes \xi_0^{\bar{\lambda}_h} \one_{\lambda_h^{-1}}.
\end{aligned}
\end{equation}
which is a function which takes values in $H^*_\orb(\cB G_\tau;\bC)^{\otimes h}$.

\subsection{The equivariant $J$-function and the disk potential}
Let $\{u_i\}_{i=1}^N$ be a homogeneous basis of $H^*_{\bT,\orb}(\cX;\bQ)$, and $\{u^i\}_{i=1}^N$ be its dual basis.
Define
$$
\btau =\sum_{i=1}^N \tau_i u_i =\btau_0+ \btau_2 + \btau_{>2}
$$
where
$$
\btau_0\in H^0_{\bT,\orb}(\cX;\bC),\quad \btau_2\in H^2_{\bT,\orb}(\cX;\bC),\quad
\btau_{>2} \in H^{>2}_{\bT,\orb}(\cX;\bC).
$$

The $J$-function \cite{T, CG07, Gi96}
is a $H^*_{\bT,\orb}(\cX)$-valued function:
$$
J(\btau,z):= 1+\sum_{\beta\ge 0, n\ge 0} \frac{1}{n!}\sum_{i=1}^N \langle 1,\btau^n, \frac{u_i}{z-\bar \psi}\rangle^{\cX,\bT}_{0,\beta} u^i.
$$
Then
$$
\iota_\si^* J(\btau,z)\Big|_{\wu_1=\wu,\, \wu_2=f\wu,\, \wu_3=0} =\sum_{k\in G_\si} J^f_{\si,k}(\btau,z) \one_k,
$$
where
$$
J_{\si,k}^f(\btau,z)= \delta_{k,1} +\sum_{\beta\ge 0, n\ge 0} \frac{1}{n!} \sum_{i=1}^N \langle 1, (\iota_f^*\btau)^n, 
\frac{|G_\si|\phi^f_{\si,k^{-1}}}{z-\bar \psi}\rangle^{\cX,\bT_f}_{0,\beta}.
$$

As a special case of Proposition \ref{outer-psi},
\begin{eqnarray*}
\langle \gamma_1,\ldots, \gamma_n\rangle_{0,\beta+d_0 b,(d_0 ,k)}^{\cX,(\cL,f)}
&=& D'_{d_0 ,\lambda,f} \int_{[\Mbar_{0,n+1}(\cX,\beta)]^\vir}
 \frac{ \big(\prod_{i=1}^n\ev_i^*\gamma_i\cup \ev_{n+1}^*\phi^f_{(h^+(d_0,\lambda))^{-1}}\big)}{\frac{\wu}{d_0}(\frac{\wu}{d_0} -\bar{\psi}_{n+1})}\\
&=& D'_{d_0 ,\lambda,f} \langle 1,\gamma_1,\dots,\gamma_n,\frac{\phi^f_{(h^+(d_0,\lambda))^{-1}}}{\frac{\wu}{d_0} -\bar{\psi}} \rangle^\cX_{0,\beta};
\end{eqnarray*}
\begin{eqnarray*}
F^{\cX,(\cL,f)}_{0,1}(\btau_2,Q^b, X_1) &=& \sum_{\beta,n\geq 0} 
\sum_{(d_0 ,\lambda)\in H_1(\cL;\bZ)} \frac{1}{n!}\langle (\iota_f^*\btau_2)^n\rangle^{\cX,(\cL,f)}_{0,\beta+d_0 b,(d_0 ,\lambda)}
(Q^bX)^{d_0} \xi_0^{\bar{\lambda}}\one_{\lambda^{-1}}\\
&=& \frac{1}{|G_\si|}\sum_{(d_0,\lambda)\in H_1(\cL;\bZ)} (Q^b X_1)^{d_0} D'_{d_0 ,\lambda,f}
J^f_{\si,h^+(d_0,\lambda)}(\btau_2 ,\frac{\wu}{d_0})\xi_0^{\bar{\lambda}}\one_{\lambda^{-1}} .
\end{eqnarray*}
\begin{proposition}\label{pro:F-J} Let $X=Q^b X_1$.
If $(\cL,f)$ is a framed outer brane, then
$$
F^{\cX,(\cL,f)}_{0,1}(\btau_2,X) = \frac{1}{|G_\si|}
\sum_{(d_0 ,\lambda)\in H_1(\cL;\bZ)} X^{d_0} D'_{d_0 ,\lambda,f}J^f_{\si,h^+(d_0,\lambda)}(\btau_2,\frac{\wu}{d_0})
\xi_0^{\bar{\lambda}}\one_{\lambda^{-1}}.
$$
If $(\cL,f)$ is a framed inner brane, then
\begin{eqnarray*}
&& F^{\cX,(\cL,f)}_{0,1}(\btau_2, Q, X) \\
&=& \frac{1}{|G_\si|}\sum_{(d_0,\lambda)\in H_1(\cL;\bZ), d_0 >0} X^{d_0} D'_{d_0 ,\lambda,f}
J^f_{\si_+,h^+(d_0,\lambda)}(\btau_2,\frac{\wu}{d_0}) \xi_0^{\bar{\lambda}}\one_{\lambda^{-1}} \\
&& + \frac{1}{|G_\si|} \sum_{(d_0 ,\lambda)\in H_1(\cL;\bZ), d_0 <0}
X^{d_0} Q^{-d_0 \alpha} D'_{d_0,\lambda,f}J^f_{\si_-,h^-(d_0,\lambda)}(\btau_2,\frac{\wu}{d_0}) \cdot \frac{\fr_+}{\fr_-} \xi_0^{\bar{\lambda}}\one_{\lambda^{-1}}.
\end{eqnarray*}
\end{proposition}

\section{Mirror symmetry for the disk amplitudes} \label{sec:mirror}

\subsection{The equivariant $I$-function and the equivariant mirror theorem} \label{sec:I}
We choose $p_1,\ldots, p_k \in \bL^\vee\cap \tNef_\cX$ such that
\begin{itemize}
\item $\{p_1,\ldots,p_k\}$ is a $\bQ$-basis of $\bL^\vee_\bQ$.
\item $\{\bar{p}_1,\ldots, \bar{p}_{k'}\}$ is a $\bQ$-basis
of $H^2(\cX;\bQ)$.
\item $p_a=D_{3+a}$ for $a= k'+1,\dots,k$.
\end{itemize}
We define \emph{charges} $m^{(a)}_i\in \bQ$ by $D_i=\sum_{a=1}^k m^{(a)}_i p_a$.

Let $q'_0,q_1,\ldots, q_k$ be $k+1$ formal variables, and define
$q^\beta = q_1^{\langle p_1,\beta\rangle} \cdots q_k^{\langle p_k,\beta\rangle}$
for $\beta\in \bK$. We take the equivariant lifting $\bar p^\cT_a\in H^2_\bT(\cX;\bQ)$ of $\bar p_a\in H^2(\cX;\bQ)$. The equivariant $I$-function is an
$H^*_{\orb,\bT} (\cX)$-valued power series defined as follows \cite{Ir09}:
\begin{eqnarray*}
I(q_0',q,z) &=& e^{\frac{\log q'_0+\sum_{a=1}^{k'} \bar p_a^\cT \log q_a}{z}}
\sum_{\beta\in \bK_\eff} q^\beta \prod_{i=1}^{r'}
\frac{\prod_{m=\lceil \langle D_i, \beta \rangle \rceil}^\infty (\bar D^\cT_i  +(\langle D_i,\beta\rangle -m)z)}{\prod_{m=0}^\infty(\bar D^\cT_i+(\langle D_i,\beta\rangle -m)z)} \\
&& \quad \cdot \prod_{i=r'+1}^r \frac{\prod_{m=\lceil \langle D_i, \beta \rangle \rceil}^\infty (\langle D_i,\beta\rangle -m)z}
{\prod_{m=0}^\infty(\langle D_i,\beta\rangle -m)z} \one_{v(\beta)}
\end{eqnarray*}
where $q^\beta=\prod_{a=1}^k q_a^{\langle p_a,\beta\rangle}$. Note that $\langle p_a,\beta\rangle\geq 0$ for $\beta\in \bK_{\eff}$.
The equivariant $I$-function can be rewritten as
\begin{eqnarray*}
I(q_0',q,z) &=& e^{\frac{\log q'_0+\sum_{a=1}^{k'} \bar p_a^\cT \log q_a}{z}}
\sum_{\beta\in \bK_\eff} \frac{q^\beta}{z^{\langle \hat \rho ,\beta\rangle + \age(v(\beta))} } \prod_{i=1}^{r'}
\frac{\prod_{m=\lceil \langle D_i, \beta \rangle \rceil}^\infty ( \frac{\bar D^\cT_i}{z}  +\langle D_i,\beta\rangle -m)}
{\prod_{m=0}^\infty(\frac{\bar D^\cT_i}{z}+ \langle D_i,\beta\rangle -m)} \\
&& \quad \cdot \prod_{i=r'+1}^r \frac{\prod_{m=\lceil \langle D_i, \beta \rangle \rceil}^\infty (\langle D_i,\beta\rangle -m)}{\prod_{m=0}^\infty(\langle D_i,\beta\rangle -m)}
\one_{v(\beta)}
\end{eqnarray*}
where $\hat \rho = D_1+\cdots +D_r \in \tC_\cX$.

Since $\cX$ is a Calabi-Yau orbifold, $\age(v)$ is an integer for any $v\in \BoxS$. Then
$$
H^{\leq 2}_{\orb,\bT}(\cX) = H^0_{\orb.\bT}(\cX) \oplus H^2_{\orb,\bT}(\cX).
$$
Let $\cQ=\bQ(\wu_1,\wu_2,\wu_3)$ be the fractional field of $H^*_{\bT}(\mathrm{point};\bQ)$.
\begin{align*}
H^0_{\orb,\bT}(\cX;\cQ) &= \cQ \one , \\
H^2_{\orb,\bT}(\cX;\cQ) &= \bigoplus_{a=1}^{k'}\cQ {\bar{p}_a^\cT}\oplus \bigoplus_{\substack{v\in {\BoxS}\\ \age(v)=1}} \cQ \one_v.
\end{align*}
Recall that the embedding of the stacky fixed point $\fp_\sigma$ is $\iota_\sigma:\fp_\sigma\to \cX$. We choose the lifting $\bar{p}^\cT_a$ such that $\iota_\sigma^*\bar{p}^\cT_a=0$.

For $i=1,\ldots, r$, we will define $\Omega_i \subset \bK_\eff-\{0\}$
and $A_i(q)$ supported on $\Omega_i$. We observe that, if $\beta\in \bK_\eff$ and $v(\beta)=0$ then
$\langle D_i,\beta\rangle\in \bZ$ for $i=1,\ldots, r$.
\begin{itemize}
\item For $i=1,\ldots,r'$, let
$$
\Omega_i =\left \{ \beta\in \bK_\eff: v(\beta)=0, \langle D_i,\beta\rangle <0 \textup{ and  }
\langle D_j, \beta \rangle \geq 0 \textup{ for } j\in \{1,\ldots, r\}-\{i\} \right \}.
$$
Then $\Omega_i\subset \{ \beta\in \bK_\eff: v(\beta)=0, \beta\neq 0\}$. We define
$$
A_i(q):=\sum_{\beta \in \Omega_i} q^\beta
\frac{(-1)^{-\langle D_i,\beta\rangle-1}(-\langle D_i, \beta\rangle -1)! }{
\prod_{j\in\{1,\ldots, r\}-\{i\}}\langle D_j,\beta\rangle!}.
$$
\item For $i=r'+1,\ldots, r$, let
$$
\Omega_i := \{ \beta\in \bK_\eff: v(\beta)= b_i, \langle D_j, \beta\rangle \notin \bZ_{<0} \textup{ for }j=1,\ldots, r\},
$$
and define
$$
A_i(q) = \sum_{\beta\in \Omega_i} q^\beta
\prod_{j=1}^r \frac{\prod_{m=\lceil \langle D_j, \beta \rangle \rceil}^\infty (\langle D_j,\beta\rangle -m)}{\prod_{m=0}^\infty(\langle D_j,\beta\rangle -m)}.
$$
\end{itemize}

Let $\si$ be the smallest cone containing $b_i$. Then
$$
b_i=\sum_{j\in I_\si'}c_j(b_i)b_j,
$$
where $c_j(b_i)\in (0,1)$ and $\sum_{j\in I_\si'}c_j(b_i)=1$. There exists a unique $D_i^\vee \in \bL_\bQ$ such that
$$
\langle D_j, D_i^\vee\rangle =\begin{cases}
1, & j=i,\\
-c_j(b_i), &  j\in I'_\si,\\
0, & j\in I_\si-\{i\}.
\end{cases}
$$
Then
$$
A_i(q)=q^{D_i^\vee} + \textup{ higher order terms }
$$
$$
I(q_0',q,z)= 1 +\frac{1}{z}(\log q_0' \one+\sum_{a=1}^{k'} \log(q_a)\bar{p}^{\cT}_a +
\sum_{i=1}^{r'} A_i(q) \bar{D}_i^{\cT}+ \sum_{i=r'+1}^r A_i(q) \one_{b_i}) + o(z^{-1}).
$$
For $i=1,\ldots, r'$,
$$
\bar{\cD}^{\cT}_i =\sum_{a=1}^{k'}m_i^{(a)} \bar{p}_a^{\cT} +\lambda_i
$$
where $\lambda_i\in H^2(B\bT;\bQ)$. Let $S_a(q):=\sum_{i=1}^{r'} m_i^{(a)} A_i(q)$. Then
$$
I(q_0',q,z)= 1 +\frac{1}{z}( (\log q_0' +\sum_{i=1}^{r'}\lambda_i A_i(q)) \one
 +\sum_{a=1}^{k'} (\log(q_a) + S_a(q)) \bar{p}^{\cT}_a
+ \sum_{i=r'+1}^r A_i(q) \one_{b_i}) + o(z^{-1}).
$$

Recall that the $\bT$-equivariant $J$-function for $\cX$ is
$$
J(\btau,z)=1+\sum_{\beta\ge 0, n\ge 0} \sum_{i=1}^N \frac{1}{n!}
\langle 1,\btau^n, \frac{u_i}{z-\bar \psi}\rangle^{\cX,\bT}_{0,\beta} u^i,
$$
where $\{u_i\}_{i=1}^N$ is an $H^*(B\bT)$-basis of $H^*_\bT(\cX;\bQ)$ and $\{u^i\}_{i=1}^N$ is the dual basis. Assume $u_0=\one$, $u_a=\bar p^\cT_a$ for $a=1,\dots, k'$ and $u_a=\one_{b_{a+3}}$ for $a=k'+1,\dots, k$.
The mirror theorem for toric orbifolds \cite{CCIT} implies following theorem.
\begin{theorem}[Coates-Corti-Iritani-Tseng {\cite{CCIT}}]
\label{thm:mirror}
If the toric orbifold $\cX$ satisfies Assumption \ref{semi-proj}, then
$$
e^{\frac{\tau_0(q_0',q)}{z}  } J(\btau_2(q),z) = I(q_0',q, z),
$$
where the equivariant closed mirror map $(q_0',q)\mapsto  \tau_0(q_0',q)1 + \btau_2(q)$ is
determined by the first-order term in the asymptotic expansion of the $I$-function
$$
I(q_0',q,z)=1+\frac{\tau_0 (q_0',q) 1 +\btau_2(q)}{z}+o(z^{-1}).
$$
More explicitly, the equivariant closed mirror map is given by
\begin{eqnarray}
\label{eqn:closed-mirror-map}
\tau_0 &=& \log(q_0') +\sum_{i=1}^{r'} \lambda_i A_i(q), \nonumber\\
\tau_a &=& \begin{cases}
\log(q_a)+ S_a(q), & 1\leq a\leq k',\\
A_{a-3}(q), & k'+1\leq a\leq k.
\end{cases}.
\end{eqnarray}
\end{theorem}

\subsection{The pullback of the disk potential under the mirror map}

By Proposition \ref{pro:F-J}, if $(\cL,f)$ is a framed outer brane, then
$$
F^{\cX,(\cL,f)}_{0,1}(\btau_2,Q,X) = \frac{1}{|G_\si|} 
\sum_{(d_0 ,\lambda)\in H_1(\cL;\bZ)} X^{d_0} D'_{d_0,\lambda,f}J^f_{\si,h(d_0,\lambda)}(\btau_2,\frac{\wu}{d_0})
\xi_0^{\bar{\lambda}} \one_{\lambda^{-1}}.
$$
Let $F^{\cX,(\cL,f)}(q,X)$ be the pullback of
$F^{\cX,(\cL,f)}_{0,1}(\btau_2,Q, X)$ under the closed mirror map.

By Proposition \ref{pro:F-J}, if $(\cL,f)$ is a framed inner brane, then
\begin{eqnarray*}
&& F^{\cX,(\cL,f)}_{0,1}(\btau_2 ,Q, X) \\
&=&\frac{1}{|G_\si|}  \sum_{(d_0,\lambda)\in H_1(\cL;\bZ), d_0 >0} X^{d_0} D'_{d_0,\lambda,f}J^{f}_{\si_+,h^+(d_0,\lambda)}(\btau_2,\frac{\wu}{d_0})
\xi_0^{\bar{\lambda}} \one_{\lambda^{-1}}  \\
&& + \frac{1}{|G_\si|} \sum_{(d_0,k^+, k^-)\in H_{\tau,\si_+,\si_-}, d_0 <0} X^{d_0} Q^{-d_0 \alpha}
D'_{d_0,\lambda,f}J^{f}_{\si_-,h^-(d_0,\lambda)}(\btau_2,\frac{\wu}{d_0 })\cdot \frac{\fr_+}{\fr_-} \xi_0^{\bar{\lambda}}\one_{\lambda^{-1}}.
\end{eqnarray*}
Given $\si\in \Si(3)$, $k\in G_\si$, and $f\in \bZ$, define $I_{\si,k}^f(q,z)$ by
$$
\iota_\si^* I(q,z)\Big|_{\wu_1=\wu,\, \wu_2=f\wu,\, \wu_3=0} =\sum_{k\in G_\si}I_{\si,k}^f(q,z)\one_k.
$$
Since a toric Calabi-Yau orbifold satisfies the weak Fano condition, by the equivariant mirror theorem (Theorem \ref{thm:mirror}), we may write $F^{\cX,(\cL,f)}(q,X)$ in terms of $I_{\si,k}^f(q,z)$ in case of an outer brane, and in terms of $I_{\si_+,k^+}^f(q,z)$ and $I_{\si_-,k^-}^{f}(q,z)$ in case of an inner brane.
\begin{lemma}
\label{lemm:F(q,X)}
If $(\cL,f)$ is a framed outer brane, then
\begin{equation}
F_{0,1}^{\cX,(\cL,f)}(q,X) =  \frac{1}{|G_\si|} \sum_{(d_0,\lambda)\in H_1(\cL;\bZ)} X^{d_0} D'_{d_0,\lambda,f}
e^{\frac{-d_0 \tau_0(q) }{\wu}} I^f_{\si,h(d_0,\lambda)}(q,\frac{\wu}{d_0}) \xi_0^{\bar{\lambda}} \one_{\lambda^{-1}}.
\end{equation}
If $(\cL,f)$ is a framed inner brane, then
\begin{align*}
& F_{0,1}^{\cX,(\cL,f^+,f^-)}(q,X) \\
=& \frac{1}{|G_\si|} \sum_{(d_0,\lambda)\in H_1(\cL;\bZ), d_0 >0} X^{d_0} D'_{d_0,\lambda,f}
e^{\frac{-d_0 \tau_0(q) }{\wu}}I^{f}_{\si_+,h^+(d_0,\lambda)}(q,\frac{\wu}{d_0})
 \xi_0^{\bar{\lambda}} \one_{\lambda^{-1}}  \\
& + \frac{1}{|G_\si|}\sum_{(d_0,\lambda)\in H_1(\cL;\bZ), d_0 <0} X^{d_0}
Q^{-d_0 \alpha} D'_{d_0,\lambda,f} e^{\frac{-d_0 \tau_0(q)}{\wu}} I^{f}_{\si_-,h^-(d_0,\lambda)}(q,\frac{\wu}{d_0}) \cdot \frac{\fr_+}{\fr_-}
\xi_0^{\bar{\lambda}} \one_{\lambda^{-1}}.
\end{align*}
\end{lemma}

Let $(\cL,f)$ be a framed brane, and let $\tau$, $\si=\si_+, \si_-$ be defined as in Section \ref{sec:AV-brane}.
Recall that
\begin{align*}
& I'_\si = \{ i\in \{1,\ldots,r'\}: \rho_i\subset \si\} =\{1,2, 3 \},\quad
I_\si  = \{1,\ldots, r\} \setminus I'_\si,\\
& I'_\tau = \{i\in \{1,\ldots, r'\}:\rho_i\subset \tau\}=  \{2, 3\},\quad
I_\tau =\{1,\ldots, r\} \setminus I'_\tau,\\
& \bK_{\eff,\si} =   \{ \beta\in \bL_\bQ: \langle D_i, \beta\rangle \in \bZ_{\geq 0} \textup{ for } i\in I_\si\}.
\end{align*}
In case that $\cL$ is inner,
\begin{align*}
&I'_{\si_-}=\{2,3,4\},\ I_{\si_-}=\{1,\dots,r\}\backslash I'_{\si_-}\\
&\bK_{\eff,\si-} =   \{ \beta\in \bL_\bQ: \langle D_i, \beta\rangle \in \bZ_{\geq 0} \textup{ for } i\in I_{{\si_-}}\}.
\end{align*}
Let $\wb_{\si,i}=\iota_\si^* \bar{D}^\cT_i \in H^2_{\bT}(\fp_\si;\bQ)=H^2(B\bT;\bQ)$
for $1\leq i\leq r$, and then $\wb_i=0$ for $r'+1\leq i\leq r$.  For $\beta\in \bK_{\eff,\si}$, define an $H^*({ B\bT};\bQ)$-valued
\begin{equation}
I(\si,\beta): =
\prod_{i=1}^r \frac{\prod_{m=\lceil \langle D_i, \beta \rangle \rceil}^\infty (\wb_{\si,i} +(\langle D_i,\beta\rangle -m)\frac{\wu_1}{d_0})}
{\prod_{m=0}^\infty(\wb_{\si,i} +(\langle D_i,\beta\rangle -m)\frac{\wu_1}{d_0})}
\end{equation}
Recall that $\iota_\si^*\bar{p}_a^{\cT}=0$, so
$$
\iota_\si^* I(q,z)|_{z=\frac{\wu_1}{d_0}} = \sum_{\beta\in \bK_{\eff,\si} }
e^{\frac{d_0}{\wu_1}\log q_0'} q^\beta I(\si,\beta) \one_{v(\beta)}.
$$
With the above notation, if $\cL$ is an outer brane  we can rewrite $F_{0,1}^{\cX,(\cL,f)}(q,X)$ as
$$
F^{\cX,(\cL,f)}_{0,1}(q,X)=\frac{1}{|G_\si|} \sum_{(d_0,\lambda)\in H_1(\cL;\bZ)}
\sum_{\beta\in \bK_{\eff,\si}, v(\beta)=h(d_0,\lambda)} x^{d_0} q^\beta D'_{d_0,\lambda,f}I^f(\si,\beta)
\xi_0^{\bar{\lambda}}  \one_{\lambda^{-1}}
$$
where $I^f(\si,\beta)=I(\si,\beta)|_{\wu_1=\wu,\, \wu_2=f\wu,\, \wu_3=0}$, and
$$
x=X \exp\big(\frac{\log q_0'-\tau_0(q)}{\wu_1}\big)
$$
is the B-brane moduli parameter.

Following \cite{LM01,Ma01}, we define {\em extended charge vectors}
$$
\{m^{(a)}_i\}_{i=1,\dots,r}^{a=0,\dots,k}=\begin{pmatrix}
 w_1 & w_2 & w_3 & 0 & \dots\\
& & \{ m^{(a)}_i\}_{i=1,\dots,r}^{a=1,\dots,k} &
\end{pmatrix},
$$
such that $m^{(0)}_i=w_i$ for $i=1,2,3$ and $m^{(0)}_i=0$ for $i=4,\dots, r$. Recall that
$$
\tau_0 + \sum_{a=1}^{k'} \tau_a  \bar{p}_a^{\cT} +  \sum_{a=k'+1}^k \tau_a \one_{b_{a+3}} =
\log q_0' + \sum_{a=1}^{k'} \log q_a \bar{p}_a^{\cT} +\sum_{i=1}^{r'} A_i(q)\bar{\cD}^{\cT}_i +\sum_{i=r'+1}^r A_i(q)\one_{b_i}.
$$
We pull back the above identity under $\iota_\si^*$. Since
$$
\iota_\si^*\bar{p}_a^{\cT}=0,\quad
\iota_\si^*\cD^{\cT}_i \Bigr|_{\wu_1=\wu,\, \wu_2=f\wu,\, \wu_3 = 0} =  m^{(0)}_i \wu,
$$
we get
$$
\tau_0(q_0',q)=\log q'_0+ \sum_{i=1}^{r'} m^{(0)}_i  A_i(q) \wv.
$$
So the open mirror map is given by
\begin{equation}
\label{eqn:open-mirror-map}
\log  X = \log x + \sum_{i=1}^{r'} m^{(0)}_i A_i(q).
\end{equation}
If $\cL$ is inner, we further set $Q=q$. Denote the pullback of the disk potential $W^{\cX,(\cL,f)}(q,x)$ to be the pullback of $F_{0,1}^{\cX,(\cL,f)}(q,X)$ under this open mirror map. Then by Lemma \ref{lemm:F(q,X)}
\begin{equation}
|G_\si| W^{\cX,(\cL,f)}(q,x)=\begin{cases}\displaystyle{\sum_{\substack{d_0>0,\beta\in \bK_{\eff,\si} \\ v(\beta)=h(d_0,\lambda)}} x^{d_0}q^\beta D'_{d_0,\lambda,f} I^f(\sigma,\beta)
\xi_0^{\bar{\lambda}} \one_{\lambda^{-1}} },& \text{$\cL$ is outer,}\\
\\
\displaystyle{\sum_{\substack{d_0>0,\beta \in \bK_{\eff,\si_+} \\ v(\beta)=h^+(d_0,\lambda)}} x^{d_0}q^\beta D'_{d_0,\lambda,f} I^f(\sigma^+,\beta)
\xi_0^{\bar{\lambda}}}\one_{\lambda^{-1}} \\\qquad\qquad\displaystyle{+\frac{\fr_+}{\fr_-}\sum_{\substack{d_0<0, \beta\in \bK_{\eff,\si_-} \\ v(\beta)=h^-(d_0,\lambda)}} x^{d_0}q^{\beta-d_0\alpha} D'_{d_0,\lambda,f} I^f(\sigma^-,\beta) \xi_0^{\bar{\lambda}} \one_{\lambda^{-1}} },& \text{$\cL$ is inner}.
\end{cases}
\label{eqn:W}
\end{equation}



Given $\tbeta=(d_0,\beta)\in \bZ\times \bK_\si$, define the extended or open sector pairing to be
$$
\langle D_i, \tbeta \rangle= m_i^{(0)}d_0 +\langle D_i, \beta \rangle.$$

Recall that $\{D_i: i\in I_\si\}$ is a $\bQ$-basis of $\bL^\vee_\bQ \cong \bQ^k$ and a $\bZ$-basis of $\bK_\si^{{ \vee}}\cong \bZ^k$.
Let $v_a= D_{a+3}$ for $a=1,\dots, k$., and let $\{ h_a\}_{a=1,\ldots,k}$ be the dual $\bQ$-basis
of $\bL_\bQ$. Then $\{h_a\}_{a=1,\ldots,k}$ is a $\bZ$-basis
of $\bK_\si\cong \bZ^k$, and
$$
\bK_{\eff,\si} =\sum_{a=0}^k\bZ_{\geq 0} h_a.
$$
Given any $(d_0,\beta)\in \bZ\times \bK_\si$, define
$$
q^{\tbeta} = x^{d_0} q^\beta = x^{d_0} \prod_{a=1}^k q_a^{\langle p_a,\beta\rangle}.
$$
Define
$$
\bK_{\eff}(\cX,\cL) =\{\tbeta=(d_0,\beta)\in \bZ  \times \bK_{\eff,\si}: \langle D_1,\tbeta\rangle\in \bZ_{\geq 0},d_0\neq 0 \}.
$$
\begin{theorem} \label{thm:amplitude-term-outer}
Assuming the Aganagic-Vafa brane $(\cL,f)$ is either inner or outer,
\begin{align}
\label{eqn:amplitude-term-outer}
W^{\cX,(\cL,f)}(q,x)=\sum_{\substack{\tbeta\in \bK_{\eff}(\cX,\cL) \\ v(\beta)=h(d_0,\lambda)}} q^\tbeta A^{\cX,(\cL,f)}_{\tbeta}\xi_0^{\bar{\lambda}}\one_{\lambda^{-1}}.
\end{align}
where
\begin{align*}
A^{\cX,(\cL,f)}_{\tbeta=(d_0,\beta)}= \frac{- (-1)^{\lfloor \langle D_3,\tbeta \rangle \rfloor} }{m d_0
\prod_{i\in I_\tau} \langle D_i,\tbeta\rangle!}
\cdot \frac{\Gamma(-\langle D_{3},\tbeta \rangle) }{\Gamma(\langle D_{2},\tbeta \rangle+1)}. 
\end{align*}
\end{theorem}

\begin{proof}
Assume $\cL$ is an outer brane. Let $\tbeta=(d_0,\beta)$, and let $\ep_j =c_{j}(v)$ for $j=1,2,3$. By \eqref{eqn:W} (and reall that $|G_\si|=\fr\fm$)
$$
W^{\cX,(\cL,f)}(q,x)= \frac{1}{\fr \fm} \sum_{\substack{d_0>0\\ \beta\in \bK_{\eff,\si} \\ v(\beta)=h(d_0,\lambda)}} 
x^{d_0}q^\beta D'_{d_0,\lambda,f} I^f(\sigma,\beta)\xi_0^{\bar{\lambda}}\one_{\lambda^{-1}}.
$$
Given any $\beta\in \bK_{\eff,\si}$, we have $\lceil \langle D_{j},\beta\rangle \rceil -\ep_j =\langle D_{j},\beta\rangle$ for $j=1, 2,3$, and $\lceil\langle D_i, \beta\rangle \rceil =\langle D_i,\beta\rangle$ for $i\in I_\si$.
The disk factors
\begin{align*}
D'_{d_0,\lambda,f}&= -(-1)^{\lfloor d_0 w_3 +\frac{\bar \lambda}{m}\rfloor}
\frac{\fr}{d_0} \left ( \frac{\wu}{d_0} \right)^{\age(h(d_0,\lambda))}
\frac{\prod_{a=1}^{\lfloor d_0 w_1 \rfloor + \age(h(d_0,\lambda)) -1} (d_0w_2 + a-\ep_2)}{\lfloor d_0w_1 \rfloor!} \\
&= -(-1)^{\lfloor d_0 w_3+\frac{\bar \lambda}{m}\rfloor }
\frac{\fr}{d_0} \left ( \frac{\wu}{d_0} \right)^{\age(h(d_0,\lambda))}\frac{1}{\Gamma(w_1d_0-\epsilon_1+1)}\cdot \frac{\Gamma(-w_3d_0+\epsilon_3)}{\Gamma(w_2d_0-\epsilon_2+1)}
\end{align*}
The pullback of the coefficients of the $I$-function is
\begin{align*}
& I^f(\si,\beta)
=& (\frac{\wu}{d_0})^{-\age(h(d_0,\lambda))}\frac{1}{\prod_{i\in I_\sigma}\Gamma(\langle D_i,\beta \rangle+1)}\cdot \frac{\Gamma(w_1d_0-\epsilon_1+1)}{\Gamma(\langle D_1,\tbeta \rangle +1)}\cdot \frac{\Gamma(w_2d_0-\epsilon_2+1)}{\Gamma(\langle D_2,\tbeta \rangle +1)}\cdot \frac{\Gamma(w_3d_0-\epsilon_3+1)}{\Gamma(\langle D_3,\tbeta \rangle +1)}.
\end{align*}
Hence
$$
W^{\cX,(\cL,f)}(q,x)=\sum_{\substack{\tbeta\in \bK_\eff(\cX,\cL)\\v(\beta)=h(d_0,\lambda)}} x^{d_0} q^\beta A^{\cX,(\cL,f)}_{(d_0,\beta)} \xi_0^{\bar{\lambda}} \one_{\lambda^{-1}}
$$
where
$$
A^{\cX, (\cL,f)}_{\tbeta} =
(-1)^{\lfloor \langle D_3,\tbeta \rangle \rfloor}  \frac{-1}{\fm d_0} \frac{1}{\prod_{i\in I_\tau}\Gamma(\langle D_i,\tbeta \rangle +1)}\cdot \frac{\Gamma (-\langle D_3,\tbeta \rangle )}{\Gamma (\langle D_2,\tbeta \rangle+1)}.
$$
Note that $\bK_\eff(\cX,\cL)\subset \bZ_{>0} \times \bK_{\eff,\si}$, and for any $(d_0,\beta)\in (\bZ_{>0}\times\bK_{\eff,\si}) \backslash \bK_\eff(\cX,\cL)$, $A^{\cX,(\cL,f)}_{(d_0,\beta)}=0$.

In case that $\cL$ is inner, by \eqref{eqn:W}
$$
W^{\cX,(\cL,f)}(q,x) = I_+ + I_-,
$$
where
\begin{align*}
I_+&=\sum_{\substack{d_0>0, \beta\in \bK_{\eff,\si_+}\\v(\beta)=h^+(d_0,\lambda)}} x^{d_0} q^\beta \frac{-1}{\fm d_0} \frac{(-1)^{\lfloor w^+_3 d_0 +\langle D_3,\beta \rangle \rfloor}}{\prod_{i\in I_{\sigma^+}}\Gamma(\langle D_i,\beta \rangle +1)}\cdot \frac{\Gamma (-w^+_3d_0 -\langle D_3, \beta \rangle )}{\Gamma (w^+_2d_0+\langle D_2,\beta \rangle+1)\Gamma(w^+_1d_0+\langle D_1,\beta \rangle +1)}\xi_0^{\bar{\lambda}}\one_{\lambda^{-1}},\\
I_-&=\sum_{\substack{d_0<0, \beta\in \bK_{\eff,\si_-}\\v(\beta)=h^-(d_0,\lambda)}} x^{d_0}q^{\beta-d_0\alpha} \frac{-1}{\fm d_0} \frac{(-1)^{\lfloor w^-_2 d_0 +\langle D_2,\beta \rangle \rfloor}}{\prod_{i\in I_{\sigma^-}}\Gamma(\langle D_i,\beta \rangle +1)}\cdot \frac{\Gamma (-w^-_2d_0 -\langle D_2, \beta \rangle )}{\Gamma (w^-_3d_0+\langle D_3,\beta \rangle+1)\Gamma(w^-_1d_0+\langle D_4,\beta \rangle +1)}\xi_0^{\bar{\lambda}}\one_{\lambda^{-1}}\\
&=\sum_{\substack{\beta\in \bK_{\eff,\si_-}\\\langle \beta,D_4\rangle+w_1^-d_0\in \bZ_{\geq 0}\\d_0<0, v(\beta)=h^-(d_0,\lambda)}} x^{d_0}q^{\beta-d_0\alpha} 
\frac{-1}{\fm d_0} \frac{(-1)^{\lfloor w^-_2 d_0 +\langle D_2,\beta \rangle \rfloor}}{\prod_{i\in I_{\sigma^-}}\Gamma(\langle D_i,\beta \rangle +1)}\cdot \frac{\Gamma (-w^-_2d_0 -\langle D_2, \beta \rangle )}{\Gamma (w^-_3d_0+\langle D_3,\beta \rangle+1)\Gamma(w^-_1d_0+\langle D_4,\beta \rangle +1)}\xi_0^{\bar{\lambda}}\one_{\lambda^{-1}}.
\end{align*}
We have
$$
\langle D_{1}, \alpha \rangle = w^+_1,\quad
\langle D_{2}, \alpha \rangle =w^+_2-w_2^-,\quad
\langle D_{3}, \alpha \rangle=w^+_3-w_3^-,\quad
\langle D_{4}, \alpha \rangle =- w^-_1,
$$
and $\langle D_i,\alpha\rangle =0$ for $i\in I\backslash\{1, 2, 3, 4\}$.
So for $\beta\in \bK_{\eff,\si_-}$,
\begin{align*}
\langle D_{1},\beta\rangle &= \langle D_{1}, \beta-d_0\alpha\rangle + d_0w_1^+,\\
w_2^-d_0 +\langle D_{2}, \beta \rangle &= w^+_2d_0 +\langle D_{2}, \beta-d_0\alpha\rangle, \\
w_3^-d_0 +\langle D_{3}, \beta \rangle &= w_3^+d_0 +\langle D_{3}, \beta-d_0\alpha\rangle, \\
d_0w_1^- +\langle D_{4},\beta\rangle &= \langle D_{4}, \beta-d_0\alpha\rangle.
\end{align*}
Since the conditions $\langle \beta,D_4\rangle+w_1^-d_0\in \bZ_{\geq_0}$ and $\beta \in \bK_{\eff,\si_-}$ implies $(d_0,\beta-d_0\alpha) \in \bK_\eff(\cX,\cL)$,
\begin{align*}
I_-= \sum_{\substack{(d_0,\beta-d_0\alpha)\in \bK_\eff(\cX,\cL)\\d_0<0}} & x^{d_0}q^{\beta-d_0\alpha}  \frac{-1}{\fm d_0} \frac{(-1)^{\lfloor w^+_2 d_0 +\langle D_2,\beta-d_0\alpha \rangle \rfloor}}{\prod_{i\in I_{\sigma^+}}\Gamma(\langle D_i,\beta-d_0\alpha \rangle +1)}
\\&\cdot \frac{\Gamma (-w^+_2d_0 -\langle D_2, \beta-d_0\alpha \rangle )}{\Gamma (w^+_3d_0+\langle D_3,\beta-d_0\alpha \rangle+1)\Gamma(w^+_1d_0+\langle D_1,\beta-d_0\alpha \rangle +1)}\xi_0^{\bar{\lambda}}\one_{\lambda^{-1}}.
\end{align*}
So
$$
I_++I_-=\sum_{\substack{\tbeta\in \bK_{\eff}(\cX,\cL)\\v(\beta)=h^+(d_0,\lambda)}} x^{d_0} q^\beta \frac{-1}{\fm d_0} \frac{(-1)^{\lfloor \langle D_3,\tbeta \rangle \rfloor}}{\prod_{i\in I_{\sigma^+}}\Gamma(\langle D_i,\tbeta \rangle +1)}\cdot \frac{\Gamma ( -\langle D_3, \tbeta \rangle )}{\Gamma (\langle D_2,\tbeta \rangle+1)\Gamma(\langle D_1,\tbeta \rangle +1)}\xi_0^{\bar{\lambda}}\one_{\lambda^{-1}}.
$$

\end{proof}
\begin{remark}
When $\cL$ is an outer brane, the condition $\tbeta=(d_0,\beta)\in \bK(\cX,\cL)$ implies $d_0>0$. One has
$$
\sum_{a=1}^k \langle D_{a+3},\beta \rangle  b_{a+3}=-\sum_{i=1}^3 \langle D_i,\beta \rangle b_i.
$$
Since $\langle D_{a+3},\beta \rangle \geq 0$, the fan $\Si$ is convex, and this is an outer brane, every $1$-cone $b_i$ is on the same side of the plane spanned by $b_2,b_3$. Therefore, $\langle \beta, D_1 \rangle\leq 0$. From $w_1 d_0 +\langle D_1,\beta\rangle\in \bZ_{\geq 0}$ we see that $d_0>0$.
\end{remark}


\subsection{The B-model and the framed mirror curve}

The mirror B-model to the toric Calabi-Yau threefold $\cX$ is another non-compact Calabi-Yau hypersurface $Y\subset \bC^2\times (\bC^*)^2$, constructed as the Hori-Vafa mirror \cite{HV00}. It is equivalent to an affine mirror curve $C_q \subset (\bC^*)^2$. We state the relevant mirror prediction for disk amplitudes from \cite{AV00, AKV02}.

\subsubsection{Toric degeneration}
\label{sec:toric-degeneration}
The main reference of this subsection is \cite[Section 3]{NS06}.

The set
\[
\Theta_0=\bigcap_{I\in \cA} \sum_{i\in I} \bQ_{\geq 0} D_i\subset \bL^\vee_\bQ
\]
is a top dimensional convex cone in $\bL^\vee_\bQ \cong \bQ^{k}$.  The cone $\Theta_0$ together with its faces is a fan in $\bL^\vee_\bR$ denoted by $\Theta$. This fan determines a $k$-dimensional affine toric variety $X_{\Theta}$.

Consider the exact sequence induced from \eqref{eqn:exact-sequence-M} (notice $N_\tor=0$)
\[
0\longrightarrow M'  \stackrel{\phi'^\vee}{\longrightarrow} \tM' \stackrel{\psi'^\vee}{\longrightarrow} \bL^\vee \longrightarrow 0
\]
where $M'=M/\langle \wv_3 \rangle$ and $\tM'=\tM/ \langle \phi^\vee(\wv_3)
\rangle.$ Let $D'^{\cT}_i$ be the image of $D^\cT_i$ when passing to $\tM'$. For any proper subset $I\subset \{1,\dots,r\}$ and a cone $\upsilon \in \Theta$, define
\[
\widetilde \Xi_I=\sum_{i\in I} \bQ_{\geq 0} D'^{\cT}_i , \quad \tTheta_{I,\upsilon}=(\psi'^\vee)^{-1} (\upsilon)\cap \widetilde \Xi_I.
\]
Define a fan
\[
\tTheta=\{ \tTheta_{I,\upsilon}| I \subsetneq\{1,\dots,r\},\upsilon\in \Theta\}\sqcup \{0\}.
\]
This fan determines a toric variety $X_{\tTheta}$. There is a fan
morphism $\rho': \widetilde \Theta\to \Theta$, which induces a flat
family of toric surfaces $\rho: X_{\tTheta}\to X_{\Theta}$.

Let $\Theta_0'\subset \bL^\vee_{\bQ}$ be the cone spanned by $p_1,\ldots, p_{k}$. Let $\bL'^\vee := \bigoplus_{a=1}^{k} \bZ p_a$ and
let $\bL'$ be the dual lattice. Then $\bL'^\vee$ is a sublattice of
$\bL^\vee$ of finite index, and $\bL$ is a sublattice of $\bL'$ of finite index. Let $\Theta_0^\vee$ and $\Theta_0'^\vee$ be the dual cones of $\Theta_0$ and
$\Theta_0'$, respectively.  We have inclusions
$$
\Theta_0'\subset \Theta_0,\quad
\Theta_0^\vee \subset \Theta_0'^\vee.
$$
Since that $\Theta_0^\vee\cap \bL$ is a subset of $\Theta_0'^\vee \cap \bL'$, we have an injective ring homomorphism
$$
\bC[\Theta_0^\vee\cap \bL]\to \bC[\Theta_0'^\vee \cap \bL']=\bC[q_1,\ldots, q_k]
$$
where $q_1,\ldots, q_{k}$ are the variables in Section \ref{sec:I}.  Taking the spectrum, we obtain
a morphism
$$
\bA^{k}= \mathrm{Spec}\left(\bC[q_1,\ldots,q_{k}]\right) \longrightarrow
X_{\Theta}= \mathrm{Spec}\left(\bC[\Theta_0^\vee \cap  \bL]\right).
$$
and a cartesian diagram
\begin{equation} \label{eqn:pullback}
\begin{CD}
\fX @>{\tnu}>> X_{\tTheta} \\
@V{\trho}VV  @V{\rho}VV\\
\bA^{k} @>{\nu}>> X_{\Theta}
\end{CD}
\end{equation}
where $\trho:\fX\to \bA^{k}$ is a flat family of toric surfaces.


We choose a K\"{a}hler class $[\omega(\eta)]\in H^2(X_\Si;\bZ)$ associated to a lattice point
$\eta\in \bL^\vee$; $[\omega(\eta)]$  is the first Chern class of some
ample line bundle over $X_\Si$.  Then it determines a toric graph
$\Gamma\in M'_\bR\cong \bR^2$ up to translation by an element in $M'\cong\bZ^2$ (see Section \ref{sec:toric-graph}). The toric graph gives a {\em polyhedral decomposition} of $M_\bQ$ in the sense of \cite[Section 3]{NS06}.
It is a covering $\cP$ of $M_\bQ$ by strongly convex lattice polyhedra. The asymptotic fan of $\cP$
is defined to be
$$
\Si_{\cP}:=\{ \lim_{a\to 0}a\Xi  \subset M'_\bQ:  \Xi\in \cP\}.
$$
The fan $\Si_{\cP}=\widetilde \Theta \cap \rho'^{-1}(0)$ defines a toric surface $\bS$, which is the same as the toric surface given by the defining polytope $\Delta$. For each $\Pi\in \cP$, let $C(\Pi)\subset M'_\bQ\times\bQ_{\geq 0}$ be the closure of the
cone over $\Xi\times \{1\}$ in $M'_\bQ\times \bQ$. Then
$$
\tSi_\cP:=\{ \sigma\text{ is a face of $C(\Pi)$}: \Pi\in
\cP\}=\widetilde \Theta \cap \rho'^{-1} (\bQ_{\geq0} \eta)
$$
is a fan in $M'_\bQ\times \bQ$ with support $|\tSi_\cP|= M_\bQ'\times
\bQ_{\geq 0}$. The projection $\pi':M'_\bQ\times \bQ\to \bQ$ to the second factor defines a map
from the fan $\tSi_\cP$ to the fan $\{0, \bQ_{\geq 0} \}$. This map of fans determines
a flat toric morphism $\pi: X_{\tSi_{\cP}}\to \bA^1$, where
$X_{\tSi_{\cP}}$ is the toric 3-fold defined by the fan $\tSi_\cP$. Let $t$ be a closed point in $\bA^1$, and let $X_t$ denote the fiber of $\pi$ over $t$. Then $X_t\cong \bS$ for
$t\neq 0$. As shown in \cite{NS06},  when $t=0$, we have a union of irreducible components, where each $\bS_\upsilon$ is the toric surface defined by the polytope $\Delta_\upsilon$ (recall each $3$-cone is a cone over a triangle $\Delta_\upsilon\subset \Delta$ in $N_\bR$)
$$
X_0 =\bigcup_{\upsilon\in \Si(3)}\bS_{\upsilon}.
$$
If $\upsilon' \in \Si(2)$, $\upsilon\in \Si(3)$ and $\upsilon' \subset \upsilon$, $\upsilon'$ corresponds to a torus invariant divisor $\bD_{\upsilon'} \subset \bS_{\upsilon}$. We have the following commutative diagram
\begin{equation} \label{eqn:pullback-2}
  \begin{tikzcd}
       X_0\arrow[hookrightarrow]{r} \arrow{d}
    &  X_{\tSi_\cP} \arrow[hookrightarrow]{r}\arrow{d}{\pi}
    & \fX \arrow{r}{\tnu}\arrow{d}{\trho}
    & X_{\tTheta} \arrow{d}{\rho}
     \\
      \{0\} \arrow[hookrightarrow]{r}
    &  \bA^1 \arrow[hookrightarrow]{r}
    & \bA^k \arrow{r}[swap]{\nu}
    & X_{\Theta}.
  \end{tikzcd}
\end{equation}

The polytope $\mathrm{Hull}(\tb_1,\dots,\tb_r)\subset \tN$ lies
on the hyperplane $\langle \phi^\vee(\wu_3),\bullet \rangle =1$. It
determines a polytope on $\tN'=\{\langle \phi^\vee(\wu_3),\bullet \rangle
=0\}$ up to a translation. The associated line bundle $\mathfrak L$ on
$X_{\widetilde \Theta}$ has sections $u_i,\ i=1,\dots, r$ associated to each
integer point in this polytope. Define
\[
u=\sum_{i=1}^{r} u_i,\quad \widetilde \fC=u^{-1}(0).
\]

The divisor $\widetilde \fC\subset X_{\tTheta}$. Let $\fC:=\tnu^{-1}(\widetilde\fC)\subset \fX$
be the pullback divisor under the morphism $\tnu:\fX\to X_{\tTheta}$.  Then $\fC\to \bA^{k}$ is a flat family of curves over $\bA^{k}$.

For $q\neq 0$, $\fC_q= \trho^{-1}(q)\cap \fC$ can
be identified with the zero locus of
\begin{equation}
  \label{eqn:mirror-curve}
  H_q(x,y)= x^{\fr}y^{-\fs-\fr f} + y^{\fm} +1 + \sum_{a=1}^k s_a (q) x^{m_a} y^{n_a-f m_a},
\end{equation}
where $x^{\fr} y^{-\fs-\fr f}= u_1 u_3^{-1}$, $y^{\fm}=u_2 u_3^{-1}$ while 
$s_a(q)x^{m_a} y^{n_a-f m_a}=u_{3+a} u_3^{-1}$ for $a=1,\ldots,k$.  Here $x,y$ are affine coordinates of the toric surface $\bS$.

For any $\beta\in \bK_\eff(\cX,\cL)$, let $s^\beta=\prod_{a=1}^k s_a^{\langle D_{a+3},\beta\rangle}$. 
If we write $p_a=\sum_{b=1}^k p^b_a D_{b+3}$, we have
\begin{equation}
\label{eqn:q-to-s}
s_a=\prod_{b=1}^k q_b^{p^a_b},\quad s^\beta=q^\beta.
\end{equation}
Denote $s^\tbeta=x^{d_0} s^\beta$ for $\tbeta=\langle d_0,\beta\rangle$.

When $q=0$, we have a union of irreducible components
$$
\fC_0 =\bigcup_{\upsilon \in \Si(3)}\bar C_{0,\upsilon}.
$$
Each irreducible component $\bS_{{\upsilon}}$ of the central fiber $X_0$ is given by the equation $\{u_i=0,\ b_i\notin \upsilon\}$. On $\bS_{{\si_+}}$, the coordinates in the affine chart $u_3\neq 0$ are
$$
x^{\fr} y^{-\fs-\fr f}= u_1 u_3^{-1},\ y^m=u_2 u_3^{-1},$$
while on $\bS_{{\si_-}}$, the coordinates are
$$
u_4u_3^{-1}=(q^{\alpha}x^{-1})^{\fs^-}y^{n_1+s^-f}, u_2u_3^{-1}=y^{\fm}.
$$
Here $b_4=m_1\wv_1 + n_1 \wv_2+\wv_3$, $m_1=-\fs^-$ and $\alpha=[\fl_{\tau}]$.  Define
\begin{align*}
U=&\{ (q_1,\ldots,q_k)\in (\bC^*)^{k}\times \bC^{k-k'}:
\fC_{q} \textup{ is smooth }\\
&\text{and intersects $\partial \bS$
transversally at distinct points} \}.
\end{align*}
Then $U$ is a dense open subset of $\bA^{k}$.

\subsubsection{Mirror curve and the mirror conjecture for disk amplitudes}

When $q\neq 0$, we denote $C_q=\fC_q\setminus (\partial \bS)$. Thus the mirror curve $C_q\subset (\bC^*)^2$ is given by Equation \eqref{eqn:mirror-curve}. On $\bar C_{0,\upsilon}$, when $x=0$, there are $m$ points, called large radius limit (LRL) points. They are given by
$$
x=0,\quad  y^{\fm}=-1.
$$
If $\cL$ is outer, these points are smooth points in $\bar C_{0,\si}\subset \fC_0$; if $\cL$ is inner, they are the nodal points $\bar C_{0,\si_+}\cap \bar C_{0,\si_-}$.

The group $G^*_\si=\{(t_1,t_2)\in (\bC^*)^2|t_1^{\fr} =t_2^{\fs},t_2^{\fm}=1\}$ fits into the short exact sequence
$$
1\to \bmu_{\fr}^* \to G^*_\si \to \bmu_m^* \to 1,
$$
where $G_\sigma^{{ *}}\to \bmu^*_{\fm}$ is given by $(t_1,t_2)\mapsto t_2$. 
Let \[\chi_1=(e^{2\pi\sqrt{-1}\frac{1}{\fr}},1),\quad \chi_2=(e^{2\pi\sqrt{-1}\frac{\fs}{\fr \fm}},e^{2\pi\sqrt{-1}\frac{1}{\fm}}).\]
Then $G^*_\sigma=\{\chi_1^j \chi_2^l|j\in\{0,\dots,\fr -1\},l\in\{0,\dots,\fm-1\}\}$. It pairs with $G_\si$ by
\[
\chi_1(h)=e^{2\pi\sqrt{-1}c_1(h)},\quad \chi_2(h)=e^{2\pi\sqrt{-1}c_2(h)},\quad h\in G;
\]
and acts on the family of compactified mirror curves $\bar{\mathcal C}$ by
\begin{align*}
\chi_1\cdot (x,y,s_a)&=(e^{2\pi\sqrt{-1}\frac{1}{\fr}}x,y,e^{-2\pi\sqrt{-1}c^\sigma_1(b_{a+3})}s_a),\\
\chi_2\cdot (x,y,s_a)&=(e^{2\pi\sqrt{-1}\frac{s+\fr f}{\fr \fm}}x,e^{2\pi\sqrt{-1}\frac{1}{\fm}}y,e^{-2\pi\sqrt{-1}c^\sigma_2(b_{a+3})}s_a).
\end{align*}
Here $c_i^\sigma(b_{a+3})$ is defined as $h_{a+3}=\sum_{i=1}^3c_i^\sigma(b_{a+3}) b_i$. The group $\bmu_{\fm}^*$ acts freely and transitively on the set of LRL points ($x=0$ on $\bar C_{0,\si_+}$).


Given $\bar{\eta}\in \{0,1,\ldots, \fm -1\}$, let $\eta\in G_\tau ^*$ be the element
associated to the character
$$
\chi_\eta: G_\tau \to \bC^*,\quad \chi_\eta(\lambda)= \exp(\frac{2\pi\sqrt{-1}}{\fm}
\bar{\eta}\bar{\lambda}).
$$
Then $\bar{\eta}\mapsto \eta$ is a bijection from $\{0,1,\ldots,\fm-1\}$ to $G_\tau^*$.
Given $\eta\in G_\tau^*$, define $\mathfrak{u}_\eta \in \fC_0$ by
$$
\mathfrak u_\eta
=(0,e^{\pi\sqrt{-1}(-1+ 2\bar{\eta})/\fm}).
$$

For a small $\epsilon$, one can always find small $\epsilon'(\epsilon)<\epsilon$ such that when $\|q\|< \epsilon'(\epsilon)$ the following set
$$
U^{\epsilon,\epsilon'}=\begin{cases}
\{(x,q),|x|<\epsilon,U\}, \quad \text{$\cL$ is outer;}\\
\{(x,q),|x|<\epsilon, \text{$|q^\alpha x^{-1}|<\epsilon$ whenever $\|q\|<\epsilon'$}\} , \quad \text{$\cL$ is inner;}
\end{cases}
$$
is not empty. Let $U^\epsilon=U^{\epsilon,\epsilon'(\epsilon)}\times \{\|q\|< \epsilon'(\epsilon),q\in U\}\subset \fC$. When $\epsilon$ is sufficiently small, $U^{\epsilon,\epsilon'(\epsilon)}$ is a disjoint union of $\fm$ small { contractible} regions when $\cL$ is outer, or is a disjoint union of $m$ annuli when $\cL$ is inner. Let $U^\epsilon_\eta$ be the unique connected component of $U^\epsilon$ containing $\mathfrak u_\eta$. So $\log y$ is well defined on $U^\epsilon_\eta$ up to an integral multiple of $2\pi \sqrt{-1}$, and it could be written as a power series in $x$ when $\cL$ is outer, and a Laurent series in $x$ when $\cL$ is inner.

Define
$$
\phi_\eta:=\frac{1}{\fm}\sum_{\lambda\in G_\tau} \chi_\eta(\lambda^{-1})\one_\lambda.
$$
Then $\{ \phi_\eta: \eta\in G_1^*\}$ is the canonical basis of $H^*_\orb(\cB\mu_{\fm})$, and
$$
\one_\lambda=\sum_{\eta \in \bmu_{\fm}^*} \chi_\eta (\lambda) \phi_\eta.
$$

We prove the following mirror theorem for disk amplitudes.  Note that the ambiguity in $\log y$ does not play any role in the statement.
\begin{theorem} 
\label{conj:akv}
\[x\frac{\partial}{\partial x}\sum_{\eta \in \bmu_m^*}(\log y|_{U^\epsilon_\eta})\phi_\eta  =\left (x\frac{\partial}{\partial x}\right)^2 W^{\cX,(\cL,f)}(q,x)=\left (x\frac{\partial}{\partial x}\right)^2 F_{0,1}^{\cX,(\cL,f)}(\btau_2,X).\]
Here $s$ and $q$ are related by \eqref{eqn:q-to-s}. The A-model flat coordinates $\btau_2,X$ and the B-model coordinates $q,x$ are related by the mirror maps \eqref{eqn:closed-mirror-map} and \eqref{eqn:open-mirror-map}.
\end{theorem}

\begin{remark}
\begin{itemize} 
\item[(i)] Theorem \ref{conj:akv} for smooth toric Calabi-Yau 3-folds was conjectured in \cite{AV00, AKV02} and proved
in  \cite{FL}. 

\item[(ii)]  Theorem \ref{conj:akv} can also be written as
\[\int \sum_{\eta\in \bmu_m^*}(\log y\frac{dx}{x}|_{U^\epsilon_\eta})\phi_\eta`` =" W^{\cX,(\cL,f)}(q,x),\]
where the integral is indefinite and $``="$  means their instanton parts are equal in the following sense. The left side is the sum of a power series with no constant term in $x$ and an extra term in the form of $f(q)\log x+c$. The power series part is equal to the right side. Note that the constant ambiguity in the indefinite integral is irrelevant here. If $\cX$ is a smooth variety, then $\bmu_m^*$ is trivial, and we revert to the original form of the conjecture in \cite{AV00, AKV02}. We will prove this conjecture in the next subsection.
\end{itemize} 

\end{remark}

\subsection{Open mirror theorem for disk amplitudes}
\label{sec:open-mirror}

\begin{lemma}\label{lm:exppoly}
The solution $v$ to the exponential polynomial equation
\begin{equation}\label{eqn:tzero}
\sum_{a=0}^k t_a e^{r_a v}- e^v +1=0,
\end{equation}
around $t_0=\dots={t}_k=0, v=0$ is in the following power series form
\begin{align}
v=\sum_{\substack{l_0,\dots,l_k=0\\(l_0,\dots,l_k)\neq 0}}^\infty \frac{(r_0l_0 + \dots r_k l_k -1)_{(l_0+\dots+l_k-1)}}{l_0!\dots l_k!}
t_0^{l_0}\dots t_k^{l_k}.
\label{eqn:exppoly}
\end{align}
Here we adopt the Pochhammer symbol $$(a)_n=\frac{\Gamma(a+1)}{\Gamma(a-n+1)}=\begin{cases} a(a-1)\cdots(a-n+1),\quad n>0;\\ 1,\quad n=0;\\\frac{1}{(a+1)\dots(a-n)},\quad n<0;\end{cases}$$ where $a\in \bC$ and $n\in \bZ$.
\end{lemma}
\begin{proof}
See Appendix \ref{app:exppoly}.
\end{proof}

Starting from the above observation, we prove Theorem \ref{conj:akv} in this section.
In order to find the expansion of $\log y$ on $U^\epsilon_\eta$, we assume
$$
\log y=\log \xi_0 + \frac{2\pi\sqrt{-1}}{\fm}\bar{\eta}+\frac{v(q,x)}{\fm}
= \frac{\pi\sqrt{-1}}{\fm}(-1+2\bar{\eta}) + \frac{v(q,x)}{\fm}
$$
where $v$ is a power series in $q$ and $x$.  Setting
\begin{align*}
&\xi_{\bar\eta}=e^{\frac{2\pi\sqrt{-1}}{\fm}\bar{\eta}},\ t_0=x^{\fr} (\xi_0\xi_{\bar \eta})^{-s-\fr f},\ r_0=-w_2\fr,\\
&t_a=s_ax^{m_a} (\xi_0\xi_{\bar\eta})^{n_a-fm_a},\ r_a=\frac{n_a-fm_a}{\fm},
\end{align*}
the mirror curve is
$$
H(x,y)=\sum_{a=0}^k t_a e^{r_a v}-e^v+1=0.
$$
Let $\bD_\eff(\cX,\cL)=\{\tbeta=(d_0,\beta) \in \bZ\times \bL|\langle \tbeta,D_i \rangle\in \bZ_{\geq 0},i\neq 2,3\}$. So $\bK_\eff(\cX,\cL)\subset \bD_\eff(\cX,\cL)$.
By Lemma \ref{lm:exppoly},
\begin{align*}
v&=\sum_{\substack{l_0,\dots,l_k=0\\(l_0,\dots, l_k)\neq 0}}^\infty\frac{(r_0l_0+\dots r_k l_k -1)_{(l_0+\dots + l_k -1)}}{l_0!\dots l_k!}\prod_{a=0}^k t_a^{l_a}.\\
&=-\sum_{\tbeta\in \bD_\eff(\cX,\cL)} (\xi_0\xi_{\bar{\eta}})^{-m\langle D_2,\tbeta\rangle}\frac{(-\langle D_2,\tbeta \rangle-1)_{(-\langle D_2,\tbeta \rangle - \langle D_3,\tbeta \rangle -1)}}{\prod_{i\in I_\tau} \langle D_i,\tbeta \rangle!} x^{d_0} \prod_{a=1}^ks_a^{\langle D_{a+3}, \tbeta\rangle }\\
&=-\sum_{\tbeta\in \bD_\eff(\cX,\cL)} (\xi_0)^{-m\langle D_3,\tbeta\rangle}(\xi_{\bar \eta})^{-m\langle D_2,\tbeta \rangle} \frac{(-\langle D_3,\tbeta \rangle-1)_{(-\langle D_2,\tbeta \rangle - \langle D_3,\tbeta \rangle -1)}}{\prod_{i\in I_\tau} \langle D_i,\tbeta \rangle!} x^{d_0}q^\beta.
\end{align*}

Suppose that $\cL$ is an outer brane.
For any $\tbeta=(d_0,\beta) \in \bD_\eff(\cX,\cL)$, we have
$$
d_0(\wv_1+f\wv_2)-v(\beta) +N_\si \in G_\tau\subset G_\si = N/N_\si.
$$
Let $\lambda = d_0(\wv_1+f\wv_2)-v(\beta)+N_\si \in G_\tau$. Then $h(d_0,\lambda)=v(\beta)
\in G_\si$.
If $\cL$ is an inner brane, we replace $\si$ by $\si_+$ in the above discussion and define $\lambda\in G_1$ similarly. Then
$$
\langle D_2,\tbeta \rangle  \in \frac{\bar{\lambda}}{\fm} +\bZ,\quad
\langle D_3,\tbeta \rangle  \in -\frac{\bar{\lambda}}{\fm} + \bZ.
$$
So
$$
\xi_{\bar\eta}^{-\langle D_2,\tbeta\rangle} =
\exp(-\frac{2\pi\sqrt{-1}}{\fm}\bar{\eta}\bar{\lambda})
= \chi_\eta(\lambda^{-1})
$$

It follows that
\begin{align*}
x\frac{\partial}{\partial x}\sum_{\eta \in \bmu_m^*}(\log y)_{U_\eta^\epsilon}\phi_\eta
=&-\sum_{\substack{\tbeta\in \bK_\eff(\cX,\cL)\\h(d_0,\lambda)=v(\beta)}} (\xi_0)^{-\fm\langle D_3,\tbeta\rangle} \frac{d_0(-\langle D_3,\tbeta \rangle-1)_{(-\langle D_2,\tbeta \rangle - \langle D_3,\tbeta \rangle -1)}}{\fm\prod_{i\in I_\tau} \langle D_i,\tbeta \rangle!} \sum_{\eta \in \bmu^*_m}\chi_\eta (\lambda^{-1})\phi_\eta  x^{d_0}q^\beta\\
=&-\sum_{\substack{\tbeta\in \bK_\eff(\cX,\cL)\\h(d_0,\lambda)=v(\beta)}} (\xi_0)^{-\fm\lfloor \langle D_3,\tbeta\rangle \rfloor} 
\frac{d_0(-\langle D_3,\tbeta \rangle-1)_{(-\langle D_2,\tbeta \rangle - \langle D_3,\tbeta \rangle -1)}}{\fm\prod_{i\in I_\tau} \langle D_i,\tbeta \rangle!} 
\sum_{\eta \in \bmu^*_{\fm}}(\xi_0)^{\bar \lambda} \chi_\eta(\lambda^{-1})\phi_\eta   x^{d_0}q^\beta\\
=&\sum_{\substack{\tbeta\in \bK_\eff(\cX,\cL)\\h(d_0,\lambda)=v(\beta)}} -(-1)^{\lfloor \langle D_3,\tbeta\rangle \rfloor} \frac{d_0\Gamma(-\langle D_3,\tbeta \rangle)}{\fm\Gamma(\langle D_2,\tbeta \rangle+1)
\prod_{i\in I_\tau} \Gamma(\langle D_i,\tbeta \rangle+1)}  x^{d_0}q^\beta \xi_0^{\bar{\lambda}}\one_{\lambda^{-1}}.
\end{align*}
On the other hand,
\begin{align*}
&\big (x\frac{\partial}{\partial x}\big)^2 W^{\cX,(\cL,f)}(x,q)\\
=&\sum_{\substack{\tbeta\in \bK_\eff(\cX,\cL)\\h(d_0,\lambda)=v(\beta)}}- (-1)^{\lfloor \langle D_3,\tbeta\rangle \rfloor} 
\frac{d_0\Gamma(-\langle D_3,\tbeta \rangle)}{m\Gamma(\langle D_2,\tbeta \rangle+1)\prod_{i\in I_\tau} \Gamma(\langle D_i,\tbeta \rangle+1)}  x^{d_0}q^\beta \xi_0^{\bar{\lambda}}\one_{\lambda^{-1}}.
\end{align*}
Thus Theorem \ref{conj:akv} follows.

\appendix

\section{Proof of Lemma  \ref{lm:exppoly}}
\label{app:exppoly}

In this appendix, we obtain a power series solution  to the following exponential polynomial where $r_a\in \bR$
\begin{equation} \label{eqn:tone}
1-e^v+ \sum_{a=1}^k  t_a e^{r_a v}=0
\end{equation}
around $t_1=\dots=t_k=0$ by oscillatory integral and inverse Laplace transform. Note that the notation here is slightly different from
that in Lemma \ref{lm:exppoly}: the sum in the above Equation \eqref{eqn:tone} starts from $a=1$,
whereas the sum in Equation \eqref{eqn:tzero} in Lemma \ref{lm:exppoly} starts from $a=0$.

We also consider the following equation where $f,r_a\in \bZ_{>0}$
\begin{equation}
  \label{eqn:curve}
L(X,Y)=1+XY^{-f}+Y+\sum_{a=1}^k s_a Y^{r_a}=0.
\end{equation}
Let $X=e^{-x}$ and $Y=e^{-y}$. This equation identifies with Equation \eqref{eqn:tone} after setting $X=0$ and a change of variables $v=\sqrt{-1}\pi -y, t_a=(-1)^{r_a}s_a$.

\begin{lemma}
For Equation \eqref{eqn:tone}, one can expand $v$ as a power series in $t_a$, where each coefficient is a rational function of $r_a$. For Equation \eqref{eqn:curve}, the variable $Y$ can be expanded as a power series of $(-1)^{r_a}s_a$ and $(-1)^fX$ around $Y=-1$ with each coefficient rational in $r_a$ and $f$. One can also expand $Y$ as a power series of $s_a$ and $(-X)^{\frac{1}{f}}$ around $Y=0$ with each coefficient rational in $r_a$ and $f$.
\label{lem:rational}
\end{lemma}
\begin{proof}
  This is done by elementary recursive calculation. We illustrate the expansion at $Y=0$ for Equation \eqref{eqn:curve}. The equation can be written as
  \[
  Y^f+Y^{f+1}+\sum_{a=1}^k s_a Y^{r_a+f}=((-X)^{\frac{1}{f}})^f.
  \]
  Implicit function theorem (applying to $Y(Y+1+\sum_{a=1}^k s_aY^{r_a})^{\frac{1}{f}}=(-X)^{\frac{1}{f}}$) says $Y$ is analytic in $(-X)^{\frac{1}{f}}$ and $s_a$ around $(X,Y,s)=(0,0,0)$, and recursive calculation shows each coefficient is a rational function of lower degree coefficients.
\end{proof}

We consider an affine curve
\begin{equation*}
C:=\{(X,Y)\in (\bC^*)^2\mid L(X,Y)=0\}
\end{equation*}
and its partial compactification $\bar C \subset \bC^2$ with two points $(X,Y)=(0,0)$
and $(X,Y)=(0,-1+O(s))$ added. Let $e^{-x_0}$ be the branch point of the map $(X,Y)\mapsto X$ such that 
$e^{-x_0}=-\frac{1}{4}+O(s)$. Let $\gamma_{s}$
be the oriented Lefschetz thimble which passes through the ramification point
$(e^{-x_0}, e^{-y_0})=(-\frac{1}{4}+O(s), -\frac{1}{2}+O(s))$ and  goes
from $(X,Y)=(0,-1+O(s))$ to $(X,Y)=(0,0)$.
So the coordinate on $\gamma_s$ is $z$ such that $x-x_0=z^2$. We choose the sign of $z$ such that
$(X,Y)=(0,0)$ is at $z=+\infty$.
\begin{lemma}
\label{lem:oscillatory}
$$
\int_{\gamma_{s}}  e^{-ux}  ydx=\sum_{l_1,\dots,l_k \geq 0} e^{ \sqrt{-1}\pi(-(f+1)u+\sum_{a=1}^k r_al_a)}
\frac{\Gamma(u)\Gamma(fu+\sum_{a=1}^k r_a l_a)}{\Gamma((f+1)u+\sum_{a=1}^k ( r_a-1)l_a+1)}
\frac{\prod_{a=1}^k s_a^{l_a}}{l_1!\dots l_k!}.
$$
\end{lemma}
\begin{proof}
Consider a Landau-Ginzburg model $W_{s}:(\bC^*)^3\to \bC$, where
\begin{align*}
W_{s}=X_1 X_2^{-f}X_3+X_2 X_3 +X_3+\sum_{a=1}^k s_a X_2^{r_a} X_3 -u \log X_1.
\end{align*}
Define $\vt_1 =X_1 X_2^{-f} X_3,  \vt_2 = X_2 X_3, \hat t_1 =X_1 X_2^{-f},   \hat t_2= X_2.$

Let $X_3\in \Gamma_3$ be a cycle that counter-clockwise encircles the positive real axis, starting and ending on the positive real infinity. We require the argument of each $X_3\in \Gamma_3$ takes every value in $(0,2\pi)$ once. Define the relative connected cycle $\Gamma_{s}$ to be
\begin{align*}
\Gamma_{s}=&\{(X_1,X_2,X_3)\in (\bC^*)^3 \mid \vt_1>0,\vt_2>0, X_3\in \Gamma_3,
\text{when $X_3<0$ and $s=0$, $X_2 \in \bR^-$} \}.
\end{align*}
When $|s|<\epsilon$ for small $\epsilon$, the superpotential $\mathrm{Re}(W)\to \infty$ in the non-compact direction of $\Gamma_{s}$.
On $\Gamma_{s}$ the logarithm is taken in the following way: when $X_3<0$ and $s=0$,
$$
\arg(X_1)= -(f+1)\pi, \quad \arg(X_2)= -\pi,\quad \arg(X_3)=\pi.
$$
Since the cycle $\Gamma_{s}$ is simply-connected and deforms continuously with respect to $s$,
this choice  is fixed. Evaluate the following oscillatory integral of $W_{s}$
\begin{align*}
I(u)=&\int_{\Gamma_s} e^{-W_{s}} \frac{dX_1}{X_1} \frac{dX_2}{X_2} \frac{dX_3}{X_3}\\
=&\int_{\Gamma_s} \exp(-\sum_{a=1}^k s_a \vt_2^{r_a} X_3^{1-r_a}-\vt_1-\vt_2-X_3 +u \log \vt_1 
+ f u\log\vt_2
- (f+1)u \log X_3)  \frac{d\vt_1}{\vt_1}\frac{d\vt_2}{\vt_2}\frac{dX_3}{X_3}\\
=&\int_{\Gamma_{s} } e^{-\sum_{a=1}^k s_a \vt_2^{r_a} X_3^{1-r_a}-\vt_1-\vt_2-X_3} \vt_1^u 
\vt_2^{fu}
e^{-(f+1)u (\log X_3-\sqrt{-1}\pi)} e^{-(f+1)\sqrt{-1}\pi u}  \frac{d\vt_1}{\vt_1}\frac{d\vt_2}{\vt_2}\frac{dX_3}{X_3}\\
=&-e^{-(f+1)\sqrt{-1}\pi u} \sum_{l_1,\dots, l_k\geq 0}
(-1)^{\sum_{a=1}^k (r_a-1)l_a} \prod_{a=1}^k\frac{(-s_a)^{l_a}}{l_a!}
 \Big ( \int_{\vt_1>0}  e^{-\vt_1} \vt_1^{u-1} d\vt_1 \Big) \\
&\cdot \Big( \int_{\vt_2>0} e^{-\vt_2}   \vt_2^{\sum_{a=0}^k r_a l_a+fu-1}
d\vt_2 \Big ) \cdot \Big (\int_{X_3 \in \Gamma_3} e^{-X_3}
e^{(\log X_3-\sqrt{-1}\pi)( \sum_{a=1}^k (1-r_a)l_a-(f+1)u-1)} dX_3 \Big )\\
=&2\pi\sqrt{-1}e^{-(f+1)\sqrt{-1}\pi u} \sum_{l_1,\dots,l_k\geq 0}
(-1)^ {\sum_{a=1}^k r_a l_a } \prod_{a=1}^k\frac{s_a^{l_a}}{l_a!}
\frac{\Gamma(u) \Gamma(fu+\sum_{a=1}^k r_a l_a )}{\Gamma((f+1)u+\sum_{a=1}^k (r_a -1) l_a +1)}.
\end{align*}
Here we use the Hankel's formula
$$
\frac{\sqrt{-1}}{2\pi}\int_{\Gamma_3} e^{-z(\log(t)-\sqrt{-1}\pi)}e^{-t} dt=\frac{1}{\Gamma(z)}.
$$

By Hori-Iqbal-Vafa \cite{HIV00}, this oscillatory integral could be  reduced to a Laplace transform on the curve $C$. Introduce two variables $v^+,v^-\in \bC$, and the extended cycle $$\tGamma_s=\Gamma_s \times \{v^+=-\overline{v^-}\}.$$ Define $H=\frac{W_s}{X_3}$.  Define the holomorphic volume form  $$\Omega=\frac{dX_1}{ X_1}\frac{d X_2}{X_2}\frac{dv^-}{v^-}=dx dy \frac{dv^-}{v^-}.$$

We reduce the oscillatory integral to the curve $C$ as follows. Let $\tGammar=\{(\hat t_1,\hat t_2)|\arg \hat t_1=\arg \hat t_2\}\times \{v^+=-\overline{v^-}\}.$
\begin{align*}
I(u)=&\frac{1}{2\sqrt{-1}\pi} \int_{\tGamma_s} e^{-X_3 (H-v^+ v^-)}e^{-ux} \frac{d  X_1}{ X_1}\frac{d  X_2}{X_2} dX_3 dv^+ dv^-\\
=& \frac{1}{2\sqrt{-1}\pi} \int_{\tGammar} \delta(H-v^+ v^-) e^{-ux} \frac{d  X_1}{ X_1}\frac{d  X_2}{X_2} dv^+ dv^-\\
=&-  \int_{\tGammar \cap \{H-v^+ v^-=0\}} e^{-ux} \frac{d  X_1}{ X_1}\frac{d  X_2}{X_2}\frac{dv^-}{v^-}.
\end{align*}
This integration is further reduced to the curve $C=\{H(e^{-x},e^{-y})=0\}$ as follows.
$$
I(u)=- \int_{\tGammar \cap \{H-v^+ v^-=0\}} e^{-ux}dx dy  \frac{dv^-}{v^-}\\
= 2\sqrt{-1}\pi\int_{\gamma_s}  e^{-ux} y dx.
$$
Notice that we use the fact $d(e^{-ux} ydx\frac{dv^-}{v^-})=-e^{-ux} \Omega$ near $\tGammar \cap \{H-v^+ v^-=0\}$.
\end{proof}

The function $\frac{dy}{dx}$ is a meromorphic function on the partially compactified curve $\bar C$ with the only pole at $x_0$. Its expansion at $(X,Y)=(0,-1+O(s))$ is a power series in $X$, while its expansion at $(X,Y)=(0,0)$ is a series in $X^{\frac{1}{f}}$. Denote $g^\pm(x)=\frac{dy}{dx}\vert_{z=\pm\sqrt{x-x_0}}$. Then $g^-$ is a power series in $X$ and $s_a$, while $g^+$ is a power series in $X^{\frac{1}{f}}$ and $s_a$. Since they are expansions of $\frac{dy}{dx}=\frac{XdY}{YdX}$ regarding to the curve equation \eqref{eqn:curve}, as series of $(-1)^f X$, $(-1)^{r_a}s_a$ and $(-X)^{\frac{1}{f}}$, $s_a$ respectively, their coefficients are rational in $r_a$ and $f$ by Lemma \ref{lem:rational}.

By Lemma \ref{lem:oscillatory}, the ``classical Laplace transform" is
\begin{align*}
\mathfrak G(u)=&\int_{x-x_0\in \bR^+} e^{-u(x-x_0)}(g^+(x)-g^-(x))d(x-x_0)\\
=&\int_{x-x_0\in \bR^+}e^{-u(x-x_0)} \left(\frac{dy}{dx}\right) d(x-x_0)=u e^{ux_0}\int_{\gamma_s} e^{-ux} ydx\\
=&\sum_{l_1,\dots,l_k\geq 0} u e^{ux_0}e^{\sqrt{-1}\pi(-(f+1)u+\sum_{a=1}^k r_al_a)}
\frac{\Gamma(u)\Gamma(fu+\sum_{a=1}^kr_a l_a)}{\Gamma((f+1)u  + \sum_{a=1}^k(r_a-1)l_a+1)}\frac{\prod_{a=1}^k s_a^{l_a}}{l_1!\dots l_k!}.
\end{align*}
By the inverse Laplace transform formula,
\begin{align*}
(g^+-g^-)=\int_{u=-\infty \sqrt{-1}+ T}^{u=+\infty\sqrt{-1}+T}\mathfrak G(u)e^{u(x-x_0)}du,
\end{align*}
where $T$ is large enough such that all poles of $\mathfrak G(u)$ is on the left of the integration contour.  Here the inverse Laplace transform takes residues around poles of $\Gamma(u)$ and $\Gamma(fu+\sum_{a=1}^kr_al_a)$. Taking the residues around all poles (other than the possible pole at $u=0$) of $\Gamma(fu+\sum_{a=1}^k r_a l_a)$ gives a series of $(-X)^{\frac{1}{f}}$ with coefficients rational in $r_a$ and $f$, denoted by $h^+$
\begin{align*}
h^+&=\sum_{l>0,l_1,\dots, l_k\geq 0}h^+_{l,l_1,\dots,l_k} \prod_{a=1}^k s_a^{l_a}((-X)^{\frac{1}{f}})^l
\\&=\sum_{l>0,l_1,\dots, l_k\geq 0}(-l/f)e^{\sqrt{-1}\pi(-\frac{f+1}{f}l+\sum_{a=1}^k r_a l_a)}(X^{\frac{1}{f}})^l \frac{\Gamma(-\frac{l}{f})\Res_{u=-\frac{l}{f}}\Gamma(fu+\sum_{a=1}^k r_a l_a)}{\Gamma(-\frac{l(f+1)}{f}+\sum_{a=1}^k(r_a-1)l_a+1)}\frac{\prod_{a=1}^k s_a^{l_a}}{l_1!\dots l_k!}\\
&=\sum_{l>0,l_1,\dots, l_k\geq 0}(-l/f)(-X)^{\frac{l}{f}}\frac{\Gamma(-\frac{l}{f})}{\Gamma(-\frac{l(f+1)}{f}+\sum_{a=1}^k(r_a-1)l_a+1)f}\frac{\prod_{a=1}^k s_a^{l_a}}{l!l_1!\dots l_k!};
\end{align*}
while taking residues around the poles of $\Gamma(u)$ (other than the possible pole at $u=0$) we get a power series in $X$
\begin{align*}
h^-=&\sum_{l>0, l_1,\dots,l_k\geq 0} (-l) X^l e^{\sqrt{-1}\pi(-(f+1)l+\sum_{a=1}^k r_al_a)}
\frac{\mathrm{Res}_{u=-l}(\Gamma(u))\Gamma(-fl+\sum_{a=1}^kr_a l_a)}{\Gamma(-(f+1)l +\sum_{a=1}^k(r_a-1)l_a+1)}\frac{\prod_{a=1}^k s_a^{l_a}}{l_1!\dots l_k!}\\
=&\sum_{l>0, l_1,\dots,l_k\geq 0} (-l) ((-1)^fX)^l
\frac{\Gamma(-fl+\sum_{a=1}^kr_a l_a)}{\Gamma(-(f+1)l +\sum_{a=1}^k(r_a-1)l_a+1)l!}\frac{\prod_{a=1}^k ((-1)^{r_a}s_a)^{l_a}}{l_1!\dots l_k!}.
\end{align*}
So $g^+-g^-=h^++h^-+\text{const.}$, where the constant difference (in $X$) arises since we don't consider the residue around $u=0$. For any degree $l\geq 1$, choose $f>l$ such that the term $((-X)^\frac{1}{f})^l\prod_{a=1}^k s_a^{l_a}$ in $g^+$ is not a monomial in $X$, and thus can only come from $h^+$. Since the coefficient of the term is rational in $f$, it has to be equal to the corresponding term $h^+_{l,l_1,\dots,l_k}$ for all $f>0$. Therefore $g^+=h^++\mathrm{const.}$ and $-g^-=h^-$. Here note that $g^-$ is the expansion of $\frac{dy}{dx}$, and since $y$ is analytic in $X$ at $(X,Y)=(0,-1+O(s))$, $-g^-$ has no degree $0$ term and does not differ from $h^-$ by a degree $0$ term in $X$.

Suppose the expansion of $y$ at $(X,Y)=(0,-1+O(s))$ is $y=A_0 +\sum_{l>0} A_l  X^l$, then the expansion of $\frac{dy}{dx}$ at this point is
$$
g^-=\frac{dy}{dx}=-\sum_{l>0} l A_l X^l.
$$
Therefore for $l\geq 1$,
\begin{equation}
  \label{eqn:A_l}
A_l=-\sum_{l_1,\dots,l_k\geq 0} e^{\sqrt{-1}\pi(- fl+\sum_{a=1}^k r_al_a)}
\frac{\Gamma(-fl+\sum_{a=1}^kr_a l_a)}{\Gamma(-(f+1)l + \sum_{a=1}^k(r_a-1)l_a+1)}\frac{\prod_{a=1}^k s_a^{l_a}}{l_1!\dots l_k!}.
\end{equation}

We prove Lemma \ref{lm:exppoly} by induction. The statement is true for $k=0$ trivially. Assume it is true for $k=\fm-1\ (\fm\geq 1)$. For $k=\fm$ we first assume $r_{\fm}\in \bZ_{<0}$ and all other $r_a$ ($a=1,\dots,\fm-1$) are positive integers. By the induction assumption we know the expansion is given as in Lemma \ref{lm:exppoly} for terms of degree $0$ in $t_{\fm}$. Let $f=-r_{\fm}$. 
After a change of variables $v=\sqrt{-1}\pi-y$, $(-1)^{r_a}s_a=t_a$ for $a=1,\dots,{\fm-1}$ and $X=(-1)^ft_{\fm}$ we obtain  Equation \eqref{eqn:curve}. Then from Equation \eqref{eqn:A_l} we know the expansion of $y$ for positive degree terms in $X$, thus conclude that for positive degree terms in $t_{\fm}$ the lemma also holds. Then for all degrees the lemma holds
\[
v=\sum_{\substack{l_1,\dots,l_k=1\\(l_1,\dots,l_k)\neq 0}}^\infty \frac{(r_1l_1 + \dots r_k l_k -1)_{(l_1+\dots+l_k-1)}}{l_1!\dots l_k!}
t_1^{l_1}\dots t_k^{l_k}.
\]
Each coefficient is a rational function of $r_1,\dots,r_{\fm}$. The above equation holds for $r_1,\dots, r_{\fm-1}\in \bZ_{>0}$ 
and $r_{\fm}\in \bZ_{<0}$, so it is true for all $r_a\in \bR$.

\end{document}